\documentclass[11pt]{amsart}
\usepackage{amssymb,latexsym}
\usepackage{amsthm}

\numberwithin{equation}{section}

\def\hyp{\hskip.5pt\vbox
{\hbox{\vrule width2.5ptheight0.5ptdepth0pt}\vskip2pt}\hskip.5pt}
\def\hs{\hskip.7pt}
\def\hh{\hskip.4pt}
\def\nh{\hskip-.7pt}
\font\smallbf=ptmbo at 10pt
\font\medbf=ptmbo at 11.5pt
\font\bigrm=cmr10 scaled\magstep 1

\def\txm{T\hskip-3pt_x\w\hn M}
\def\lx{S\nh}

\def\w{^{\phantom i}}
\def\nnh{\hskip-1.5pt}
\def\hn{\hskip-.4pt}

\def\bz{y}
\def\bbR{\mathrm{I\!R}}
\def\rto{\bbR\hskip-.5pt^2}

\newcommand{\bbC}{{\mathchoice {\setbox0=\hbox{$\displaystyle\mathrm{C}$}
\hbox{\hbox to0pt{\kern0.4\wd0\vrule height0.9\ht0\hss}\box0}} 
{\setbox0=\hbox{$\textstyle\mathrm{C}$}\hbox{\hbox 
to0pt{\kern0.4\wd0\vrule height0.9\ht0\hss}\box0}} 
{\setbox0=\hbox{$\scriptstyle\mathrm{C}$}\hbox{\hbox 
to0pt{\kern0.4\wd0\vrule height0.9\ht0\hss}\box0}} 
{\setbox0=\hbox{$\scriptscriptstyle\mathrm{C}$}\hbox{\hbox 
to0pt{\kern0.4\wd0\vrule height0.9\ht0\hss}\box0}}}} 
\def\hg{\hat{g\hskip2pt}\hskip-1.3pt}
\newcommand{\bbRP}{\bbR\mathrm{P}}
\newcommand{\bbCP}{\bbC\mathrm{P}}
\newcommand{\bbCH}{\bbC\mathrm{H}}
\def\w{^{\phantom i}}

\theoremstyle{plain}
\newtheorem{proposition}{Proposition}[section]

\newtheorem{lemma}[proposition]{Lemma}
\newtheorem{theorem}[proposition]{Theorem}

\theoremstyle{remark}
\newtheorem{remark}[proposition]{Remark}
\newtheorem{example}[proposition]{Example}

\title[Harmonic curvature in dimension four]{Harmonic curvature in dimension 
four}
\author{Andrzej Derdzinski}
\address{Department of Mathematics, 
The Ohio State University, Columbus, OH 43210, USA}
\email{andrzej@math.ohio-state.edu}
\thanks{Research supported in part by a FAPESP\hn-\hh OSU 2015 Regular
Research Award (FAPESP grant: 2015/50265-6)}
\subjclass[2020]{Primary 53B20; Secondary 53C25}
\keywords{harmonic curvature, Co\-daz\-zi tensor}
\begin{document}
\begin{abstract}
We provide a step towards classifying 
Riemannian four-manifolds in which the curvature tensor has zero 
divergence, or -- equivalently -- the Ricci tensor Ric satisfies the 
Codazzi equation. Every known compact manifold of this type belongs to one
of five otherwise-familiar classes of examples. The main result 
consists in showing that, if such a manifold (not necessarily compact or even 
complete) lies outside of the five classes -- a non-vacuous assumption --
then, at all points of a dense open subset, Ric has four distinct eigenvalues,
while suitable local coordinates simultaneously diagonalize Ric, the metric
and, in a natural sense, also the curvature tensor. Furthermore, 
in a local orthonormal frame formed by Ricci eigenvectors, the connection form 
(or, curvature tensor) has just twelve (or, respectively, six)
possibly-nonzero components, which together satisfy a specific system, not
depending on the point, of homogeneous polynomial equations. A part of the
classification problem is thus reduced to a question in real algebraic
geometry.
\end{abstract}
\maketitle

\setcounter{section}{0}
\section*{Introduction}\label{in}
One says that a Riemannian manifold has {\it harmonic curvature\/} if its 
curvature tensor $\,R\,$ satisfies the lo\-cal-co\-or\-di\-nate relation 
$\,R_{ijk}\w{}^p{}\nnh_{,\hs p}\w=0$, that is,
\begin{equation}\label{dvr}
\mathrm{div}\,R\,=\,0\hh.
\end{equation}
See \cite[Sect.~16.33]{besse}. Let us now consider the condition
\begin{equation}\label{kpc}
\begin{array}{l}
(K\nnh+c)^3\nh+3(K\nnh+c)\Delta K\nh-6\hskip.7pt|dK|^2\nh=r^3\nh,\hskip2.3pt
\mathrm{where}\\
r,c\in\bbR\,\hs\mathrm{\ and\ }\hs\,K\nnh+c\ne0\hs\,\mathrm{\ at\ 
every\ point\ of\ }\,Q\hh,
\end{array}
\end{equation}
imposed on the Gauss\-i\-an curvature $\,K\,$ of a Riemannian surface 
$\,(Q,h)$, with
$\,\Delta=h^{ij}\nabla\hskip-2pt_{\!i}\w\nnh\nabla\hskip-2.3pt_{\!j}\w$ denoting
the $\,h$-La\-plac\-i\-an and $\,|\hskip3.3pt|\,$ the $\,h$-norm.

The following four-man\-i\-folds all have harmonic curvature 
(Remark~\ref{hrmcv}).
\begin{equation}\label{kne}
\begin{array}{rl}
\mathrm{a)}&\mathrm{Ein\-stein\ manifolds\ of\ dimension\ four.}\\
\mathrm{b)}&\mathrm{Con\-for\-mal\-ly\ flat\ 4}\hyp\mathrm{man\-i\-folds\ 
of\ constant\ scalar\ curvature.}\\
\mathrm{c)}&\mathrm{Riemannian\nh\ products\nh\ of\nh\ a\nh\ 
one}\hyp\hh\mathrm{di\-men\-sion\-al\nh\ manifold\nh\ and\nh\ a}\\
&\mathrm{con\-for\-mal\-ly\ flat\ 
3}\hh\hyp\mathrm{man\-i\-fold\ with\ constant\ scalar\ curvature.}\\
\mathrm{d)}&\mathrm{Products\ of\ surfaces\ having\ constant\ 
Gauss\-i\-an\ curvatures.}\\
\mathrm{e)}&\mathrm{Warped\ products\ 
}\,(Q\nh\times\nh S,(h\times\hskip-.2pth\hn^c)/(K\nnh+c)^2)\mathrm{,\ where\ 
}\,(Q,h)\\
&\mathrm{and\ }\,(S,h\hn^c)\,\mathrm{\ are\ Riemannian\ surfaces,\ 
}\,(S,h\hn^c)\,\mathrm{\ is\ of\ constant}\\
&\mathrm{Gauss\-i\-an\ curvature\ }\,c\hh\mathrm{,\ and\ }\,(Q,h)\,\mathrm{\ 
satisfies\ (\ref{kpc})}.
\end{array}
\end{equation}
In (\ref{kne}.e) we treat $\,K\,$ as a function on $\,Q\nh\times\nh S$, 
constant along the $\,S\,$ factor. Thus, $\,2(K\nnh+c)\,$ equals the scalar 
curvature of the product metric $\,h\times\hskip-.2pth\hn^c\nh$. 

All known examples of {\it compact\/} four-man\-i\-folds $\,(M\nh,g)\,$ with 
$\,\mathrm{div}\,R=0$ belong to the five (non-disjoint)
lo\-cal-i\-som\-e\-try types (\ref{kne}) in the sense that
\begin{equation}\label{lit}
\mathrm{each\ }\,x\in M\,\mathrm{\ has\ a\ neighborhood\ isometric\ to\ 
one\ of\ (\ref{kne}).}
\end{equation}
However, $\,\mathrm{div}\,R=0\,$ in some {\it complete\/} Riemannian
four-man\-i\-folds not containing open sub\-man\-i\-folds of types 
(\ref{kne}). See \cite{derdzinski-piccione} and Remark~\ref{nonvc}.

This paper is a step towards classifying Riemannian 
four-man\-i\-folds with harmonic curvature that lie outside of the 
five classes (\ref{kne}). The next section states in full detail our two main 
results, here summarized only briefly.

According to the first of them, Theorem~\ref{ricev}, at generic points of 
such a manifold $\,(M\nh,g)$, its Ric\-ci tensor $\,\mathrm{Ric}\,$ has four
distinct eigen\-values, and suitable local coordinates simultaneously
di\-ag\-o\-nal\-ize $\,g,\,\mathrm{Ric}\,$ and $\,R$. Note that the local
or\-tho\-nor\-mal frame $\,e\nh_i\w,\,i=1,\dots,4$, obtained by normalizing
the coordinate vector fields, then gives rise to an orthogonal web of
co\-di\-men\-sion-one foliations or, equivalently, satisfies, for all distinct 
$\,i,j\in\{1,2,3,4\}$, the Lie-brack\-et relations
\begin{equation}\label{lbr}
[\hs e\nh_i\w,e\nh_j\w]\,=\hs F\hskip-4pt_{ji}\w e\nh_i\w\,
-\hs F\hskip-3.2pt_{ij}\w e\nh_j\w\hskip8pt\mathrm{(no\ summation),\ with\ 
some\ functions\ }F\hskip-4pt_{ji}\w\hh.
\end{equation}
As shown by Tod \cite{tod}, co\-or\-di\-nate-di\-ag\-o\-nal\-iz\-abil\-i\-ty
of a metric, in dimension four, generically imposes restrictions on the third
derivative of the Weyl tensor. Simplicity of the Ric\-ci eigen\-values
implies in turn that a har\-mon\-ic-cur\-va\-ture manifold satisfying the
above assumptions cannot be a 
nontrivial warped product with any fibre dimension $\,p\,$ greater than one 
\cite[Corollary 1.3]{derdzinski-piccione}, although the case $\,p=1\,$ does
occur \cite{derdzinski-piccione}.

The second result, Theorem~\ref{algvr}, states that the twelve functions 
$\,F\hskip-4pt_{ji}\w$ in (\ref{lbr}) and the six 
sec\-tion\-al-cur\-va\-ture functions 
$\,R_{ijij}\w\nh=R\hh(e\nh_i\w,e\nh_j\w,e\nh_i\w,e\nh_j\w)\,$
together form, at every generic point $\,x$, a solution of a specific system,
not depending on $\,x$, of homogeneous polynomial equations. 
Thus, when these $\,F\hskip-4pt_{ji}\w$ and $\,R_{ijij}\w$ are 
treated as the components of a mapping $\,\varPhi\,$ from a neighborhood of a 
generic point into $\,\bbR\nnh^{1\nh8}\nh$, the values of $\,\varPhi\,$ 
lie in an explicitly defined real algebraic variety 
$\,\mathcal{V}\nh\subseteq\nnh\bbR\nnh^{1\nh8}\nh$. 
Consequently, Theorem~\ref{algvr} relates the classification of 
four-di\-men\-sion\-al Riemannian manifolds 
with $\,\mathrm{div}\,R=0$, different from the types (\ref{kne}), to a problem 
in real algebraic geometry.

The text is organized as follows. Sections~\ref{ds} and~\ref{sd} provide
detailed statements of the main results and an outline of the proof of parts
(a) --  (b) in Theorem~\ref{ricev}. Preliminary and expository material, 
presented in Sections~\ref{pr} through~\ref{tf}, and~\ref{cw}, is followed by 
Section~\ref{ap}, describing the three {\it a priori\/} possible cases that
arise under the hypotheses Theorem~\ref{ricev}. After the exclusion of two of
these cases (Sections~\ref{co} --~\ref{cc}), the conclusions about the third
case lead, in Section~\ref{nz}, to proofs of Theorem~\ref{ricev}(a)\hs-\hs(b)
and Theorem~\ref{algvr}. The four final sections are devoted to proving part
(c) of Theorem~\ref{ricev}.

\section{Detailed statements of the main results}\label{ds}
As shown by DeTurck and Goldschmidt \cite{deturck-goldschmidt}, in suitable 
local coordinates,
\begin{equation}\label{ana}
\mathrm{every\ metric\ with\ }\,\mathrm{div}\,R=0\,\mathrm{\ is\ 
real}\hyp\mathrm{analytic.}
\end{equation}
For a fixed oriented Riemannian four-man\-i\-fold $\,(M\nh,g)\,$ with 
$\,\,\mathrm{div}\,R=0$, let us denote by 
$\,\text{\smallbf r}\in\{1,2,3,4\}\,$ and $\,\text{\smallbf w}\in\{1,2,3\}\,$ 
the maximum number of distinct eigen\-values of the Ric\-ci tensor 
$\,\,\mathrm{Ric}\,\,$ acting on the tangent bundle $\,T\nh M$ and, 
respectively, of the self-dual Weyl tensor $\,W^+\nnh$ acting on the bundle of 
self-dual bi\-vec\-tors (see Section~\ref{aw}). Due to (\ref{ana}), both 
maxima $\,\text{\smallbf r}\,$ and $\,\text{\smallbf w}$ are simultaneously 
attained at all points of a dense open subset of $\,M\nh$.

Proofs of Lemma~\ref{other}, (a), (b) in  
Theorem~\ref{ricev} along with Theorem~\ref{algvr}, and Theorem~\ref{ricev}(c)
are given, respectively, in Sections~\ref{hc},~\ref{nz} and~\ref{dz}
--~\ref{lc}.
\begin{lemma}\label{other}
For any oriented Riemannian four-man\-i\-fold\/ $\,(M\nh,g)\,$ having\/
$\,\,\mathrm{div}\,R=0$, the following two conditions are 
equivalent\hh{\rm:}
\begin{enumerate}
  \def\theenumi{{\rm\roman{enumi}}}
\item $(M\nh,g)\,$ belongs to one of the lo\-cal-i\-som\-e\-try types\/
{\rm(\ref{kne})}, as in\/ {\rm(\ref{lit})},
\item $\hh g\,$ is locally reducible, or\/ $\,\text{\smallbf r}\in\{1,2\}$, 
or\/ $\,\text{\smallbf w}\in\{1,2\}$.
\end{enumerate}
\end{lemma}
\begin{theorem}\label{ricev}
Suppose that\/ $\,\hs\mathrm{div}\,R=0\,$ for the curvature tensor\/ $\,R\,$ 
of an oriented Riemannian four-man\-i\-fold\/ $\,(M\nh,g)\,$ which does not 
satisfy\/ {\rm(\ref{lit})}. The following conclusions then hold on some dense
open set\/ $\,\,U\nh\subseteq\nh M\nh$.
\begin{enumerate}
  \def\theenumi{{\rm\alph{enumi}}}
\item Locally in\/ $\,\,U\,$ there exist functions\/ 
$\,F\hskip-4pt_{ji}\w$ and real-an\-a\-lyt\-ic or\-tho\-nor\-mal vector 
fields\/ $\,e\nh_i\w$ di\-ag\-o\-nal\-izing\/ $\,\,\mathrm{Ric}\hh$, with the 
Lie brackets given by\/ 
$\,[\hs e\hs\nh_i\w,e\nh_j\w]=F\hskip-4pt_{ji}\w e\nh_i\w
-F\hskip-3.2pt_{ij}\w e\nh_j\w\,$ whenever\/ $\,i,j\in\{1,2,3,4\}\,$ are 
distinct.
\item $e\nh_i\w$ also di\-ag\-o\-nal\-ize\/ $\,R$, in the sense of 
Section\/~{\rm\ref{aw}}.
\item $\text{\smallbf r}=\nh4\,$ and\/ $\,\mathrm{Ric}\hs\,$ has four distinct 
eigen\-values at every point of\/ $\,\,U\nnh$.
\end{enumerate}
\end{theorem}
About the lo\-cal-co\-or\-di\-nate aspect of (a), mentioned in the
Introduction, see Remark \ref{mltpl},

Theorem~\ref{ricev} is non-vac\-u\-ous (Remark~\ref{nonvc}) and, according to 
Lemma~\ref{other}, the assumptions made about $\,(M\nh,g)\,$ amount to its
being locally irreducible, four-di\-men\-sion\-al, oriented and
having $\,\mathrm{div}\,R=0\,$ along with
\begin{equation}\label{tis}
(\text{\smallbf r},\text{\smallbf w})\nnh\in\nnh\{3,4\}\nh\times\nh\{3\}\,
\mathrm{\ or,\hskip3.2ptequivalently,\ 
}\,\text{\smallbf r}\notin\{1,2\}\,\mathrm{\ and\ 
}\,\text{\smallbf w}\notin\{1,2\}\hh.
\end{equation}
The next theorem and the remainder of the paper, except Section~\ref{hc}, use
the convention that the indices $\,i,j,k,l\,$ always range over
$\,\{1,\dots,4\}$, repeated indices are {\it not summed over\/} and, unless
stated otherwise,
\begin{equation}\label{ind}
\begin{array}{l}
\mathrm{if\hs\ some\hs\ of\hs\ }\hs\,i,j,k,l\hs\,\mathrm{\hs\ appear\hs\ 
in\hs\ an\hs\ equality,\hskip3ptthey\hs\ are\hs\ assumed}\\
\mathrm{to\nh\ be\nh\ mutually\nh\ distinct\nh\ and\nh\ preceded\nh\ by\nh\ 
a\nh\ universal\nh\ quantifier.}
\end{array}
\end{equation}
Thus, the presence of $\,i,j,k\,$ will tacitly imply the preamble
\begin{equation}\label{pre}
\mathrm{for\ all\ }\,i,j,k\in\{1,\dots,4\}\,\mathrm{\ with\ 
}\,i\ne j\ne k\ne i\hh.
\end{equation}
Rather than using the sec\-tion\-al-cur\-va\-ture functions 
$\,R_{ijij}\w\nh=R\hh(e\nh_i\w,e\nh_j\w,e\nh_i\w,e\nh_j\w)$ (see the
Introduction), it is more convenient to phrase our second main result in terms
of the analogous components 
$\,\sigma\nnh_{ij}\w=W\nh(e\nh_i\w,e\nh_j\w,e\nh_i\w,e\nh_j\w)\,$ of the Weyl 
tensor $\,W\nnh$, along with the scalar curvature $\,\,\mathrm{s}$, and the 
eigen\-value functions $\,\lambda_i\w=b(e\nh_i\w,e\nh_i\w)\,$ of the 
trace\-less Ric\-ci tensor $\,b=\,\mathrm{Ric}\,-\,\mathrm{s}g/4$. The latter 
are easily expressed through the former, cf.\ equality (\ref{wij}) below, and 
vice versa:
\begin{equation}\label{jij}
R_{ijij}\w\,=\,
\sigma\nnh_{ij}\w\,+\,\frac12\,(\lambda_i\w+\lambda_j\w)\,
+\,\frac{\mathrm{s}}{12}\hskip10pt\mathrm{if}\hskip6pti,j\in\{1,2,3,4\}
\hskip6pt\mathrm{and}\hskip6pti\ne j\hh.
\end{equation}
See the line following formula (\ref{rkl}) in Section~\ref{cw}.
\begin{theorem}\label{algvr}
For\/ $\,(M\nh,g),\hs U\nnh,e\nh_i\w$ and\/ $\,F\hskip-4pt_{ji}\w$ as in 
Theorem\/~{\rm\ref{ricev}}, $\,F\hskip-4pt_{ji}\w$ and the functions\/ 
$\,\sigma\nnh_{ij}\w,\lambda_i\w$ defined above 
satisfy the polynomial equations
\[
\begin{array}{l}
\lambda_i\w\nh+\nh\lambda_j\w\nh+\nh\lambda_k\w\nh
+\nh\lambda_l\w\nh
=\sigma\nnh_{ij}\w\nh-\nh\sigma\hskip-2.7pt_{ji}\w\nh
=\sigma\nnh_{ij}\w\nh-\nh\sigma\nnh_{kl}\w\nh
=\sigma\nnh_{ij}\w\nh+\nh\sigma\nnh_{ik}\w\nh+\nh\sigma\nnh_{il}\w\nh=0\hh,\\
{}[(\lambda_k\w-\lambda_l\w)\hh\sigma\nnh_{kl}\w\nh
+(\lambda_l\w-\lambda_i\w)\hh\sigma\nnh_{li}\w\nh
+(\lambda_i\w-\lambda_k\w)\hh\sigma\nnh_{ik}\w]
(F\hskip-3.2pt_{kl}\w F\hskip-3.2pt_{li}\w\nh
+F\hskip-3.2pt_{lk}\w F\hskip-3.2pt_{ki}\w\nh
-F\hskip-3.2pt_{ki}\w F\hskip-3.2pt_{li}\w)\\
\hskip15.3pt{}=\,[(\lambda_k\w-\lambda_l\w)\hh\sigma\nnh_{kl}\w\nh
+(\lambda_l\w-\lambda_j\w)\hh\sigma\nnh_{lj}\w\nh
+(\lambda_j\w-\lambda_k\w)\hh\sigma\hskip-2.7pt_{jk}\w]\hh
(F\hskip-3.2pt_{kl}\w F\hskip-3.2pt_{lj}\w\nh
+F\hskip-3.2pt_{lk}\w F\hskip-3.2pt_{kj}\w\nh
-F\hskip-3.2pt_{kj}\w F\hskip-3.2pt_{lj}\w)\hh,
\end{array}
\]
with the conventions\/ {\rm(\ref{ind}) -- (\ref{pre})}. Choosing the
frame\/ $\,e\nh_i\w$ is to be positive oriented, we may rewrite the second
displayed equation as
\begin{equation}\label{fsi}
H\hskip-2.5pt_{ji}\w Z\nnh_j\w\,
=\,-\hh H\nnh_{ij}\w Z_i\w\quad\mathrm{whenever\
}\,i,j\in\{1,2,3,4\}\,\mathrm{\ and\ }\,i\ne j\hh,
\end{equation}
where\/ $\,H\nnh_{ij}\w$ and\/ $\,Z_l\w$ are uniquely characterized by\/ 
$\,H\nnh_{ij}\w\nh=F\hskip-3.2pt_{kl}\w F\hskip-3.2pt_{lj}\w\nh
+F\hskip-3.2pt_{lk}\w F\hskip-3.2pt_{kj}\w\nh
-F\hskip-3.2pt_{kj}\w F\hskip-3.2pt_{lj}\w$ when\/ 
$\,\{i,j,k,l\}=\{1,2,3,4\}\,$ and\/
$\,Z_l\w\nh=(\lambda_i\w-\lambda_j\w)\hh\sigma\nnh_{ij}\w\nh
+(\lambda_j\w-\lambda_k\w)\hh\sigma\hskip-2.7pt_{jk}\w\nh
+(\lambda_k\w-\lambda_i\w)\hh\sigma\nnh_{ki}\w$ if\/ $\,(i,j,k,l)\,$ is an
even permutation of\/ $\,(1,2,3,4)$. The twelve functions\/ $\,H\nnh_{ij}\w$
are subject to a further system of polynomial equations, namely
\begin{equation}\label{fsp}
\mathrm{rank}\left[\begin{matrix}H\nnh_{12}\w&H\nnh_{13}\w&H\nnh_{14}\w&0&0&0&1
\cr
H\nnh_{21}\w&0&0&H\nnh_{23}\w&H\nnh_{24}\w&0&1\cr
0&H\nnh_{31}\w&0&H\nnh_{32}\w&0&H\nnh_{34}\w&1\cr
0&0&H\nnh_{41}\w&0&H\nnh_{42}\w&H\nnh_{43}\w&1\end{matrix}\right]
\le\,\hs3\hh.
\end{equation}
\end{theorem}

\section{Preliminaries}\label{pr}
Manifolds, mappings and tensor fields are by definition 
$\,C^\infty\nnh$-differentiable. Unless stated otherwise, a manifold is 
assumed connected. Our conventions about the exterior derivative of a 
$\,1$-form $\,\zeta\,$ and the curvature tensor $\,R\,$ of a Riemannian metric 
$\,g\,$ are such that, for tangent vector fields $\,u,v,w$,
\begin{equation}\label{rvw}
\begin{array}{l}
(d\zeta)(u,v)\,=\,d_u\w[\zeta(v)]\,-\,d_v\w[\zeta(u)]\,-\,\zeta([u,v])\hh,\\
R\hs(v,w)u\,=\,\nabla\!_{w}\w\nabla\!_{v}\w u\,-\,
\nabla\!_{v}\w\nabla\!_{w}\w u\,+\,\nabla\!_{[v,w]}\w u\hh.
\end{array}
\end{equation}
The Ric\-ci tensor $\,\,\mathrm{Ric}\,\,$ and scalar curvature 
$\,\,\mathrm{s}\,\,$ of $\,g\,$ give rise to the Schouten tensor 
$\,\,\mathrm{Sch}=\mathrm{Ric}\hs-\hs[2(n-1)]^{-1}\mathrm{s}\hs g\,\,$ and the 
Weyl con\-for\-mal tensor $\,W=R-(n-2)^{-1}\hs g\wedge\mathrm{Sch}$, in
dimensions $\,n\ge3$, where $\,\wedge\,$ is a natural bilinear pairing of 
symmetric $\,2$-tensors, valued in covariant $\,4$-tensors 
\cite[formula (1.116)]{besse}. In coordinates, with $\,R_{ij}\w$ denoting the 
components of $\,\,\mathrm{Ric}$,
\begin{equation}\label{wij}
\begin{array}{l}
W\hskip-3.2pt_{ijkl}\w\,=\,R_{ijkl}\w\,-\,(n-2)^{-1}(g_{ik}\w R_{jl}\w\,+\,
g_{jl}\w R_{ik}\w\,-\,g_{jk}\w R_{il}\w\,-\,g_{il}\w R_{jk}\w)\\
\phantom{W_{ikl}\w\,=}+\,\,(n-1)^{-1}(n-2)^{-1}\hs\mathrm{s}\hs
(g_{ik}\w g_{jl}\w\,-\,g_{jk}\w g_{il}\w)\hh.
\end{array}
\end{equation}
A {\it Co\-daz\-zi tensor\/} \cite[Sect,~16.5]{besse} on a Riemannian manifold
is a twice-co\-var\-i\-ant symmetric tensor field $\,b\,$ with $\,db=0\,$
(in coordinates: $\,b_{ki,\hs j}\w=b_{kj,\hh i}\w$)

As usual, by the {\it warped product\/} of the Riemannian manifolds 
$\,(Q,h)\,$ (the {\it base}) and $\,(\varSigma,\gamma)\,$ (the {\it fibre}) 
with the {\it warping function\/} $\,\phi:Q\to(0,\infty)\,$ we mean the 
Riemannian manifold\
\begin{equation}\label{wrp}
(M\nh,g)\,=\,(Q\hn\times\nnh\varSigma,\,h+\nh\phi^2\gamma)\hh,
\end{equation}
where $\,h,\gamma,\phi\,$ also denote the pull\-backs of the original 
$\,h,\gamma,\phi\,$ to $\,Q\hn\times\nnh\varSigma$.
\begin{remark}\label{cnfpr}Since (\ref{wrp}) amounts to $\,(M\nh,g)\,
=\,(Q\hn\times\nnh\varSigma,\,\phi^2\hh[\hs\phi^{-\nh2}\nh g+h\hs])$, a warped 
product is the same as a Riemannian manifold con\-for\-mal to a Riemannian 
product in such a way that the con\-for\-mal-fac\-tor function is constant 
along one of the constituent manifolds.
\end{remark}

\section{Algebraic Weyl tensors and di\-ag\-o\-nal\-iz\-abil\-i\-ty}\label{aw}
Given a Euclidean vector space $\,\mathcal{T}\hs$ of any dimension $\,n$, by 
an {\it algebraic Weyl tensor\/} in $\,\mathcal{T}\hs$ we mean a quadrilinear 
mapping 
$\,A:\mathcal{T}\times\mathcal{T}\times\mathcal{T}\times\mathcal{T}\nh\to\bbR$ 
having the usual symmetries of the Weyl con\-for\-mal tensor 
(skew-sym\-me\-try in the first and last pairs of arguments, the first 
Bian\-chi identity, and vanishing of the Ric\-ci contraction). If the 
Ric\-ci-con\-trac\-tion requirement is relaxed, one calls $\,A\,$ an {\it 
algebraic curvature tensor}. Any such $\,A\,$ also forms a linear 
en\-do\-mor\-phism $\,A:\mathcal{T}^{\wedge2}\nnh\to\mathcal{T}^{\wedge2}$ of 
the space $\,\mathcal{T}^{\wedge2}$ of bi\-vec\-tors, with 
$\,A(v\wedge w)=A(v,w,\,\cdot\,,\,\cdot\,)\in[\mathcal{T}^*]^{\wedge2}\nnh
=\mathcal{T}^{\wedge2}\nnh$, where $\,[\mathcal{T}^*]^{\wedge2}\nnh
=\mathcal{T}^{\wedge2}$ due to the identification 
$\,\mathcal{T}^*\nnh=\mathcal{T}\,$ provided by the inner product 
\cite[Sect.~1.108]{besse}.

Following \cite[Sect.~16.18]{besse}, one says that an orthogonal basis 
$\,e\nh_1,\dots,e_n$ of $\,\mathcal{T}$ {\it di\-ag\-o\-nal\-izes\/} an 
algebraic curvature tensor $\,A\,$ if 
$\,A(e\nh_i\w,e\nh_j\w,e\nh_k\w,e\nh_l\w)=0$ whenever 
$\,\{i,j\}\ne\{k,l\}\,$ or, equivalently, if all the nonzero exterior products 
$\,e\nh_i\w\nh\wedge e\nh_j\w$ are eigenvectors of 
$\,A:\mathcal{T}^{\wedge2}\nnh\to\mathcal{T}^{\wedge2}\nh$. Riemannian 
manifolds $\,(M\nh,g)$ whose curvature tensor is di\-ag\-o\-nal\-iz\-ed, at 
every point $\,x$, by some orthogonal basis of $\,\txm\nh$, were first studied 
by Maillot \cite{maillot} and referred to by him as having {\it pure curvature 
operator}.

An algebraic curvature tensor $\,A\,$ and a symmetric bilinear form 
$\,b\,$ on $\,\mathcal{T}$ will be called {\it simultaneously 
di\-ag\-o\-nal\-iz\-able} if some orthogonal basis of $\,\mathcal{T}$ 
di\-ag\-o\-nal\-izes both $\,b\,$ (in the usual sense) and $\,A$.

For a proof of the next fact, see \cite[Theorem 1]{derdzinski-shen}.
\begin{lemma}\label{rvovt}Let\/ $\,b\,$ be a 
Co\-daz\-zi tensor on a Riemannian manifold\/ $\,(M\nh,g)$, and let\/ 
$\,v_i\w\in \txm\,$ be eigenvectors of\, $\,b\hs(x)\hs$ at a point\/ 
$\,x\in M\hs$ correspoding to some eigen\-values $\,\lambda_i\w$, $\,i=1,2,3$. 
The curvature tensor $\,R=R\hs(x)\hs$ of\/ $\,(M\nh,g)$ at\/ $\,x\hs$ then 
satisfies the relation\/ $\,R\hs(v_1\w,v_2\w)v_3\w=\hs0\hs\,$ whenever 
$\,\lambda_1\w\ne\lambda_3\w\ne\lambda_2\w$.
\end{lemma}
From Lemma \ref{rvovt} it is immediate that a Co\-daz\-zi tensor $\,b\,$ on a 
Riemannian $\,n$-man\-i\-fold $\,(M\nh,g)\,$ and the curvature tensor $\,R\,$ 
are simultaneously di\-ag\-o\-nal\-iz\-able at each point $\,x\in M\,$ where 
$\,b\,$ has $\,n\,$ distinct eigen\-values. On the other hand, the condition 
of simultaneous di\-ag\-o\-nal\-iz\-abil\-i\-ty of $\,b\,$ and $\,R$ at any 
given point $\,x\,$ clearly implies the same condition for $\,b\,$ and the 
Ric\-ci tensor $\,\,\mathrm{Ric}\,\,$ (that is, the bundle en\-do\-mor\-phisms 
of $\,TM\,$ corresponding to $\,b\,$ and $\,\,\mathrm{Ric}\,\,$ then commute 
at $\,x$) and, consequently, also for $\,b\,$ and the Weyl tensor 
$\,W=W\nh(x)$, cf.\ (\ref{wij}). However, in dimension $4$, even a weaker 
assumption on $\,b\,$ yields the same conclusion 
\cite[proof of Lemma 2]{derdzinski-88}:
\begin{lemma}\label{simdi}Let\/ $\,b\,$ be a Co\-daz\-zi tensor on an oriented 
Riemannian four-man\-i\-fold\/ $\,(M\nh,g)$. Then\/ $\,b\,$ and the Weyl 
tensor\/ $\,W$ are simultaneously di\-ag\-o\-nal\-iz\-able at every point at 
which\/ $\,b\,$ is not a multiple of\/ $\,g$.   
\end{lemma}
In an oriented Euclidean $\,4$-space $\,\mathcal{T}\nnh$, the Hodge star 
$\,*:\mathcal{T}^{\wedge2}\nnh\to\mathcal{T}^{\wedge2}$ acting on 
bi\-vec\-tors may be characterized by 
$\,*(e\nh_1\w\nh\wedge e\nh_2\w)=e\nh_3\w\nh\wedge e\nh_4\w$ whenever 
$\,e\nh_1\w,\dots,e\nh_4\w$ is a pos\-i\-tive-ori\-ent\-ed or\-tho\-nor\-mal 
basis of $\,\mathcal{T}\nnh$. This makes $\,*\,$ an involution, with 
$\,\mathcal{T}^{\wedge2}\nh=\mathcal{L}^+\nnh\oplus\hs\mathcal{L}^-\nh$, where 
$\,\mathcal{L}^+\nnh=\mathrm{Ker}\hskip1.2pt(*\hs-\hs\mathrm{Id})\,$ and 
$\,\mathcal{L}^-\nnh=\mathrm{Ker}\hskip1.2pt(*\hs+\hs\mathrm{Id})$, the 
spaces of {\it self-dual\/} and {\it anti-self-dual\/} bi\-vec\-tors, are the 
eigen\-spaces of $\,*\hs$. Any $\,e\nh_1\w,\dots,e\nh_4\w$ as above clearly 
lead to a basis of $\,\mathcal{L}^\pm$ formed by the bi\-vec\-tors
\begin{equation}\label{bas}
e\nh_1\w\nh\wedge e\nh_2\w\pm e\nh_3\w\nh\wedge e\nh_4\w\hh,\hskip7pt 
e\nh_1\w\nh\wedge e\nh_3\w\pm e\nh_4\w\nh\wedge e\nh_2\w\hh,\hskip7pt 
e\nh_1\w\nh\wedge e\nh_4\w\pm e\nh_2\w\nh\wedge e\nh_3\w\,\in\,\mathcal{L}^\pm.
\end{equation}
Any algebraic Weyl tensor
$\,A:\mathcal{T}^{\wedge2}\nnh\to\mathcal{T}^{\wedge2}$ leaves the subspaces
$\,\mathcal{L}^\pm$ invariant, since \cite[Theorem 1.3]{singer-thorpe} it
commutes with $\,*\hs$, which results in
\begin{equation}\label{apm}
\mathrm{the\ restrictions\
}\,A\nh^\pm:\mathcal{L}^\pm\nh\to\mathcal{L}^\pm\nh\mathrm{,\ both\ self}\hyp
\mathrm{ad\-joint\ and\ trace\-less.}
\end{equation}
If $\,\hs\mathcal{T}\nh=\txm\nh$, where $\,(M\nh,g)\,$ is an oriented 
Riemannian four-man\-i\-fold and $\,x\in M\nh$, we denote $\,\mathcal{L}^\pm$ 
by $\,\Lambda_x^{\!\pm\!}M\nh$, which leads to the vector sub\-bundles 
$\,\Lambda^{\!\pm\!}M$ of 
$\,[TM]^{\wedge2}\nh=\Lambda^{\!+\!}M\oplus\Lambda^{\!-\!}M\nh$. The 
restrictions $\,W^\pm\nnh:\Lambda^{\!\pm\!}M\to\Lambda^{\!\pm\!}M\,$ of the 
Weyl tensor $\,W\,$ of $\,(M\nh,g)\,$ satisfy in view of (\ref{apm}) the
conditions
\begin{equation}\label{trw}
\mathrm{tr}\hskip2ptW\hh^+\,=\,\,\mathrm{tr}\hskip2ptW\hh^-\,=\,0\hh.
\end{equation}
\begin{lemma}\label{sgesg}Given an or\-tho\-nor\-mal basis of a Euclidean\/ 
$\,4$-space\/ $\,\mathcal{T}\hs$ which di\-ag\-o\-nal\-izes an algebraic Weyl 
tensor\/ $\,A$, let\/ $\,\sigma\nnh_{ij}\w=\sigma\hskip-2.7pt_{ji}\w$ be its 
eigen\-values, with\/ 
$\,\sigma\nnh_{ij}\w=A(e\nh_i\w,e\nh_j\w,e\nh_i\w,e\nh_j\w)$, where\/
$\,i,j\in\{1,2,3,4\}\,$ and\/ $\,i\ne j$. If\/ 
$\,\{i,j,k,l\}=\{1,2,3,4\}$, then, for either fixed orientation,
$\,A\nh^+$ has the same eigen\-values\/
$\,\sigma\nnh_{ij}\w,\sigma\nnh_{ik}\w,\sigma\nnh_{il}\w$
as\/ $\,A\nh^-\nh$, while $\,\sigma\nnh_{ij}\w=\sigma\nnh_{kl}\w$ and\/ 
$\,\sigma\nnh_{ij}\w+\sigma\nnh_{ik}\w+\sigma\nnh_{il}\w=0$.
\end{lemma}
\begin{proof}With standard normalizations, 
$\,A(e\nh_i\w\nh\wedge e\nh_j\w)
=\sigma\nnh_{ij}\w\hs e\nh_i\w\nh\wedge e\nh_j\w$ 
(no summation). If $\,\{i,j,k,l\}=\{1,2,3,4\}\,$ and we use the orientation 
determined by $\,e\nh_i\w,e\nh_j\w,e\nh_k\w,e\nh_l\w$ then, by 
(\ref{bas}), $\,e\nh_i\w\nh\wedge e\nh_j\w+e\nh_k\w\nh\wedge e\nh_l\w$ lies in 
$\,\mathcal{L}^+\nh$, and so does its $\,A\nh$-im\-age 
$\,\sigma\nnh_{ij}\w\hs e\nh_i\w\nh\wedge e\nh_j\w
+\sigma\nnh_{kl}\w\hs e\nh_k\w\nh\wedge e\nh_l\w$, equal to its own
$\,*$-im\-age $\,\sigma\nnh_{kl}\w\hs e\nh_i\w\nh\wedge e\nh_j\w
+\sigma\nnh_{ij}\w\hs e\nh_k\w\nh\wedge e\nh_l\w$. 
This last equality gives $\,\sigma\nnh_{ij}\w=\sigma\nnh_{kl}\w$, while 
$\,\sigma\nnh_{ij}\w+\sigma\nnh_{ik}\w+\sigma\nnh_{il}\w$ vanishes, being the
trace of $\,A(e\nh_i\w,\,\cdot\,,e\nh_i\w,\,\cdot\,)$, that is,
$\,a(e\nh_i\w,e\nh_i\w)\,$ for the Ric\-ci contraction $\,a\,$ of $\,A$.
Now
$\,e\nh_i\w\nh\wedge e\nh_j\w\pm e\nh_k\w\nh\wedge e\nh_l\w\in\mathcal{L}^\pm$ 
is an eigen\-vector of both $\,A\nh^\pm$ with the eigen\-value 
$\,\sigma\nnh_{ij}\w=\sigma\nnh_{kl}\w$, due to (\ref{bas}) with 
$\,A(e\nh_i\w\nh\wedge e\nh_j\w)=\sigma\nnh_{ij}\w e\nh_i\w\nh\wedge e\nh_j\w$.
\end{proof}
\begin{remark}\label{eottf}The mapping that assigns to a 
pos\-i\-tive-ori\-ent\-ed or\-tho\-nor\-mal basis $\,e\nh_1\w,\dots,e\nh_4\w$ 
of $\,\mathcal{T}\,$ the pair (\ref{bas}) of pos\-i\-tive-ori\-ent\-ed 
orthogonal bases of $\,\mathcal{L}^+$ and $\,\mathcal{L}^-\nh$, with all 
vectors of length $\,\sqrt{2\,}$, is a two-fold covering, equi\-var\-i\-ant 
relative to the two-fold covering homo\-mor\-phism 
$\,\mathrm{SO}(4)\to\mathrm{SO}(3)\nnh\times\mathrm{SO}(3)$, while 
$\,\mathcal{L}^\pm$ are both canonically oriented 
\cite[Sect.~16.58]{besse}.
\end{remark}
\begin{remark}\label{specw}Let $\,(M\nh,g)\,$ be a K\"ah\-ler manifold of real 
dimension four, with the canonical orientation. Its self-dual Weyl tensor 
$\,W^+$ acting on the bundle $\,\Lambda^{\!+\!}M$ of self-dual bi\-vec\-tors 
then has fewer than three distinct eigen\-val\-ues at every point
\cite[formula~(16.64)]{besse}.
\end{remark}
\begin{remark}\label{prsfm}In an oriented Riemannian four-man\-i\-fold 
$\,(M\nh,g)\,$ with $\,g\,$ con\-for\-mal to a product $\,\hg\,$ of surface 
metrics, the conclusion of Remark~\ref{specw} applies to both $\,W^+$ and 
$\,W^-\nnh$. (As a special case, $\,(M\nh,g)\,$ might be here a warped product 
of two orientable Riemannian surfaces, cf.\ Remark~\ref{cnfpr}.) This follows 
since $\,\hg\,$ then is a K\"ah\-ler metric for two local complex structures 
compatible with the two mutually opposite orientations.
\end{remark}

\section{Harmonic curvature}\label{hc}
For any Riemannian manifold, the second Bian\-chi identity implies the 
equality 
$\,\,\mathrm{div}\,R\,=\,-\,d\,\mathrm{Ric}$, where the Ric\-ci tensor 
$\,\,\mathrm{Ric}\,\,$ treated as a $1$-form valued in $1$-forms. (Its 
coordinate version reads 
$\,R_{ijk}\w{}^p{}\nnh_{,\hs p}\w=R_{ki,\hs j}\w-R_{kj,\hh i}\w$.) This leads 
to equivalence between (\ref{dvr}) and the Co\-daz\-zi equation 
\begin{equation}\label{drz}
d\,\mathrm{Ric}\,=\,0\hh,\hskip16pt\mathrm{that\ is,}\hskip8ptR_{ki,\hs j}\w
=R_{kj,\hh i}\w\hh.
\end{equation}
Consequently, $\,\mathrm{div}\,R=0\,$ if and only if $\,\mathrm{Ric}\,$ is a 
Co\-daz\-zi tensor (Section \ref{pr}). Contracting the identity 
$\,R_{ijk}\w{}^p{}\nnh_{,\hs p}\w=R_{ki,\hs j}\w-R_{kj,\hh i}\w$ 
with $\,g^{ik}$ one gets
\begin{equation}\label{brt}
2\,\mathrm{div}\,\mathrm{Ric}\,=\,d\hs\mathrm{s}\hh,\hskip7pt\mathrm{that\ 
is,}\hskip6pt2g^{jk}R_{ij,k}\w\hs=\,\mathrm{s}_{,\hh i}\w
\end{equation}
for any Riemannian metric $\,g$. Therefore, from (\ref{drz}), 
\begin{equation}\label{cst}
\mathrm{whenever\ }\,\,\mathrm{div}\,R\,=\,0\,\,\mathrm{\ the\ scalar\ 
curvature\ }\,\,\mathrm{s}\,\,\mathrm{\ must\ be\ constant.}
\end{equation}
Since $\,2(n-1)(n-2)\,\mathrm{div}\,W=
-(n-3)\hs d\hs[2(n-1)\,\mathrm{Ric}\hs-\hs\mathrm{s}\hs g\hh]\,$ for 
Riemannian metrics $\,g\,$ in dimensions $\,n\ge4$, cf.\ 
\cite[Sect.~16.3]{besse}, (\ref{drz}) and (\ref{cst}) imply that
\begin{equation}\label{dvw}
\mathrm{div}\,R\,=\,0\hskip16pt\mathrm{if\ and\ only\ if}\hskip14pt
\mathrm{div}\,W\,=\,0\hskip6pt\mathrm{and}\hskip6ptd\hs\mathrm{s}\,
=\,0\hh.
\end{equation}
As an obvious consequence of (\ref{dvr}), a Riemannian product has harmonic 
curvature if and only if so do both factor manifolds. For a surface metric, 
harmonic curvature means constant Gauss\-i\-an curvature, which follows from 
(\ref{cst}). In dimension $\,n=3\,$ one always has $\,W\nnh=0$, and
con\-for\-mal flat\-ness amounts to the condition 
$\,d\hs[2(n-1)\,\mathrm{Ric}\hs-\hs\mathrm{s}\hs g\hh]=0$, cf.\ 
\cite[Sect.~1.170]{besse}, so that, by (\ref{drz}), having harmonic 
curvature is the same as being con\-for\-mal\-ly flat and of constant scalar 
curvature. 
\begin{remark}\label{cffpr}A Riemannian product is con\-for\-mal\-ly flat if 
and only if both factors have constant sectional curvatures $\,K,K'$ and 
$\,K'\nnh=\nh-K\,$ or one factor is of dimension $\,1$. See
\cite[Section 5]{yau-73}, \cite[Sect.~1.167]{besse}.
\end{remark}
\begin{remark}\label{hrmcv}Condition (\ref{dvr}) for the manifolds 
(\ref{kne}.a) -- (\ref{kne}.b), or (\ref{kne}.c) -- (\ref{kne}.d), or 
(\ref{kne}.e), follows from (\ref{drz}) and (\ref{dvw}), or from 
the paragraph following (\ref{dvw}) or, respectively, from 
\cite[Lemma 3(ii)]{derdzinski-88}.
\end{remark}
\begin{lemma}\label{warpd}Let\/ $\hs\mathrm{div}\,R=0$ for a warped product\/ 
\hbox{$(M\nh,g)\nh=\nh(I\nh\times\nnh N,\,dt^2\nnh+\nh Fh)$} of an open
interval\/ 
$\,I\subset\bbR\,$ carrying the standard metric\/ $\,dt^2\nnh$ and a 
Riemannian\/ $3\hh$-man\-i\-fold\/ $\,(N,h)$, where\/ $\,F:I\to(0,\infty)\,$  
and\/ $\,dt^2\nh,F,h\,$ are identified with their pull\-backs to\/ $\,M\nh$. 
Then\/ $\,(M\nh,g)\,$ is of type\/ {\rm(\ref{kne}.c)} or\/ 
{\rm(\ref{kne}.b)}, depending on whether\/ $\,F\hs$ is constant or not.
\end{lemma}
\begin{proof}If $\,F\,$ is nonconstant, $\,h$, being an Ein\-stein metric 
\cite[Lemma 4]{derdzinski-80}, has constant sectional curvature. We can now 
use Remarks~\ref{cnfpr} and~\ref{cffpr}.
\end{proof}
For an oriented Riemannian four-man\-i\-fold $\,(M\nh,g)\,$ having 
$\,\,\mathrm{div}\,R=0$, we denote by $\,\text{\smallbf w}^-\in\{1,2,3\}\,$ 
the maximum number of distinct eigen\-values of the anti-self-dual Weyl tensor 
$\,W^-\nnh$ acting on the bundle $\,\Lambda^{\!-\!}M\,$ of anti-self-dual 
bi\-vec\-tors (Section~\ref{aw}). This amounts to the analog of 
$\,\text{\smallbf w}$, defined in Section~\ref{ds}, for $\,(M\nh,g)\,$ with 
the opposite orientation, and
\begin{equation}\label{wme}
\mathrm{if\ }\,\,\mathrm{div}\,R\,=\,0\hh,\mathrm{\ one\ has\ 
}\,\,\text{\smallbf w}^-\hs=\,\hs\text{\smallbf w}\hskip12pt\mathrm{unless\ 
}\,g\,\mathrm{\ is\ an\ Ein\-stein\ metric.}
\end{equation}
Namely, (\ref{drz}) and Lemma~\ref{simdi} then imply simultaneous 
di\-ag\-o\-nal\-iz\-abil\-i\-ty of $\,\mathrm{Ric}\,$ and the Weyl tensor 
$\,W$ at every point of an open dense subset of $\,M\nh$, cf.\ (\ref{ana}), 
and we can apply Lemma~\ref{sgesg} to $\,A=W\nh(x)\,$ at any point 
$\,x\in M\nh$.
\begin{lemma}\label{onetw}A non-Ein\-stein oriented\/ $\,4$-man\-i\-fold\/ 
$\,(M\nh,g)\,$ with\/ $\,\mathrm{div}\,R=0$ has\/ 
$\,\text{\smallbf w}=\nnh1$, or\/ $\,\text{\smallbf w}=\nh2$, if and only if\/ 
$\,g\,$ is con\-for\-mal\-ly\ flat or, respectively, every point of\/ $\,M\,$ 
lies in an open sub\-man\-i\-fold of type\/ 
{\rm(\ref{kne}.d) -- (\ref{kne}.e)}.
\end{lemma}
\begin{proof}The claim about $\,\text{\smallbf w}=\nnh1\,$ trivially 
follows from (\ref{trw}) and (\ref{wme}). The warp\-ed-prod\-uct case of 
Remark~\ref{prsfm} yields the `if' part for $\,\text{\smallbf w}=\nh2\,$ by 
showing that $\,\text{\smallbf w}\le\nh2\,$ (and 
$\,\text{\smallbf w}\ne\nnh1\,$ since the relation $\,\,K\ne-c\,$ in 
(\ref{kpc}) precludes con\-for\-mal flat\-ness via Remark~\ref{cffpr}). 
Now let $\,\text{\smallbf w}=\nh2$. By 
\cite[Prop.~1]{derdzinski-88}, $\,W^+\nh\ne0\,$ everywhere, and we may 
consider the metric $\,\hg=|W|^{2/3}g\,$ on $\,M\nh$, con\-for\-mal to $\,g$. 
Since (\ref{dvw}) gives $\,\mathrm{div}\,W\nnh=0\,$ whenever 
$\,\mathrm{div}\,R\hn=0$, (\ref{wme}) and \cite[Theorem~2]{derdzinski-88}
imply that $\,\hg$, restricted to some neighborhood of any point at which 
$\,\mathrm{Ric}\hs\ne\hs\mathrm{s}g/4$, is a product of surface metrics. Due 
to real-an\-a\-lyt\-ic\-i\-ty -- see (\ref{ana}) -- the same must also be the 
case for points $\,x\,$ having $\,\mathrm{Ric}\hs=\hs\mathrm{s}g/4$, as a 
simply connected neighborhood $\,\,U$ of $\,x\,$ contains a nonempty open 
connected subset carrying a $\,\hg$-par\-al\-lel two\hh-di\-men\-sion\-al 
distribution, and the distribution clearly admits an extension to $\,\,U\nh$. 
Some constant multiple of $\,\hg\,$ then has, locally, the form 
$\,h\times\hskip-.2pth\hn^c$ of (\ref{kne}.e), with 
$\,g=(h\times\hskip-.2pth\hn^c)/(K\nnh+c)^2\nh$. Namely, as our hypotheses 
give $\,\nabla\hs\mathrm{Ric}\ne0\,$ somewhere, this last claim follows from 
\cite[Theorem 2 and Lemma 3(i)]{derdzinski-88} for points with 
$\,\mathrm{Ric}\hs\ne\hs\mathrm{s}g/4$, while real-an\-a\-lyt\-ic\-i\-ty of 
the metrics involved (including the surface metrics $\,h,h\hn^c$ 
defined, locally, at all points of $\,M$) allows us to relax the requirement 
that $\,\mathrm{Ric}\hs\ne\hs\mathrm{s}g/4$, completing the proof.
\end{proof}
\begin{proof}[Proof of Lemma~\ref{other}]First, (i) leads to (ii). Namely, for 
(\ref{kne}.a) (or (\ref{kne}.b), or (\ref{kne}.c)) one has 
$\,\text{\smallbf r}=\nnh1\,$ (or $\,\text{\smallbf w}=\nnh1$, or local 
reducibility). In (\ref{kne}.d) or (\ref{kne}.e), the warp\-ed-prod\-uct claim 
in Remark~\ref{prsfm} and (\ref{ana}) give $\,\text{\smallbf w}\le2$.

Conversely, let (ii) hold. If $\,g\,$ is locally reducible, we have
(\ref{kne}.c) or (\ref{kne}.d), cf.\ the paragraph following (\ref{dvw}),
while the case $\,\text{\smallbf r}=\nnh1\,$ yields (\ref{kne}.a). Suppose now
that $\,g\,$ is not locally reducible and $\,\text{\smallbf r}>\nnh1$. Thus, 
$\,\nabla\hs\mathrm{Ric}\ne0$ somewhere. If $\,\text{\smallbf r}=2$, it
follows from \cite[Theorem 1(i)]{derdzinski-80}, via Remark~\ref{cnfpr}, that
$\,(M\nh,g)$ has an open sub\-man\-i\-fold con\-for\-mal to the Riemannian
product of an interval and a \hbox{three\hs-}\hskip0ptman\-i\-fold of constant
sectional curvature, making $\,g$ con\-for\-mally flat due to
Remark~\ref{cffpr} combined with (\ref{ana}), and so (\ref{cst}) then yields
(\ref{kne}.b). This leaves the cases $\,\text{\smallbf w}=\nnh1\,$ and
$\,\text{\smallbf w}=2$, in which, since $\,\text{\smallbf r}>\nnh1$,
Lemma~\ref{onetw} gives (\ref{kne}.b), (\ref{kne}.d) or (\ref{kne}.e).
\end{proof}

\section{The local types {\bigrm(\ref{kne}.a)} -- 
{\bigrm(\ref{kne}.d)}}\label{lt}
The focus of our discussion does {\it not\/} include the local types 
(\ref{kne}.a) -- (\ref{kne}.d), since each of them is of independent 
interest and has been studied extensively. We list here some known facts about 
them, in the compact case.

The simplest examples of compact Ein\-stein four-man\-i\-folds are, arguably, 
spaces of constant curvature. Their complex counterparts ($\bbCP^2\nh$, 
complex $\,2$-tori, and compact quotients of $\,\bbCH^2$) carry well-known 
K\"ah\-ler-Ein\-stein metrics, as does any Riemannian product of two oriented
surfaces having the same constant Gauss\-i\-an curvature.

Generally, for a compact complex manifold $\,M\,$ to admit a 
K\"ah\-ler-Ein\-stein metric, its Lie algebra $\,\mathfrak{h}(M)\,$ of 
holomorphic vector fields must be reductive, as shown by Matsushima 
\cite{matsushima}, while $\,c_1\w(M)\,$ has to be negative, zero or positive. 
Conversely, when $\,c_1\w(M)<0$, the Ca\-la\-bi conjecture, proved by Aubin 
\cite{aubin-78} and Yau \cite{yau-78}, guarantees that $\,M\,$ carries a 
K\"ah\-ler-Ein\-stein metric, unique up to a factor. Also, Yau's proof 
\cite{yau-78} of another conjecture made by Ca\-la\-bi implies in particular
the existence of Ric\-ci-flat K\"ah\-ler metrics on K3 surfaces (which,
besides the complex $\,2$-tori, are the only K\"ah\-ler-type compact complex
surfaces having $\,c_1\w(M)=0$). 

For del Pezzo surfaces (compact complex surfaces $\,M\,$ with $\,c_1\w(M)>0$) 
the analog of the Ca\-la\-bi conjecture is false. In fact, defining $\,M\,$ to 
be the one-point or two-point blow-up of $\,\bbCP^2\nh$, one has 
$\,c_1\w(M)>0$, yet no K\"ah\-ler-Ein\-stein metric exists on $\,M\nh$, since 
$\,\mathfrak{h}(M)\,$ is not reductive. However, these two surfaces carry 
con\-for\-mal\-ly-K\"ah\-ler Ein\-stein metrics: the former was constructed by 
Page \cite{page}, the latter discovered, much more recently, by Chen, LeBrun 
and Weber \cite{chen-lebrun-weber}.

On the other hand, Tian \cite{tian} showed that all the remaining del Pezzo 
surfaces do admit K\"ah\-ler-Ein\-stein metrics. Besides $\,\bbCP^2$ and 
$\,\bbCP^1\hskip-2pt\times\bbCP^1\nh$, these surfaces arise as $\,k$-point 
blow-ups of $\,\bbCP^2\nh$, for $\,k\in\{3,4,\dots,8\}$.

The class of con\-for\-mal\-ly flat manifolds includes spaces of constant 
curvature, as well as the Riemannian products listed in Remark \ref{cffpr}, 
and is closed under a family of con\-nect\-ed-sum operations, cf.\ 
\cite[p. 479]{schoen}. As shown by Kuiper \cite[Theorem 6]{kuiper}, a compact 
simply connected con\-for\-mal\-ly flat manifold must be con\-for\-mal\-ly 
dif\-feo\-mor\-phic to a standard sphere.

For any compact con\-for\-mal\-ly flat manifold, the additional requirement 
that the scalar curvature be constant can always be realized, according to 
Aubin's and Schoen's solutions of the Yamabe problem \cite{aubin-75,schoen}, 
by a suitable con\-for\-mal change of the metric. In dimensions 
$\,n\in\{3,4\}\,$ relevant to us, this result is due to Schoen \cite{schoen}.

\section{Compact manifolds of the local type {\bigrm(\ref{kne}.e)}}\label{tf}
Following \cite[Example 4]{derdzinski-88}, we now describe how 
(\ref{kne}.e) leads to compact Riemannian four-man\-i\-folds 
$\,(M\nh,g)\,$ having$\,\,\mathrm{div}\,R=0$, with $\,M\,$ obtained as total 
spaces of flat $\,\mathrm{SO}(3)\hs$ bundles of $2$-spheres over closed 
surfaces.
\begin{example}\label{flbdl}Given $\,c,r\in(0,\infty)$, a metric $\,h\hn^c$ of 
constant Gauss\-i\-an curvature $\,c\,$ on $\,S^2\nnh$, a closed Riemannian 
surface $\,(Q,h)\,$ with the Gauss\-i\-an curvature $\,K\,$ satisfying 
(\ref{kpc}) as well as having $\,K\nnh+c>0$ everywhere in $\,Q$, and a group 
homo\-mor\-phism $\,\varphi:\pi\to\mathrm{SO}(3)$, for 
$\,\pi=\pi\nnh_1\w\nh Q$, we define $\,(\tilde M\nh,\tilde g)$ to be the 
manifold obtained when, in (\ref{kne}.e), one sets $\,S=S^2$ and, instead of 
$\,(Q,h)$, uses its Riemannian universal covering space 
$\,(\tilde Q,\tilde h)$. Then $\,\tilde g\,$ descends to a metric $\,g\,$ on 
the quotient manifold $\,M=\tilde M/\pi$, the free properly discontinuous 
action of $\,\pi\,$ on $\,\tilde M=\tilde Q\nh\times S^2$ by 
$\,\tilde g$-i\-so\-me\-tries  \cite[Example 4]{derdzinski-88} being given by 
$\,\gamma(x,y)=(\gamma(x),[\varphi(\gamma)](y))\,$ whenever 
$\,(\gamma,x,y)\in\pi\times\tilde Q\nh\times S^2\nnh$, with 
$\,\gamma(x)\,$ corresponding to the action of $\,\pi\,$ on $\,\tilde Q\,$ via 
deck transformations. 
\end{example}
An obvious question that arises is whether Example \ref{flbdl} really gives 
rise to anything interesting, which here means manifolds not belonging to the 
local types (\ref{kne}.a), (\ref{kne}.b), (\ref{kne}.c) or (\ref{kne}.d). 
As explained below, the answer 
is known to be `yes' for $\,Q\,$ ho\-meo\-mor\-phic to $\,S^2\nh$, or 
$\,\bbRP^2\nh$, or a closed orientable surface of genus greater than $\,1$.

According to \cite[Proposition 4]{derdzinski-88} (or, 
\cite[Proposition 2]{derdzinski-88}), on the closed orientable surface of any 
genus $\,p>1\,$ (or, respectively, on $\,\bbRP^2$) there exists 
an uncountable set $\,\mathcal{E}\,$ of pairwise non\-homo\-thet\-ic metrics 
$\,h\,$ having the properties required in Example \ref{flbdl} (and, in the 
case of $\,\bbRP^2$, rotationally invariant). The set $\,\mathcal{E}\,$ is 
ho\-meo\-mor\-phic to $\,\bbR^{6p-5}$ and contains a co\-di\-men\-sion-one 
subset formed by metrics of constant Gauss\-i\-an curvature; or, respectively, 
$\,\mathcal{E}\,$ is the union of a countably infinite family of subsets 
ho\-meo\-mor\-phic to $\,\bbR\,$ which all contain a fixed 
con\-stant-cur\-va\-ture metric, and are otherwise mutually disjoint.

Any $\,h\,$ as above on $\,\bbRP^2$ can obviously be pulled back to $\,S^2\nh$.

The metrics $\,h\,$ just mentioned all give rise, as in Example \ref{flbdl}, 
to compact manifolds of the local type (\ref{kne}.e) which do not 
simultaneously represent any of the local types (\ref{kne}.a), (\ref{kne}.b) 
or (\ref{kne}.c) \cite[Theorems 4 and 5]{derdzinski-88}. However, all such 
non\-flat metrics known to exist on the torus or Klein bottle lead to 
four-man\-i\-folds that also belong to type (\ref{kne}.c). See 
\cite[Example 5]{derdzinski-88}.

\section{Outline of proof of
Theorem~\ref{ricev}{\rm(a)\hs--\hs(b)}}\label{sd}
For a fixed oriented Riemannian 
four-man\-i\-fold $\,(M\nh,g)\,$ with $\,\,\mathrm{div}\,R=0$, let 
$\,\text{\smallbf r}\in\{1,2,3,4\}\,$ and 
$\,\text{\smallbf w}^\pm\in\{1,2,3\}\,$ denote the maximum number of 
distinct eigen\-values of the Ric\-ci tensor $\,\,\mathrm{Ric}\,\,$ (acting on 
the tangent bundle $\,T\nh M$) and, respectively, of the (anti)self-dual Weyl 
tensor $\,W^\pm$, acting on the bundle $\,\Lambda^{\!\pm\!}M\,$ of 
(anti)self-dual bi\-vec\-tors. See Section \ref{aw}. For simplicity we write
$\,\text{\smallbf w}\,$ instead of $\,\text{\smallbf w}^+\nh$. Note that, by
(\ref{wme}), $\,\text{\smallbf w}^-\nh=\text{\smallbf w}\,$ unless
$\,(M\nh,g)\,$ is an Ein\-stein manifold. In view of DeTurck and
Goldschmidt's result (\ref{ana}), there is a dense open subset of $\,M\,$ 
consisting of {\it generic\/} points, meaning
\begin{equation}\label{gnr}
\mathrm{points\ at\ which\ the\ maxima\ 
}\,\text{\smallbf r},\hs\text{\smallbf w}\,\mathrm{\ are\ simultaneously\ 
attained.}
\end{equation}
Four possible cases may occur:
\begin{enumerate}
  \def\theenumi{{\rm\Alph{enumi}}}
\item $\text{\smallbf r}\,=\hs1$: an Ein\-stein manifold -- type
  (\ref{kne}.a).
\item $\text{\smallbf r}\,>\hs1\,$ and $\,\hs\text{\smallbf w}=\hskip-1pt1$: 
type (\ref{kne}.b), as a consequence of Lemma~\ref{onetw}.
\item $\text{\smallbf r}\,>\hs1\,$ and $\,\hs\text{\smallbf w}=2\hs$: locally,
type (\ref{kne}.d) or (\ref{kne}.e) -- see Lemma~\ref{onetw}.
\item $\text{\smallbf r}>\nnh1\,$ and $\,\hs\text{\smallbf w}=\hn3$. By 
Lemma~\ref{codwe}(i), some neighborhood $\,\,U$ of any generic point $\,x\,$ 
admits or\-tho\-nor\-mal analytic vector fields $\,e\nh_1\w,\ldots,e\nh_4\w$ 
which di\-ag\-o\-nal\-ize both $\,\,W\,$ (in the sense of Section \ref{aw}),
and $\,\,\mathrm{Ric}$. 
\end{enumerate}
In case (D), let $\,\text{\smallbf d}\in\{0,1,2,3,4\}\,$ denote the maximal 
number of integers $\,l\in\{1,2,3,4\}$ for which there exist $\,i,j,k\,$ with 
$\,\{i,j,k,l\}=\{1,2,3,4\}$ and 
$\,g(\nabla\hskip-2.7pt_{e\nh_i\w}\hskip-2.4pte\nh_j\w,e\nh_k\w)\ne0\,$
somewhere in $\,\,U\nh$. As shown in Lemma \ref{zotwo} and Section~\ref{nt}, 
$\,\text{\smallbf d}\notin\{2,3,4\}$, so that there are just two possible
subcases:
\begin{enumerate}
  \def\theenumi{{\rm\Roman{enumi}}}
\item[{\rm(D1)}] $\text{\smallbf d}\,=\,1$: according to
Theorem~\ref{nuone}, $\,(M\nh,g)\,$ is, locally, of type (\ref{kne}.c).
\item[{\rm\hbox{(D\hs0)}}] $\text{\smallbf d}=0$, that is, 
$\,g(\nabla\hskip-2.7pt_{e\nh_i\w}\hskip-2.4pte\nh_j\w,e\nh_k\w)=0\,$
on $\,\,U$ whenever $\,i\ne j\ne k\ne i$. 
\end{enumerate}
Subcase (D0) clearly yields assertions (a) -- (b) in Theorem~\ref{ricev},
under the assumption (equivalent, by Lemma~\ref{other}, to the hypotheses
of Theorem~\ref{ricev}) that $\,(M\nh,g)\,$ contains no open
sub\-man\-i\-folds of types (\ref{kne}.a) -- (\ref{kne}.e). 
With Lemma~\ref{onetw} already established, {\it{\rm(a) --
(b)} in Theorem\/~{\rm\ref{ricev}} will thus follow from 
Lemmas\/~{\rm~\ref{codwe}(i), \ref{zotwo}}, the claims made in
Section\/~{\rm\ref{nt}}, and Theorem\/~{\rm\ref{nuone}}}.
\begin{remark}\label{nonvc}According to \cite[the lines following formula 
(0.3)]{derdzinski-piccione}, there exist complete, locally irreducible, 
non-Ric\-ci-par\-al\-lel Riemannian four-man\-i\-folds $\,(M\nh,g)$, which 
are not con\-for\-mal\-ly flat, having -- in addition to some further 
properties -- harmonic curvature and $\,\text{\smallbf r}=4$. (Their 
lo\-cal-i\-som\-e\-try types form a five\hh-di\-men\-sion\-al moduli space.) 
All those manifolds satisfy the hypotheses of our Theorem~\ref{ricev}. In
fact, $\,\text{\smallbf w}\,$ must equal $\,3$, as the case 
$\,\text{\smallbf w}\le2$ would, by Lemma~\ref{onetw}, lead to the local 
type (\ref{kne}.d) or (\ref{kne}.e), 
making $\,(M\nh,g)$, locally, a 
warped product 
with a two\hh-di\-men\-sion\-al fibre and harmonic curvature. Consequently 
\cite[Corollary 1.3]{derdzinski-piccione}, its Ric\-ci tensor would have a 
multiple eigen\-value at every point, contrary to the relation 
$\,\text{\smallbf r}\hs=4$.

It is not known whether the above class 
contains any compact manifolds.
\end{remark}

\section{The Co\-daz\-zi-\nnh Weyl simultaneous 
di\-ag\-o\-nal\-iz\-abil\-i\-ty}\label{cw}
Unlike in Section \ref{pr}, from now on repeated indices are {\it not\/}
summed over.
\begin{lemma}\label{codwe}Let there be given an oriented Riemannian
four-man\-i\-fold\/ $\,(M\nh,g)\,$ with a Co\-daz\-zi tensor field\/ $\,b\hs$
on\/ $\,(M\nh,g)\,$ having\/ $\,4b\ne(\mathrm{tr}\nh_g\w\hh b)g\,$ everywhere 
and an algebraic Weyl tensor field\/ $\,A\,$ such that\/
$\,\mathrm{div}\,A=0$, while\/ $\,A,b\,$ are simultaneously
di\-ag\-o\-nal\-iz\-able at each point in the sense of
Section\/~{\rm\ref{aw}}, and the bundle morphism\/
$\,A\nh^+:\Lambda^{\!+\!}M\to\Lambda^{\!+\!}M\,$ arising as the restriction of
$\,A:[TM]^{\wedge2}\to[TM]^{\wedge2}$ to self-dual bi\-vec\-tors, cf.\/ 
{\rm(\ref{apm})}, has three distinct eigen\-values at every point of\/
$\,M\nh$.

The above hypotheses imply the following conclusions.
\begin{enumerate}
  \def\theenumi{{\rm\roman{enumi}}}
\item An or\-tho\-nor\-mal frame\/ $\,e\nh_1\w,\dots,e\nh_4\w$ 
di\-ag\-o\-nal\-izing both\/ $\,A\,$ and\/ $\,b\hs$ at any $\,x\in M\,$ is
unique up to permuting and\hs/\hn or changing signs of\/ $\,e\nh_i\w$ and, 
passing to a finite covering of\/ $\,M\,$ if necessary, we may assume that\/
$\,e\nh_i\w$ are\/ $\,C^\infty$ vector fields on $\,M\nh$.
\item The directional derivative $\,D\nnh_i\w$ in the direction 
of $\,e\nh_i\w$ and the functions 
$\,\varGamma^k_{\nh\!ij},\lambda_i\w,\sigma\nnh_{ij}\w$ given, with\/
$\,i,j,k,l\,$ ranging over\/ $\,\{1,2,3,4\}$, by
\begin{equation}\label{gij}
\varGamma^k_{\nh\!ij}\,
=\,g(\nabla\hskip-2.7pt_{e\nh_i\w}\hskip-2.4pte\nh_j\w,e\nh_k\w)\hh,
\quad\lambda_i\w=b(e\nh_i\w,e\nh_i\w)\hs\quad\sigma\nnh_{ij}\w
=A(e\nh_i\w,e\hn_j,e\nh_i\w,e\hn_j)\hh,
\end{equation}
satisfy, whenever $\,\{i,j,k,l\}=\{1,2,3,4\}$, the conditions
\begin{equation}\label{skw}
\begin{array}{rl}
\mathrm{a)}&\varGamma^k_{\nh\!ij}\hs+\,\varGamma^j_{\!ik}\hs=\,0\hskip10pt
\mathrm{and}\hskip10pt\sigma\nnh_{ij}\w\,\ne\,\sigma\nnh_{ik}\w\,
\ne\,\sigma\nnh_{il}\w\,\ne\,\sigma\nnh_{ij}\w\hh,\\
\mathrm{b)}&\sigma\nnh_{ij}\w\,=\,\sigma\hskip-2.7pt_{ji}\w\,
=\,\sigma\nnh_{kl}\w\hskip10pt\mathrm{and}\hskip10pt
\sigma\nnh_{ij}\w\,+\,\sigma\nnh_{ik}\w\,+\,\sigma\nnh_{il}\w\,=\,0\hh,\\
\mathrm{c)}&(\lambda_j\w-\lambda_k\w)\varGamma^k_{\nh\!ij}\,=\,
(\lambda_k\w-\lambda_i\w)\varGamma^{\hs i}_{\!jk}\,=\,
(\lambda_i\w-\lambda_j\w)\varGamma^j_{\!ki}\hh,\\
\mathrm{d)}&(\sigma\nnh_{ij}\w\nh-\sigma\nnh_{ik}\w)\varGamma^k_{\nh\!ij}
=(\sigma\hskip-2.7pt_{jk}\w\nh-\sigma\hskip-2.7pt_{ji}\w)\varGamma^{\hs i}_{\!jk}
=(\sigma\nnh_{ki}\w\nh-\sigma\nnh_{kj})\w\varGamma^j_{\!ki}\hh,\\
\mathrm{e)}&D\nnh_i\w\lambda_j\w\,
=\,(\lambda_j\w-\lambda_i\w)\varGamma^{\hs i}_{\!j\hn j}
\hh,\\
\mathrm{f)}&D\nnh_j\w\sigma\nnh_{ij}\w
=\hs(\sigma\nnh_{ij}\w\nh-\sigma\nnh_{ik}\w)\varGamma^j_{\!kk}\nh
+\hs(\sigma\nnh_{ij}\w\nh-\sigma\nnh_{il}\w)\varGamma^j_{\!ll}\hh.
\end{array}
\end{equation}
\item At any point of\/ $\,M\nh$, the morphism\/ 
$\,A\nh^+:\Lambda^{\!+\!}M\to\Lambda^{\!+\!}M\,$ has the same 
eigen\-values, including multiplicities, as\/ 
$\,A\nh^-:\Lambda^{\!-\!}M\to\Lambda^{\!-\!}M\nh$.
\end{enumerate}
\end{lemma}
\begin{proof}Both (iii) and (\ref{skw}.a) -- (\ref{skw}.b) trivially follow 
from Lemma~\ref{sgesg} and the fact that $\,e\nh_i\w$ are or\-tho\-nor\-mal,
while (i) is immediate from Remark~\ref{eottf} along with the essential
uniqueness of eigen\-vectors of $\,A\nh^\pm\nnh$, which itself is due to 
the assumption about eigen\-values. Equalities (\ref{skw}.c) -- (\ref{skw}.f)
amount in turn to the Co\-daz\-zi equation for $\,b\,$ and the relation
$\,\,\mathrm{div}\,A=0$.
\end{proof}
Suppose now that an oriented Riemannian four-man\-i\-fold $\,(M\nh,g)\,$ has
\begin{equation}\label{rrw}
\mathrm{div}\,R\,=\,0\,\,\mathrm{\ with\ 
}\,\,\text{\smallbf r}>\nnh1\,\,\mathrm{\ and\ }\,\,\text{\smallbf w}=3\hh,
\end{equation}
$\text{\smallbf r}\,\,$ and $\,\,\text{\smallbf w}\,\,$ being defined as at
the beginning of Section \ref{sd} (or Section~\ref{ds}). The hypotheses of
Lemma~\ref{codwe} are then satisfied by $\,A\,$ equal to the Weyl
con\-for\-mal tensor $\,W\hs$ and the traceless Ric\-ci tensor 
$\,b=\hh\mathrm{Ric}\,-\,\mathrm{s}g/4\,$ of $\,g$, on any fixed connected
component of the dense open set of generic points, defined as in (\ref{gnr}).
This follows from (\ref{drz}), (\ref{dvw}), Lemma~\ref{simdi} with
$\,\text{\smallbf r}>1$, and the equality $\,\,\text{\smallbf w}=3$.
(The same would be true if we set $\,b=\,\mathrm{Ric}\,$ instead.) The
conclusions of Lemma~\ref{codwe} thus hold as well, which makes (\ref{wme}) a
consequence of Lemma~\ref{codwe}(iii). Furthermore, (\ref{rrw}) also 
implies that, whenever $\,i\ne j$, 
\begin{equation}\label{snz}
\sigma\nnh_{ij}\w\,\ne\,0\quad\mathrm{everywhere\ in\ some\ dense\ open\ 
subset\ of\ }\,\,M\hh.
\end{equation}
Otherwise, let $\,\sigma\nnh_{ij}\w=0\,$ on a nonempty open set; since 
$\,\sigma\nnh_{ij}\w$ is an eigen\-value function of both $\,W^+$ and
$\,W^-\nh$, using \cite[Proposition 16.72]{besse} one then gets
$\,W\hn\nnh=\hh0\,$ on that set, even though $\,\,\text{\smallbf w}=3$.
Finally, setting
\begin{equation}\label{rkl}
R_{ijkl}\w\,=\,g(R\hh(e\nh_i\w,e\nh_j\w)e\nh_k\w,\,e\nh_l\w)\qquad\quad
\mathrm{for}\quad i,j,k,l\in\{1,2,3,4\}\hh,
\end{equation}
we obtain (\ref{jij}) from (\ref{wij}) and (\ref{gij}) for $\,A=W\,$ and 
$\,b=\,\mathrm{Ric}\,-\,\mathrm{s}g/4$. Also,
\begin{equation}\label{riz}
R_{ijkl}\w\,=\,0\,\mathrm{\ unless\
}\,\{i,j\}=\{k,l\}\subseteq\{1,2,3,4\}\,\mathrm{\ is\ a\
}\,2\hyp\mathrm{element\ set,}
\end{equation}
as $\,R_{ijkl}\w=0\,$ in (\ref{wij}) when $\,\{i,j,k,l\}\,$ has more than two
elements, it being clearly the case for all the other terms of (\ref{wij}), 
where the components refer this time to the frame in Lemma~\ref{codwe}(i),
for $\,(A,b)=(W\hskip-2.3pt,\,\mathrm{Ric}-\mathrm{s}\hskip.7ptg/4)$.

Next, we may define the functions $\,\lx\hn_l\w$ and $\,\bz_l\w$, 
$\,l=1,2,3,4$, by 
\begin{equation}\label{ale}
\lx\hn_l\w=(\sigma\nnh_{ij}\w\nh-\sigma\nnh_{ik}\w)\varGamma^k_{\nh\!ij}\,,
\hskip6pt
\bz_l\w=(\lambda_j\w-\lambda_k\w)\varGamma^k_{\nh\!ij}\hskip4pt\mathrm{with}
\hskip4pt\{i,j,k,l\}=\{1,2,3,4\}\hh,
\end{equation}
since, due to (\ref{skw}.a) -- (\ref{skw}.d), the definition is correct,
namely, $\,\lx\hn_l\w$ and $\,\bz_l\w$ do not depend on the choice of $\,i,j\,$
and $\,k$.
\begin{remark}\label{alleq}If $\,\{i,j,k,l\}=\{1,2,3,4\}$, then 
$\,\lambda_i\w,\lambda_j\w,\lambda_k\w$ are all equal (or, all distinct)
wherever $\,\lx\hn_l\w\ne0=\bz_l\w$ (or, respectively,
$\,\lx\hn_l\w\ne0\ne\bz_l\w$). In fact, (\ref{ale}) with $\,\lx\hn_l\w\ne0\,$
gives $\,\varGamma^k_{\nh\!ij}\ne0$, and so, again by (\ref{ale}), the
relation $\,\bz_l\w=0\,$ (or, $\,\bz_l\w\ne0$) yields
$\,\lambda_i\w\nh=\lambda_j\w\nh=\lambda_k\w$ (or,
$\,\lambda_i\w\nh\ne\lambda_j\w\nh\ne\lambda_k\w\nh\ne\lambda_i\w$).
\end{remark}
\begin{lemma}\label{ggnez}Under the hypotheses of Lemma\/~{\rm\ref{codwe}},
let\/ $\,l\in\{1,2,3,4\}\,$ and\/ $\,x\in M\,$ be such that\/
$\,\lx\hn_l\w(x)\ne0\ne\bz_l\w(x)\,$ in\/ {\rm(\ref{ale})}. Then the function\/
$\,\alpha$ defined, on a neighborhood of\/ $\,x$, by
\begin{equation}\label{aea}
\alpha\,=\,\lx\hn_l\w/\nh\bz_l\w\,\ne\,0\hh,
\end{equation}
satisfies, for any\/ $\,i,j,k\,$ with\/ $\,\{i,j,k,l\}=\{1,2,3,4\}$, the
relations
\begin{enumerate}
  \def\theenumi{{\rm\roman{enumi}}}
\item $\sigma\nnh_{ij}\w\,-\,\sigma\nnh_{ik}\w\,
=\,(\lambda_j\w\,-\,\lambda_k\w)\alpha$,
\item $3\sigma\nnh_{ij}\w=(\lambda_i\w+\lambda_j\w-2\lambda_k\w)\alpha$,  
\item $D\nnh_k\w\hh\sigma\nnh_{ik}\w\,=\,D\nnh_k\w\hh\sigma\nnh_{lj}\w\,=\,
[(\lambda_k\w\,-\,\lambda_j\w)\varGamma^k_{\!j\hn j}\,+
\,(\lambda_i\w\,-\,\lambda_j\w)\varGamma^k_{\!ll}]\hs\alpha$,
\item $D\nnh_i\w\hh\sigma\hskip-2.7pt_{jk}\w\,=\,D\nnh_i\w\hh\sigma\nnh_{li}\w\,=\,
[(\lambda_j\w\,-\,\lambda_i\w)\varGamma^{\hs i}_{\!j\hn j}\,+
\,(\lambda_k\w\,-\,\lambda_i\w)\varGamma^{\hs i}_{\!kk}]\hs\alpha$,
\item $D\nnh_k\w\hh\sigma\hskip-2.7pt_{jk}\w\,=\,D\nnh_k\w\hh\sigma\nnh_{li}\w\,=\,
[(\lambda_k\w\,-\,\lambda_i\w)\varGamma^k_{\!ii}\,+
\,(\lambda_j\w\,-\,\lambda_i\w)\varGamma^k_{\!ll}]\hs\alpha$,
\item $D\nnh_i\w\hh\sigma\nnh_{ik}\w\,=\,D\nnh_i\w\hh\sigma\hskip-2.7pt_{jl}\w\,
=\,[(\lambda_i\w\,-\,\lambda_j\w)\varGamma^{\hs i}_{\!j\hn j}\,+
\,(\lambda_k\w\,-\,\lambda_j\w)\varGamma^{\hs i}_{\!ll}]\hs\alpha$,

\item $D\nh_l\w\hh\sigma\nnh_{ik}\w\,=\,D\nh_l\w\hh\sigma\hskip-2.7pt_{jl}\w\,
=\,[(\lambda_k\w\,-\,\lambda_j\w)\varGamma^{\hs l}_{\!ii}\,+
\,(\lambda_i\w\,-\,\lambda_j\w)\varGamma^{\hs l}_{\!kk}]\hs\alpha$,

\item $D\nh_l\w\hh\sigma\hskip-2.7pt_{jk}\w\,=\,D\nh_l\w\hh\sigma\nnh_{il}\w\,
=\,[(\lambda_k\w\,-\,\lambda_i\w)\varGamma^{\hs l}_{\!j\hn j}\,+
\,(\lambda_j\w\,-\,\lambda_i\w)\varGamma^{\hs l}_{\!kk}]\hs\alpha$,

\item $D\nnh_i\w\alpha\,=\,2\alpha\varGamma^{\hs i}_{\!ll}$,
\item $D\nnh_i\w\lambda_i\w\,=
\,(\lambda_i\w\,-\,\lambda_j\w)\varGamma^{\hs i}_{\!j\hn j}\,+
\,(\lambda_i\w\,-\,\lambda_k\w)\varGamma^{\hs i}_{\!kk}\,+
\,(\lambda_j\w+\lambda_k\w-2\lambda_i\w)\varGamma^{\hs i}_{\!ll}$,
\item $(\lambda_i\w-\lambda_j\w)(D\nh_l\w\alpha-2\alpha\varGamma^l_{\!kk})
=\alpha\hs(2\lambda_i\w+2\lambda_j\w
-\mathrm{tr}\nh_g\w\hh b)(\varGamma^l_{\!j\hn j}-\varGamma^l_{\!ii})$,
\item $D\nnh_i\w(\mathrm{tr}\nh_g\w\hh b)\,=
\,(\mathrm{tr}\nh_g\w\hh b\,-\,4\lambda_i\w)\varGamma^{\hs i}_{\!ll}$.
\item $\sigma\nnh_{ki}\w-\sigma\nnh_{kj}\w
=(\lambda_i\w\,-\,\lambda_j\w)\alpha$, which is the $\,i,j\,$ version of\/
{\rm(i)}.
\end{enumerate}
\end{lemma} 
\begin{proof}Fix $\,i,j,k,l\,$ with $\,\{i,j,k,l\}=\{1,2,3,4\}\,$ and
$\,\lx\hn_l\w\bz_l\w\ne0\,$ at $\,x$. Then (i) is obvious from (\ref{aea}) and
(\ref{ale}), while adding (i) to its version obtained by interchanging
$\,i,j\,$ and using (\ref{skw}.b) we get (ii). Next,
(\ref{skw}.f), with $\,j,k$ switched, (\ref{skw}.b), (i) and (xiii) yield (iii). 
Similarly, (\ref{skw}.f) for $\,l,i\,$ rather than $\,i,j$, (\ref{skw}.b), (i) 
and (xiii) imply (iv). Invariance of our assumptions under permutations of 
$\,i,j,k\,$ gives (v) and (vi) from (iii) by switching $\,i\,$ with $\,j$,
or $\,i\,$ with $\,k$, and using (\ref{skw}.b). 
Next, (vii) follows since
$\,D\nh_l\w\hh\sigma\nnh_{ik}\w=D\nh_l\w\hh\sigma\hskip-2.7pt_{jl}\w
=(\sigma\hskip-2.7pt_{jl}\w-\sigma\hskip-2.7pt_{ji}\w)\varGamma^{\hs l}_{\!ii}+
\,(\sigma\hskip-2.7pt_{jl}\w-\sigma\hskip-2.7pt_{jk}\w)\varGamma^{\hs l}_{\!kk}$ due to
(\ref{skw}.b) and (\ref{skw}.f), 
while $\,\sigma\hskip-2.7pt_{jl}\w-\sigma\hskip-2.7pt_{ji}\w
=\sigma\nnh_{ik}\w-\sigma\nnh_{ij}\w=(\lambda_k\w-\lambda_j\w)\alpha\,$ and 
$\,\sigma\hskip-2.7pt_{jl}\w-\sigma\hskip-2.7pt_{jk}\w
=\sigma\nnh_{ki}\w-\sigma\nnh_{kj}\w=(\lambda_i\w-\lambda_j\w)\alpha\,$
by (\ref{skw}.b), (i) and (xiii). Switching $\,i\,$ with $\,j\,$ in (vii), we
obtain (viii). Applying 
$\,D\nnh_k\w$ or $\,D\nnh_i\w$ to (xiii), we now see that 
$\,D\nnh_k\w\sigma\nnh_{ik}\w-D\nnh_k\w\sigma\hskip-2.7pt_{jk}\w
-(\lambda_i\w-\lambda_j\w)D\nnh_k\w\alpha
-\hbox{$\alpha(D\nnh_k\w\lambda_i\w-D\nnh_k\w\lambda_j\w)=0$}$ as well as 
$\,\alpha D\nnh_i\w\lambda_i\w=D\nnh_i\w\hh\sigma\nnh_{ik}\w
-D\nnh_i\w\hh\sigma\hskip-2.7pt_{jk}\w-(\lambda_i\w-\lambda_j\w)D\nnh_i\w\alpha
+\alpha D\nnh_i\w\lambda_j\w$.
The first of these two equalities becomes
$\,(\lambda_j\w-\lambda_i\w)
[D\nnh_k\w\alpha-2\alpha\varGamma^k_{\!ll}]=0\,$ if one replaces the
directional derivatives with the corresponding right-hand sides in
(iii), (v) and (\ref{skw}.e). As $\,\bz_l\w\ne0\,$ in (\ref{ale}),
Remark~\ref{alleq} now yields (ix). Analogously, (v), (iv), (ix) and 
(\ref{skw}.e) combined with the second equality
give (x) multiplied by $\,\alpha$, which implies (x), since 
$\,\alpha\ne0\,$ by (i) and (\ref{skw}.a). Also, $\,D\nh_l\w$ applied to 
(xiii) gives
$\,(\lambda_i\w\,-\,\lambda_j\w)D\nh_l\w\alpha
=D\nh_l\w\sigma\nnh_{ki}\w-D\nh_l\w\sigma\nnh_{kj}\w
+[(\lambda_j\w-\lambda_l\w)\varGamma^{\hs l}_{\!j\hn j}+
  (\lambda_i\w\,-\,\lambda_l\w)\varGamma^{\hs l}_{\!ii}]\hs\alpha$, where
$\,D\nh_l\w\lambda_j\w$ and $\,D\nh_l\w\lambda_i\w$ have been replaced with
the expressions provided by (\ref{skw}.e). Using (vii) and (viii), we now 
easily get (xi). Finally, as 
$\,\mathrm{tr}\nh_g\w\hh b=\lambda_i\w+\lambda_j\w+\lambda_k\w
+\lambda_l\w$, (x) and (\ref{skw}.e) yield (xii), completing the proof. 
\end{proof}

\section{Three a priori possible cases}\label{ap}
Throughout this section we assume  the hypotheses of Lemma~\ref{codwe} and use
the notations of its conclusions. 
Let $\,\text{\smallbf d}\in\{0,1,2,3,4\}$ be the maximum value in $\,M\,$ of 
the function $\,\zeta$ assigning to $\,x\in M\,$ the number $\,\zeta(x)\,$ 
of indices $\,i\in\{1,2,3,4\}\,$ for which $\,\lx\hn_i\w(x)\ne0\,$ in
(\ref{ale}). Obviously, 
$\,\zeta=\text{\smallbf d}\,$ on some nonempty open subset of $\,M\nh$.
\begin{remark}\label{invrd}The invariant $\,\text{\smallbf d}\,$ is still 
well-defined if, instead of the hypotheses of Lemma~\ref{codwe}, one assumes
(\ref{rrw}): we then just take the maximum of $\,\zeta\,$ over the (possibly
disconnected) dense open set of generic points, cf.\ (\ref{gnr}).
\end{remark}
\begin{lemma}\label{zotwo}Under the hypotheses of Lemma\/~{\rm\ref{codwe}}, 
$\,\text{\smallbf d}\in\{0,1,2\}$. In other words, for the functions\/
$\,\lx\hn_i\w$ given by\/ {\rm(\ref{ale})} we have\/ $\,\lx\hn_i\w\lx\nh_j\w\lx\nh_k\w=0\,$
whenever\/ $\,i,j,k\,$ are all distinct, that is, at any point of\/ $\,M\,$ at 
least two of\/ $\,\lx\hn_i\w$ must be zero. Furthermore, with\/ $\,\lambda_i\w$
and\/ $\,\bz_i\w$ as in\/ {\rm(\ref{gij})} and\/ {\rm(\ref{ale})},
\begin{equation}\label{aka}
\begin{array}{rl}
\mathrm{a)}&\lx\nh_k\w\lx\hn_l\w(\lambda_i\w+\lambda_j\w-\lambda_k\w-\lambda_l\w)\,=\,0
\quad\mathrm{if}\quad
\{i,j,k,l\}=\{1,2,3,4\}\hh,\\
\mathrm{b)}&\bz_k\w/\lx\nh_k\w\,=\,-\hs \bz_l\w/\lx\hn_l\w\quad\mathrm{whenever}\quad k\ne 
l\quad\mathrm{and}\quad\lx\nh_k\w\lx\hn_l\w\ne0\hh. 
\end{array}
\end{equation}
\end{lemma}
\begin{proof}Fix $\,i,j,k,l\,$ with $\,\{i,j,k,l\}=\{1,2,3,4\}$. To prove 
(\ref{aka}), suppose that $\,\lx\nh_k\w\lx\hn_l\w\ne0\,$ at some
$\,x\in M\nh$. Then, from (\ref{ale}) and (\ref{skw}.b), at $\,x$,
\[
\frac{\lambda_i\w-\lambda_k\w}{\sigma\hskip-2.7pt_{ji}\w\!-\!\sigma\hskip-2.7pt_{jk}\w}=
\frac{\bz_l\w}{\lx\hn_l\w}=
\frac{\lambda_i\w-\lambda_j\w}{\sigma\nnh_{ki}\w\!-\!\sigma\nnh_{kj}\w}=
-\frac{\lambda_i\w-\lambda_j\w}{\sigma\nnh_{li}\w\!-\!\sigma\nnh_{lj}\w}=
-\frac{\bz_k\w}{\lx\nh_k\w}=
-\frac{\lambda_j\w-\lambda_l\w}{\sigma\nnh_{ij}\w\!-\!\sigma\nnh_{il}\w}=
\frac{\lambda_l\w-\lambda_j\w}{\sigma\hskip-2.7pt_{ji}\w\!-\!\sigma\hskip-2.7pt_{jk}\w}\hh,
\]
the denominators being nonzero by (\ref{skw}.a). Now (\ref{aka}) follows.

As for the first claim, suppose on the contrary that 
$\,\{i,j,k,l\}=\{1,2,3,4\}$ and $\,\lx\hn_i\w\lx\nh_j\w\lx\nh_k\w\ne0\,$ in an open 
subset $\,U'$ of $\, M\nh$. Then, by (\ref{aka}.a), 
$\,\lambda_i\w+\lambda_j\w=\lambda_k\w+\lambda_l\w$, 
$\,\,\lambda_j\w+\lambda_k\w=\lambda_i\w+\lambda_l\w$, and 
$\,\lambda_i\w+\lambda_k\w=\lambda_j\w+\lambda_l\w$ everywhere in $\,U'$. This 
gives $\,\lambda_i\w=\lambda_j\w=\lambda_k\w=\lambda_l\w$, even though
Lemma~\ref{codwe} assumes that $\,4b\ne(\mathrm{tr}\nh_g\w\hh b)g$.
\end{proof}
We now restrict our discussion to any fixed connected component $\,U$ of the 
nonempty open subset of $\,M\,$ in which $\,\zeta=\text{\smallbf d}$. Also,
let us rearrange $\,\lx\hn_i\w$ (that is, the or\-tho\-nor\-mal vector fields 
$\,e\nh_i\w$) so as to make those among $\,\lx_1\w,\lx_2\w,\lx_3\w,\lx\nh_4\w$ which 
vanish on $\,U$ precede those which do not. In view of Lemma~\ref{zotwo}, 
one of three cases must occur: 
\begin{equation}\label{toz}
\begin{array}{rl}
\mathrm{a)}&\text{\smallbf d}\hh=2\hskip9pt\mathrm{and}\hskip9pt\lx_1\w=\lx_2\w
=0\ne\lx_3\w\lx\nh_4\w\hskip22.2pt\mathrm{everywhere\ in\ }\,\,U,\\
\mathrm{b)}&\text{\smallbf d}\hh=1\hskip9pt\mathrm{and}\hskip9pt\lx_1\w=\lx_2\w
=\lx_3\w=0\ne\lx\nh_4\w\hskip9pt\mathrm{at\ each\ point\ of\ }\,\,U,\\
\mathrm{c)}&\text{\smallbf d}\hh=0\hskip9pt\mathrm{and}\hskip9pt\lx_1\w=\lx_2\w
=\lx_3\w=\lx\nh_4\w=0\hskip9pt\mathrm{identically\ in\ }\,\,U.
\end{array}
\end{equation}
In view of (\ref{skw}.a) and (\ref{ale}), there are the following implications.
\begin{equation}\label{abc}
\begin{array}{l}
\mathrm{Case\ (\ref{toz}.a)\hskip-2.5pt:}\hskip12pt
\varGamma^k_{\nh\!ij}\,=\,0\quad\mathrm{if}\quad i\ne j\ne k\ne i
\quad\mathrm{and}\quad3,4\in\{i,j,k\}\hh.\\
\mathrm{Case\ (\ref{toz}.b)\hskip-3pt:}\hskip12pt
\varGamma^k_{\nh\!ij}\,=\,0\quad\mathrm{if}\quad i\ne j\ne k\ne i
\quad\mathrm{and}\quad4\in\{i,j,k\}\hh.\\
\mathrm{Case\ (\ref{toz}.c)\hskip-2.5pt:}\hskip13pt
\varGamma^k_{\nh\!ij}\,=\,0\hskip11pt\mathrm{whenever}\quad\hs i,j,k\hs\quad
\mathrm{are\ distinct}\hh.
\end{array}
\end{equation}
It is immediate from (\ref{aka}.a) with $\,\lx_3\w\lx\nh_4\w\ne0\,$ that
\begin{equation}\label{tra}
\mathrm{in\ case\ (\ref{toz}.a),}\hskip12pt\mathrm{tr}\nh_g\w\hh b\,
=\,2(\lambda_1\w+\lambda_2\w)\,=\,2(\lambda_3\w+\lambda_4\w)\hh.
\end{equation}
Here are some more consequences of 
(\ref{toz}.a). First, as $\,\lx_3\w\lx\nh_4\w\ne0$, we have
\begin{equation}\label{hth}
\bz_3\w\bz\nh_4\w\ne0\,\mathrm{\ everywhere\ in\ }\,\,U\nh,
\end{equation}
or else Remark~\ref{alleq} 
would make three of $\,\lambda_1\w,\lambda_2\w,\lambda_3\w,\lambda_4\w$ equal
to one another and, by (\ref{tra}), they would all be equal, contrary to the
assumption that $\,4b\ne(\mathrm{tr}\nh_g\w\hh b)g$ in Lemma~\ref{codwe}. We
may thus apply Lemma~\ref{ggnez}(i) to both $\,l\in\{3,4\}$, and then, in view
of (\ref{skw}.a),
\begin{equation}\label{lkl}
\lambda_k\w\ne\lambda_l\w\,\mathrm{\ for\
}\,k,l\in\{1,2,3,4\}\mathrm{,\ unless\ }\,k=l\,\mathrm{\ or\
}\,\{k,l\}=\{3,4\}\hh.
\end{equation}
Also, by (\ref{aka}.b) and (\ref{aea}), $\,\alpha_3\w=-\alpha_4\w=-\alpha\,$ 
if $\,\alpha=\alpha_4\w$ as in Lemma \ref{ggnez}, for $\,l=4$. Thus, 
(i) -- (x) in Lemma~\ref{ggnez} remain valid for $\,l=3$, provided that
$\,\alpha\,$ is everywhere replaced with $\,-\alpha\,$ and so, from 
Lemma~\ref{ggnez}(vii),
\begin{equation}\label{dia}
D\nnh_i\w\alpha\,=\,2\alpha\varGamma^{\hs i}_{\!kk}\mathrm{\ \ if\
}\,k\in\{3,4\}\,\mathrm{\ and\ }\,i\ne k\hh.
\end{equation}
Finally, due to (\ref{tra}), assertion (x) in 
Lemma~\ref{ggnez} with $\,\{i,l\}=\{3,4\}$ becomes 
$\,D\nnh_i\w(\lambda_i\w+\lambda_l\w)
=(\lambda_l\w-\lambda_i\w)\varGamma^{\hs i}_{\!ll}$ which, by (\ref{skw}.e), 
equals $\,D\nnh_i\w\lambda_l\w$, giving
\begin{equation}\label{dtl}
D\nh_3\w\lambda_3\w\hs=\,D\nnh_4\w\lambda_4\w\hs
=\,0\hskip22.2pt\mathrm{everywhere\ in\ }\,\,U,\\
\end{equation}

\section{Further consequences of {\rm(\ref{toz}.a)}}\label{co}
The following theorem is used in Section~\ref{nt} to show that,
in the case of harmonic curvature, 
$\,\text{\smallbf d}\ne2\,$ (cf.\ Section~\ref{sd}); it only
describes the curvature components needed for this purpose. The theorem is
valid in a more general situation, which is why, for the sake of 
completeness and a possible reference, the remaining components are
listed in the Appendix.
\begin{theorem}\label{cmwtw}Let tensor fields $\,A\,$ and $\,b\,$ on an 
oriented Riemannian four-man\-i\-fold $\,(M\nh,g)\,$ satisfy the hypotheses of
Lemma\/~{\rm\ref{codwe}} and\/ {\rm(\ref{toz}.a)}. If\/
$\,\mathrm{tr}\nh_g\w\hh b=0\,$ and\/ $\,\sigma\nnh_{1\nh2}\w\ne0\,$ everywhere
in\/ $\,M\nh$, then the following holds at each point for some function 
$\,\lambda\hh$, some constant\/ $\,\mu\hh$, and
$\,\alpha=\lx\nh_4\w/\nh\bz_4\w\,$ as in\/ {\rm(\ref{aea})}.
\begin{enumerate}
  \def\theenumi{{\rm\alph{enumi}}}
\item The eigen\-value functions $\,\lambda_i\w$ of $\,b\,$ and 
$\,\sigma\nnh_{ij}\w$ of $\,A\nh^\pm$ are given by 
$\,\lambda_1\w=\lambda=-\lambda_2\w$, $\,\lambda_3\w=\mu=-\lambda_4\w$ and\/ 
$\,3\sigma\nnh_{1\nh2}\w=3\sigma\nnh_{34}\w=-2\mu\alpha$, 
$\,3\sigma\nnh_{1\nh3}\w=3\sigma\nnh_{24}\w=(\mu+3\lambda)\alpha$, 
$\,3\sigma\nnh_{1\nh4}\w=3\sigma\nnh_{23}\w=(\mu-3\lambda)\alpha$.
\item $\sigma\nnh_{1\nh2}\w,\,\mu\,$ and\/ $\,\alpha\,$ are nonzero constants,
while\/
$\,\lambda_1\w\ne\lambda_2\w\ne\lambda_3\w\ne\lambda_4\w\ne\lambda_1\w$.
\item 
$\nabla\hskip-3.7pt_{e\nh_3\w}\hskip-3.2pte\nh_3\w\nh
=\nabla\hskip-3.7pt_{e\nh_3\w}\hskip-3.2pte\nh_4\w\nh=
\nabla\hskip-3.7pt_{e\nh_4\w}\hskip-3.2pte\nh_3\w\nh
=\nabla\hskip-3.7pt_{e\nh_4\w}\hskip-3.2pte\nh_4\w\nh
=0\,$ and\/ $\,\,R\hh(e\nh_3\w,e\nh_4\w)e\nh_3\w\nh
=R\hh(e\nh_3\w,e\nh_4\w)e\nh_4\w\nh=0$.
\item Whenever\/ $\,\{i,j\}=\{1,2\}\,$ and\/ $\,\{k,l\}=\{3,4\}$, one has
\[
\begin{array}{l}
R_{ikik}\w\,
=\,\displaystyle{\frac{D\nnh_k\w D\nnh_k\w\lambda_i\w}{\lambda_i\w
-\lambda_k\w}\,-
\,\frac{2(D\nnh_k\w\lambda_i\w)^2}{(\lambda_i\w-\lambda_k\w)^2}\,
+\,\frac{\bz_l^2}{\lambda_i\w(\lambda_i\w+\lambda_k\w)}}\hh,\\
(\lambda_i\w+\lambda_k\w)R_{ikil}\w
=D\nnh_k\w D\nh_l\w\lambda_i\w-
\,\displaystyle{\frac{2\lambda_i\w(D\nnh_k\w\lambda_i\w)D\nh_l\w\lambda_i\w}{\lambda_i^2-\lambda_k^2}}
+\,\displaystyle{\frac{(\lambda_i^2\nh+\lambda_k^2)\bz_k\w\bz_l\w}{\lambda_i\w(\lambda_i^2-\lambda_k^2)}}\hh,\\
{}[\hs e\nh_i\w,e\nh_j\w]=
-\displaystyle{\frac{D\nnh_j\w\lambda_i\w}{2\lambda_i\w}}\,e\nh_i\w
+\displaystyle{\frac{D\nnh_i\w\lambda_i\w}{2\lambda_i\w}}\,e\nh_j\w
+\displaystyle{\frac{2\lambda_i\w}{\lambda_k^2-\lambda_i^2}}\hs
(\bz_l\w e\nh_k\w+\bz_k\w e\nh_l\w)\hh,\\
(\lambda_i\w-\lambda_k\w)R_{ijik}\w
=D\nnh_j\w D\nnh_k\w\lambda_i\w-D\nnh_i\w\bz_l\w
+(\lambda_i^2\nh+2\lambda_i\w\lambda_k\w-\lambda_k^2)
\displaystyle{\frac{\bz_l\w D\nnh_i\w\lambda_i\w-(D\nnh_k\w\lambda_i\w)D\nnh_j\w\lambda_i\w}{\lambda_i\w(\lambda_i^2-\lambda_k^2)}}\\
(\lambda_k^2\nh-\lambda_i^2)R_{ijkl}\w\,
=\,2(\bz_k\w D\nnh_k\w\lambda_i\w\,
-\,\bz_l\w D\nh_l\w\lambda_i\w)\hh,\phantom{_{j_{_j}}}
\end{array}
\]
with\/ $\,\bz_k\w,\bz_l\w,R_{ijkl}\w$ as in\/ {\rm(\ref{ale})}, 
{\rm(\ref{rkl})} and\/ {\rm(\ref{rvw})}, and\/
$\,\lambda_i\w(\lambda_i^2-\lambda_k^2)\ne0$.
\end{enumerate}
\end{theorem} 
\begin{proof}As $\,\,\mathrm{tr}\nh_g\w\hh b=0$, (\ref{tra}) means nothing 
else than 
\begin{equation}\label{lot}
\lambda_1\w+\lambda_2\w\,=\,\lambda_3\w+\lambda_4\w\,=\,0\hh.
\end{equation}
Our assumption (\ref{toz}.a) leads to (\ref{hth}), and so
$\,\lx_3\w\lx\nh_4\w\bz_3\w\bz_4\w\ne0\,$ everywhere. We may thus use 
Lemma~\ref{ggnez}(xii) which, with $\,\hs\mathrm{tr}\nh_g\w\hh b=0$, yields 
$\,\lambda_1\w\varGamma^1_{\!33}=\lambda_2\w\varGamma^2_{\!33}=
\lambda_1\w\varGamma^1_{\!44}=\lambda_2\w\varGamma^2_{\!44}=0$. Since 
$\,\lambda_1\w\lambda_2\w\ne0\,$ (for otherwise (\ref{lot}) would give 
$\,\lambda_1\w=\lambda_2\w=0$, contradicting (\ref{lkl})), we thus get 
$\,\varGamma^{\hs i}_{\!33}=\varGamma^{\hs i}_{\!44}=0\,$ for $\,i=1,2$. On
the other hand, by (\ref{dtl}), (\ref{lot}) and (\ref{skw}.e), 
$\,0=-D\nnh_k\w\lambda_k\w=D\nnh_k\w\lambda_l\w=
(\lambda_l\w-\lambda_k\w)\varGamma^k_{\!ll}$ whenever 
$\,\{k,l\}=\{3,4\}$. Therefore, $\,\varGamma^3_{\!44}=
\varGamma^{\hs4}_{\!33}=0$,
or else we would have $\,\lambda_3\w=\lambda_4\w$ and, by (\ref{lot}),
$\,\lambda_3\w=\lambda_4\w=0\,$ which, due to Lemma~\ref{ggnez}(ii) and
(\ref{lot}), would yield $\,\sigma\nnh_{1\nh2}\w=0$, even though we 
assumed that $\,\sigma\nnh_{1\nh2}\w\ne0$. Finally, (\ref{abc}) implies that 
$\,\varGamma^{\hs i}_{\!34}=\varGamma^{\hs i}_{\!43}=0\,$ for $\,i=1,2$.
Consequently, by (\ref{skw}.a),
\begin{equation}\label{itt}
\begin{array}{rl}
\mathrm{i)}&\varGamma^{\hs i}_{\!33}\,=\,\varGamma^{\hs i}_{\!44}\,=\,
\varGamma^{\hs i}_{\!34}\,=\,\varGamma^{\hs i}_{\!43}\,=\,0\quad\mathrm{for\
all}\quad i=1,2,3,4\hh,\\
\mathrm{ii)}&\varGamma^l_{\nnh\!ik}\,
=\,\hs\varGamma^i_{\!kl}\,
=\,\hs\varGamma^l_{\!ki}\,=\,0\quad\mathrm{if}\quad k,l\in\{3,4\}\hh.
\end{array}
\end{equation}
In view of (\ref{gij}) and (\ref{rvw}), relations (\ref{itt}.i) prove (c).
Also, using (\ref{skw}.e), (\ref{lot}), 
(\ref{itt}) and (\ref{dtl}), we easily conclude that
$\,\lambda_3\w$ and $\,\lambda_4\w$ are constant. By (\ref{lot}), this yields
(a), with the function $\,\lambda=\lambda_1\w$ and constant
$\,\mu=\lambda_3\w$, as one sees evaluating
$\,3\sigma\nnh_{1\nh2}\w,3\sigma\nnh_{1\nh3}\w,3\sigma\nnh_{1\nh4}\w$ from
Lemma~\ref{ggnez}(ii) applied to $\,i,j,k,l\,$ such that
$\,\{i,j,k\}=\{1,2,3\}$, while $\,l=4$, 
and then invoking (\ref{skw}.b). Now (b) follows from
(\ref{lkl}), (\ref{dia}), (\ref{itt}) and (a), as $\,\sigma\nnh_{1\nh2}\w\ne0$.

Next, let us fix $\,i,j,k,l\,$ with 
$\,\{i,j\}=\{1,2\}\,$ and $\,\{k,l\}=\{3,4\}$.

Due to (\ref{gij}), (\ref{itt}.ii) and (\ref{ale}), with
$\,\lambda_1\w\ne\lambda_2\w$ by (\ref{lkl}),
\begin{equation}\label{gni}
\begin{array}{l}
\nabla\hskip-3.7pt_{e\nh_k\w}\hskip-3.2pte\nh_i\w=
\varGamma^j_{\!ki}e\nh_j\w=(\lambda_i\w-\lambda_j\w)^{-1}\bz_l\w e\nh_j\w\hh,
\qquad\nabla\hskip-3.7pt_{e\nh_i\w}\hskip-3.2pte\nh_k\w\nh
=-\varGamma^k_{\nnh\!ii}e\nh_i\w+\varGamma^j_{\!ik}e\nh_j\w\hh,\\
g(\nabla\hskip-3.7pt_{e\nh_i\w}\hskip-3pt\nabla\hskip-3.7pt_{e\nh_k\w}\hskip-3.2pte\nh_i\w,\,e\nh_k\w)=(\lambda_i\w
-\lambda_j\w)^{-1}\bz_l\w\varGamma^k_{\nh\!ij}=
(\lambda_i\w-\lambda_j\w)^{-1}(\lambda_j\w-\lambda_k\w)^{-1}\bz_l^2,\\
g(\nabla\hskip-3.7pt_{e\nh_i\w}\hskip-3pt\nabla\hskip-3.7pt_{e\nh_k\w}\hskip-3.2pte\nh_i\w,\,e\nh_l\w)=(\lambda_i\w
-\lambda_j\w)^{-1}\bz_l\w\varGamma^l_{\nnh\!ij}=
(\lambda_i\w-\lambda_j\w)^{-1}(\lambda_j\w-\lambda_l\w)^{-1}\bz_k\w\bz_l\w\hh.
\end{array}
\end{equation}
From (\ref{gij}), with 
$\,\nabla\hskip-3.7pt_{e\nh_k\w}\hskip-3.2pte\nh_k\w=
\nabla\hskip-3.7pt_{e\nh_k\w}\hskip-3.2pte\nh_l\w=0\,$ in (c), 
$\,g(\nabla\hskip-3.7pt_{e\nh_k\w}\hskip-4pt
\nabla\hskip-3.7pt_{e\nh_i\w}\hskip-3.2pt
e\nh_i\w,\,e\nh_k\w)=D\nnh_k\w[g(\nabla\hskip-3.7pt_{e\nh_i\w}\hskip-3.2pt
e\nh_i\w,\,e\nh_k\w)]=D\nnh_k\w\varGamma^k_{\nnh\!ii}$ and 
$\,g(\nabla\hskip-3.7pt_{e\nh_k\w}\hskip-4pt
\nabla\hskip-3.7pt_{e\nh_i\w}\hskip-3.2pt
e\nh_i\w,\,e\nh_l\w)=D\nnh_k\w[g(\nabla\hskip-3.7pt_{e\nh_i\w}\hskip-3.2pt
e\nh_i\w,\,e\nh_l\w)]=D\nnh_k\w\varGamma^l_{\nnh\!ii}$. Since (\ref{skw}.e)
and (\ref{lkl}) give 
$\,\varGamma^k_{\nnh\!ii}
=(\lambda_i\w-\lambda_k\w)^{-1}D\nnh_k\w\lambda_i\w$ and
$\,\varGamma^l_{\nnh\!ii}
=(\lambda_i\w-\lambda_l\w)^{-1}D\nh_l\w\lambda_i\w$, we get 
$\,g(\nabla\hskip-3.7pt_{e\nh_k\w}\hskip-4pt\nabla\hskip-3.7pt_{e\nh_i\w}
\hskip-3.2pte\nh_i\w,\,e\nh_k\w)=
D\nnh_k\w[(\lambda_i\w-\lambda_k\w)^{-1}D\nnh_k\w\lambda_i\w]\,$ and
$\,g(\nabla\hskip-3.7pt_{e\nh_k\w}\hskip-4pt\nabla\hskip-3.7pt_{e\nh_i\w}
\hskip-3.2pte\nh_i\w,\,e\nh_l\w)=
D\nnh_k\w[(\lambda_i\w-\lambda_l\w)^{-1}D\nnh_l\w\lambda_i\w]$. Thus,
with $\,\lambda_k\w$ and $\,\lambda_l\w$ both constant, cf.\ (b),
\begin{equation}\label{gnk}
\begin{array}{l}
g(\nabla\hskip-3.7pt_{e\nh_k\w}\hskip-4pt\nabla\hskip-3.7pt_{e\nh_i\w}
\hskip-3.2pte\nh_i\w,\,e\nh_k\w)=
(\lambda_i\w-\lambda_k\w)^{-1}D\nnh_k\w D\nnh_k\w\lambda_i\w
-(\lambda_i\w-\lambda_k\w)^{-2}(D\nnh_k\w\lambda_i\w)^2,\\
g(\nabla\hskip-3.7pt_{e\nh_k\w}\hskip-4pt\nabla\hskip-3.7pt_{e\nh_i\w}
\hskip-3.2pte\nh_i\w,\,e\nh_l\w)=
(\lambda_i\w-\lambda_l\w)^{-1}D\nnh_k\w D\nnh_l\w\lambda_i\w
-(\lambda_i\w-\lambda_l\w)^{-2}(D\nnh_k\w\lambda_i\w)D\nnh_l\w\lambda_i\w\hh.
\end{array}
\end{equation}
By the first line of (\ref{gni}), $\,[\hs e\hs\nh_i\w,e\nh_k\w]
=\nabla\hskip-3.7pt_{e\nh_i\w}\hskip-3.2pte\nh_k\w\nnh
-\nabla\hskip-3.7pt_{e\nh_k\w}\hskip-3.2pte\nh_i\w=
-\varGamma^k_{\nnh\!ii}e\nh_i\w
+(\varGamma^j_{\!ik}-\varGamma^j_{\!ki})e\nh_j\w$. Now (\ref{gij}) yields 
$\,g(\nabla\!_{[\hs e\hs\nh_i\w,e\nh_k\w]}e\nh_i\w,e\nh_k\w)=
-(\varGamma^k_{\nnh\!ii})^2\nh+(\varGamma^j_{\!ik}-\varGamma^j_{\!ki})
\varGamma^k_{\!ji}$ and 
$\,g(\nabla\!_{[\hs e\hs\nh_i\w,e\nh_k\w]}e\nh_i\w,e\nh_l\w)=
-\varGamma^k_{\nnh\!ii}\varGamma^l_{\nnh\!ii}
+(\varGamma^j_{\!ik}-\varGamma^j_{\!ki})\varGamma^l_{\!ji}$. Consequently,
from (\ref{skw}.e) and (\ref{ale}),
\begin{equation}\label{gik}
\begin{array}{l}
g(\nabla\!_{[\hs e\hs\nh_i\w,e\nh_k\w]}e\nh_i\w,e\nh_k\w)
=-(\lambda_i\w-\lambda_k\w)^{-2}(D\nnh_k\w\lambda_i\w)^2\\
\phantom{g(\nabla\!_{[\hs e\hs\nh_i\w,e\nh_k\w]}e\nh_i\w,e\nh_k\w)}
+\,(\lambda_k\w-\lambda_i\w)^{-1}[(\lambda_i\w-\lambda_j\w)^{-1}\nh
-(\lambda_k\w-\lambda_j\w)^{-1}]\bz_l^2,\\
g(\nabla\!_{[\hs e\hs\nh_i\w,e\nh_k\w]}e\nh_i\w,e\nh_l\w)
=-(\lambda_i\w-\lambda_k\w)(\lambda_i\w-\lambda_l\w)(D\nnh_k\w\lambda_i\w)
D\nnh_l\w\lambda_i\w\\
\phantom{g(\nabla\!_{[\hs e\hs\nh_i\w,e\nh_k\w]}e\nh_i\w,e\nh_k\w)}
+\,(\lambda_l\w-\lambda_i\w)^{-1}[(\lambda_i\w-\lambda_j\w)^{-1}\nh
-(\lambda_k\w-\lambda_j\w)^{-1}]\bz_k\w\bz_l\w.
\end{array}
\end{equation}
Relations (\ref{gni}) -- (\ref{gik}), (\ref{rvw}), (\ref{rkl}) and
(\ref{lot}) prove the first two lines of (d). Also, 
$\,\nabla\hskip-3.7pt_{e\nh_j\w}\hskip-3.2pte\nh_i\w
=-\varGamma^i_{\nnh\!j\hn j}e\nh_j\w
+\varGamma^k_{\!ji}e\nh_k\w+\varGamma^l_{\!ji}e\nh_l\w$, from (\ref{gij}) and
(\ref{skw}.a), so that 
\begin{equation}\label{nji}
\nabla\hskip-3.7pt_{e\nh_j\w}\hskip-3.2pte\nh_i\w
=(\lambda_i\w-\lambda_j\w)^{-1}(D\nnh_i\w\lambda_j\w)e\nh_j\w
+(\lambda_i\w-\lambda_k\w)^{-1}\bz_l\w e\nh_k\w
+(\lambda_i\w-\lambda_l\w)^{-1}\bz_k\w e\nh_l\w
\end{equation}
by (\ref{skw}.e) and (\ref{ale}). The third line of (d) now follows if one
switches $\,i,j\,$ in (\ref{nji}), subtracts, and uses (b) to replace
$\,\lambda_j\w,\lambda_l\w$ with $\,-\lambda_i\w,-\lambda_k\w$. On the other
hand, (\ref{gij}) yields 
$\,\nabla\hskip-3.7pt_{e\nh_i\w}\hskip-3.2pte\nh_i\w\nh
=\varGamma^j_{\nnh\!ii}e\nh_j\w+\varGamma^k_{\!ii}e\nh_k\w
+\varGamma^l_{\!ii}e\nh_l\w$, and so, from (\ref{itt}.ii), 
$\,g(\nabla\hskip-3.7pt_{e\nh_j\w}\hskip-4pt
\nabla\hskip-3.7pt_{e\nh_i\w}\hskip-3.2pt
e\nh_i\w,\,e\nh_k\w)=\varGamma^j_{\nnh\!ii}\nnh\varGamma^k_{\nnh\!j\hn j}
+D\nnh_j\w\varGamma^k_{\!ii}$. Similarly,
$\,\nabla\hskip-3.7pt_{e\nh_j\w}\hskip-3.2pte\nh_i\w
=-\varGamma^i_{\nnh\!j\hn j}e\nh_j\w
+\varGamma^k_{\!ji}e\nh_k\w+\varGamma^l_{\!ji}e\nh_l\w$, in the line
preceding (\ref{nji}), and hence 
$\,g(\nabla\hskip-3.7pt_{e\nh_i\w}\hskip-2.7pt
\nabla\hskip-3.7pt_{e\nh_j\w}\hskip-3.2pt
e\nh_i\w,\,e\nh_k\w)=-\varGamma^i_{\nnh\!j\hn j}\varGamma^k_{\nh\!ij}
+D\nnh_i\w\varGamma^k_{\!ji}$, while the third line of (d) and (\ref{itt}.ii) 
give 
$\,2\lambda_i\w\hs g(\nabla\!_{[e\hs\nh_i\w,e\nh_j\w]}e\nh_i\w,e\nh_k\w)
=-\varGamma^k_{\!ii}D\nnh_j\w\lambda_i\w
+\varGamma^k_{\!ji}D\nnh_i\w\lambda_i\w$. Replacing
$\,\varGamma^j_{\nnh\!ii},\,\varGamma^k_{\nnh\!j\hn j},\,
\varGamma^k_{\!ii},\,\varGamma^i_{\nnh\!j\hn j},\,\varGamma^k_{\nh\!ij},\,
\varGamma^k_{\!ji}$, in the
right-hand sides of the three just-de\-riv\-ed relations of the form
$\,g(\nabla\nnh...\,,e\nh_k\w)=\,...\,$, with the expressions provided by
(\ref{skw}.e) and the second equality of (\ref{ale}), and noting that, in (a), 
$\,\lambda_j\w=-\lambda_i\w$ and $\,\lambda_l\w=-\lambda_k\w$ is constatnt, we
obtain the fourth line of (d).
Finally, $\,\nabla\hskip-3.7pt_{e\nh_i\w}\hskip-3.2pte\nh_k\w\nh
=-\varGamma^k_{\nnh\!ii}e\nh_i\w+\varGamma^j_{\!ik}e\nh_j\w$ in (\ref{gni}).
Thus, 
$\,g(\nabla\hskip-3.7pt_{e\nh_j\w}\hskip-4pt
\nabla\hskip-3.7pt_{e\nh_i\w}\hskip-3.2pt
e\nh_k\w,\,e\nh_l\w)=-\varGamma^k_{\nnh\!ii}\varGamma^l_{\nnh\!ji}
+\varGamma^j_{\!ik}\varGamma^i_{\nnh\!j\hn j}$ and, with $\,i,j\,$
switched, $\,g(\nabla\hskip-3.7pt_{e\nh_i\w}\hskip-2.7pt
\nabla\hskip-3.7pt_{e\nh_j\w}\hskip-3.2pt
e\nh_k\w,\,e\nh_l\w)=-\varGamma^k_{\nnh\!j\hn j}\varGamma^l_{\nnh\!ij}
+\varGamma^i_{\!jk}\varGamma^j_{\nnh\!ii}$, while
$\,g(\nabla\!_{[e\hs\nh_i\w,e\nh_j\w]}e\nh_k\w,e\nh_l\w)=0\,$ by
(\ref{itt}.ii). The last line of (d) now follows  
if one replaces 
$\,\varGamma^k_{\nnh\!ii},\varGamma^l_{\nnh\!ji},\varGamma^j_{\!ik},
\varGamma^i_{\nnh\!j\hn j},\varGamma^k_{\nnh\!j\hn j},
\varGamma^l_{\nnh\!ij},\varGamma^i_{\!jk},\varGamma^j_{\nnh\!ii}$ as
before, using (\ref{skw}.e) and (\ref{ale}).
\end{proof}

\section{Exclusion of case {\rm(\ref{toz}.a)} when
$\,\,\mathrm{div}\,R=0$}\label{nt}
We now proceed to derive a contradiction from the assumption that (\ref{toz}.a)
holds and $\,(A,b)=(W\hskip-2.3pt,\,\mathrm{Ric}-\mathrm{s}\hskip.7ptg/4)$. 
We are allowed to invoke Theorem \ref{cmwtw}, since
$\,\sigma\nnh_{1\nh2}\w\ne0\,$ as a consequence of (\ref{snz}).

First, $\,\mathrm{s}=8\mu\alpha$, from Theorem \ref{cmwtw}(a) and
(\ref{jij}) for $\,(i,j)=(3,4)$, where $\,R_{ijij}\w=0\,$ in (\ref{jij})
by Theorem \ref{cmwtw}(c). Thus, by (\ref{jij}) and Theorem~\ref{cmwtw}(a),
\begin{equation}\label{chc}
\begin{array}{l}
R_{1212}\w=0\hh,\hskip5ptR_{1313}\w=(\alpha+1/2)(\mu+\lambda)\hh,
R_{1414}\w=(\alpha-1/2)(\mu-\lambda)\hh,\\
R_{2323}\w=(\alpha+1/2)(\mu-\lambda)\hh,
R_{2424}\w=(\alpha-1/2)(\mu+\lambda)\hh,\hskip5ptR_{3434}\w=0\hh.
\end{array}
\end{equation}
Choosing $\,(i,k)\,$ in the first equality of Theorem~\ref{cmwtw}(d) to be
$\,(1,3)$, $\,(2,3)$, $\,(2,4)\,$ or, respectively, $\,(1,4)$, we get
\begin{equation}\label{dtd}
\begin{array}{l}
\displaystyle{\frac{D\nnh_3\w D\nnh_3\w\lambda}{\lambda-\mu}\,-
\,\frac{2(D\nnh_3\w\lambda)^2}{(\lambda-\mu)^2}\,
+\,\frac{\bz_4^2}{\lambda(\lambda+\mu)}}\,=\,(\alpha+1/2)(\mu+\lambda)\hh,\\
\displaystyle{\frac{D\nnh_3\w D\nnh_3\w\lambda}{\lambda+\mu}\,-
\,\frac{2(D\nnh_3\w\lambda)^2}{(\lambda+\mu)^2}\,
+\,\frac{\bz_4^2}{\lambda(\lambda-\mu)}}\,=\,(\alpha+1/2)(\mu-\lambda)\hh,\\
\displaystyle{\frac{D\nnh_4\w D\nnh_4\w\lambda}{\lambda-\mu}\,-
\,\frac{2(D\nnh_4\w\lambda)^2}{(\lambda-\mu)^2}\,
+\,\frac{\bz_3^2}{\lambda(\lambda+\mu)}}\,=\,(\alpha-1/2)(\mu+\lambda)\hh,\\
\displaystyle{\frac{D\nnh_4\w D\nnh_4\w\lambda}{\lambda
+\mu}\,-
\,\frac{2(D\nnh_4\w\lambda)^2}{(\lambda+\mu)^2}\,
+\,\frac{\bz_3^2}{\lambda(\lambda-\mu)}}\,=\,(\alpha-1/2)(\mu-\lambda)\hh.
\end{array}
\end{equation}
The linear combinations of the first (or, last) two lines of (\ref{dtd}) with
the coefficients $\,\mu-\lambda\,$ and $\,\mu+\lambda\,$ yield
\begin{equation}\label{fmh}
\begin{array}{l}
4\mu\hh[(D\nnh_3\w\lambda)^2\nh+\bz_4^2]
=-(2\alpha+1)(\lambda^2\nh-\mu^2)^2,\phantom{_{j_j}}\\
4\mu\hh[(D\nnh_4\w\lambda)^2\nh+\bz_3^2]=-(2\alpha-1)(\lambda^2\nh-\mu^2)^2.
\end{array}
\end{equation}
On the other hand, Theorem~\ref{cmwtw}(a)\hs-\hs(b) and (\ref{hth}) imply that
\begin{equation}\label{nnz}
\lambda\mu(\lambda^2\nh-\mu^2)\bz_3\w\bz_4\w\ne0\,\mathrm{\ \ everywhere.}
\end{equation}
The last equality of Theorem~\ref{cmwtw}(d), for 
$\,(i,j,k,l)=(1,2,3,4)$, combined with (\ref{riz}) and (\ref{nnz}), gives 
$\,\bz_4\w D\nnh_4\w\lambda=\bz_3\w D\nnh_3\w\lambda$. 
Thus, at every point, the vectors $\,(D\nnh_4\w\lambda,\hs\bz_3\w)\,$ and 
$\,(D\nnh_3\w\lambda,\hs\bz_4\w)$, both nonzero due to (\ref{nnz}), are
linearly dependent in $\,\rto\nh$, and so
$\,(D\nnh_4\w\lambda,\bz_3\w)=(qD\nnh_3\w\lambda,q\bz_4\w)\,$ for some
function $\,q\,$ without zeros. Now, by (\ref{fmh}), 
$\,(2\alpha-1)=(2\alpha+1)\hh q^2$ and so, as both sides of both equalities
in (\ref{fmh}) are nonzero, cf.\ (\ref{nnz}),
\begin{equation}\label{qnz}
q\,\mathrm{\ is\ a\ nonzero\ constant}
\end{equation}
in view of Theorem~\ref{cmwtw}(b). However, according to
Theorem~\ref{cmwtw}(a),
\begin{equation}\label{lmu}
(\lambda_1\w,\lambda_2\w,\lambda_3\w,\lambda_4\w)\,
=\,(\lambda,-\lambda,\mu,-\mu)\hh.
\end{equation}
Therefore, the fourth equality of Theorem~\ref{cmwtw}(d), with the left-hand
side equal to
$\,0\,$ by (\ref{riz}), applied to fixed $\,i,j\,$ with $\,\{i,j\}=\{1,2\}\,$
and $\,(k,l)$ replaced by $\,(3,4)\,$ or, respectively, $\,(4,3)$, reads
\[
\begin{array}{l}
D\nnh_j\w D\nnh_3\w\lambda_i\w-D\nnh_i\w\bz_4\w
+(\lambda^2\nh+2\lambda_i\w\mu-\mu^2)
\displaystyle{\frac{\bz_4\w D\nnh_i\w\lambda_i\w-(D\nnh_3\w\lambda_i\w)D\nnh_j\w\lambda_i\w}{\lambda_i\w(\lambda^2-\mu^2)}}=0\hh,\\
D\nnh_j\w D\nnh_4\w\lambda_i\w-D\nnh_i\w\bz_3\w
+(\lambda^2\nh-2\lambda_i\w\mu-\mu^2)
\displaystyle{\frac{\bz_3\w D\nnh_i\w\lambda_i\w-(D\nnh_4\w\lambda_i\w)D\nnh_j\w\lambda_i\w}{\lambda_i\w(\lambda^2-\mu^2)}}=0\hh.
\end{array}
\]
Replacing the pair 
$\,(D\nnh_4\w\lambda_i\w,\bz_3\w)\,$ in the second equality above
by $\,(qD\nnh_3\w\lambda_i\w,q\bz_4\w)$ and subtracting the result from
the first equality multiplied by $\,q\,$ we obtain
\begin{equation}\label{obt}
4q\lambda_i\w\mu\hs[\bz_4\w D\nnh_i\w\lambda_i\w
-(D\nnh_3\w\lambda_i\w)D\nnh_j\w\lambda_i\w]\,=\,0\quad\mathrm{whenever\
}\,\{i,j\}=\{1,2\}\hh,
\end{equation}
due to constancy of $\,q\,$ established in (\ref{qnz}). We have two matrix
equalities
\begin{equation}\label{meq}
\mathrm{a)}\hskip4pt
\left[\begin{matrix}
D\nnh_3\w\lambda&\bz_4\w\cr
-\bz_4\w&D\nnh_3\w\lambda\end{matrix}\right]
\left[\begin{matrix}
D\nnh_1\w\lambda\cr
D\nnh_2\w\lambda\end{matrix}\right]
=\left[\begin{matrix}
0\cr
0\end{matrix}\right]\nnh,\quad
\mathrm{b)}\hskip4pt
\left[\begin{matrix}
\bz_3\w&-\bz_4\w\cr
\bz_4\w&\bz_3\w\end{matrix}\right]
\left[\begin{matrix}
D\nnh_3\w\lambda\cr
D\nnh_4\w\lambda\end{matrix}\right]
=\left[\begin{matrix}
0\cr
0\end{matrix}\right]\nnh,
\end{equation}
with {\it both determinants nonzero in view of\/} (\ref{nnz}). 
Namely, (\ref{meq}.a) follows if one lets $\,(i,j)\,$ be $\,(1,2)\,$ or 
$\,(2,1)\,$ in (\ref{obt}), and uses (\ref{nnz}) -- (\ref{lmu}). By
(\ref{meq}.a), $\,D\nnh_1\w\lambda=D\nnh_2\w\lambda=0$. Thus,   
the third equality of Theorem~\ref{cmwtw}(d) and (\ref{nnz}) give 
$\,\bz_4\w D\nnh_3\w\lambda+\bz_3\w D\nnh_4\w\lambda=0$. Now (\ref{meq}.b)
follows: $\,\bz_4\w D\nnh_4\w\lambda=\bz_3\w D\nnh_3\w\lambda$, as we
saw in the second line after (\ref{nnz}). 

Both determinants in (\ref{meq}.b) being nonzero, we conclude that
$\,\lambda\,$ is constant, which in turn contradicts the second equality of
Theorem~\ref{cmwtw}(d), since 
$\,\bz_3\w\bz\nh_4\w\ne0=R_{ikil}\w$ by (\ref{nnz}) and (\ref{riz}).

\section{Case {\rm(\ref{toz}.b)}}\label{cc}
In this section we show that four-man\-i\-folds with harmonic curvature,
having $\,\text{\smallbf d}=1\,$ (see Section~\ref{sd}) are locally of type 
{\rm(\ref{kne}.c)}.

For the Le\-vi-Ci\-vi\-ta connections $\,\nabla\,$ and $\,\hat\nabla\,$ of 
con\-for\-mal\-ly related metrics $\,g\,$ and $\,\hg=|\beta|^{1/2}g\,$ on 
a manifold, with a no\-where-zero function $\,\beta$, one has
\begin{equation}\label{tnv}
\hat\nabla\!_{v}\w w\,=\,\nabla\!_{v}\w w\,+
\,\frac{d_v\w\beta}{4\beta}\,w\,+\,\frac{d_w\w\beta}{4\beta}\,v\,
-\,g(v,w)\hs\frac{\nabla\!\beta}{4\beta}\hh,
\end{equation}
cf.\ \cite[p. 58]{besse}, $\,v,w,\hs d_v\w$ and $\,\nabla\!\beta\,$ being,
respectively, any two vector fields, the $\,v$-di\-rec\-tion\-al derivative,
and the $\,g$-grad\-i\-ent of $\,\beta$.
\begin{lemma}\label{neffz}Let\/ $A,b,(M\nh,g)$ satisfy the
assumptions of Lemma\/~{\rm\ref{codwe}} and\/ {\rm(\ref{toz}.b)}. If 
$\,\,\mathrm{tr}\nh_g\w\hh b=0$, then there exists a dense open set\/
$\,\,U\nh\subseteq\nh M\,$ such that, on every\/ connected component\/
$\,\,U'$ of\/ $\,\,U\nh$, with some function\/
$\,\beta:U'\to\bbR\smallsetminus\{0\}$, and\/ $\,e\nh_i\w$ as in 
Lemma\/~{\rm\ref{codwe}(i)},
\begin{equation}\label{nef}
\nabla\hskip-3.7pt_{e\nh_4\w}\hskip-3.2pte\nh_4\w=0\hh,\hskip7pt
4\beta\hh\nabla\hskip-3.7pt_{e\nh_i\w}\hskip-3.2pte\nh_4\w
=-\hs(D\nnh_4\w\beta)\hs e\nh_i\w\hh,\hskip7ptD\nnh_i\w\beta
=0\hskip7pt\mathrm{for}\hskip6pti=1,2,3\hh.
\end{equation}
\end{lemma}
\begin{proof}It suffices to exhibit $\,\beta:U'\to\bbR\smallsetminus\{0\}\,$
having
\begin{equation}\label{iff}
\mathrm{i)}\hskip5pt\varGamma^{\hs i}_{\!44}=D\nnh_i\w\beta
=0\hh,\quad\mathrm{ii)}\hskip5pt4\beta\hn\varGamma^{\hs i}_{\!i4}
=-D\nnh_4\w\beta\quad\mathrm{whenever\ \ }i\in\{1,2,3\}\hh.
\end{equation}
Namely, the last equality in (\ref{nef}) follows from 
(\ref{iff}.i), the first two -- from (\ref{iff}) and 
(\ref{gij}), as $\,\varGamma^{\hs4}_{\!44}=0\,$ by (\ref{skw}.a), while 
$\,\varGamma^{\hs j}_{\!i4}=0\,$ whenever 
$\{i,j,k\}=\{1,2,3\}$ due to (\ref{toz}.b) and (\ref{abc}).

Let $\,\,U\subseteq M\,$ be the open dense set of points $\,x\,$ such that
$\,\bz_4\w(x)\ne0\,$ or $\,\bz_4\w=0\,$ on a neighborhood of $\,x$. For any
connected component $\,\,U'$ of $\,\,U\nh$, one of the following two
conditions holds throughout $\,\,U'$ (see Remark~\ref{alleq}):
\begin{enumerate}
  \def\theenumi{{\rm\alph{enumi}}}
\item $\lambda_1\w\nh=\lambda_2\w\nh=\lambda_3\w$ and $\,\bz_4\w=0$,
for the eigen\-value functions $\,\lambda_i\w$ of $\,b$,
\item $\lambda_1\w\nh\ne\lambda_2\w\nh\ne\lambda_3\w\nh
\ne\lambda_1\w$ and $\,\bz_4\w\ne0$.
\end{enumerate}
We thus need to show that either of (a) -- (b) 
implies (\ref{iff}) on $\,\,U'\nh$.

First, in case (a), since $\,\mathrm{tr}\nh_g\w\hh b=0$, 
setting $\,\beta=\lambda_i\w$, $\,i=1,2,3$, we get 
$\,\lambda_4\w=-3\beta$. Also, $\,\beta\ne0\,$ everywhere, for otherwise we
would have $\,b=0$, even though Lemma~\ref{codwe} assumes that 
$\,4b\ne(\mathrm{tr}\nh_g\w\hh b)g$. Thus, from (\ref{skw}.e),
$\,D\nnh_i\w\beta=0\,$ and $\,D\nnh_i\w\lambda_4\w=0$. However, (\ref{skw}.e)
with $\,D\nnh_i\w\lambda_4\w=0\,$ gives $\,-4\beta\varGamma^i_{\!44}=0$,
$\,i=1,2,3$, proving (\ref{iff}.i). Finally,
(\ref{skw}.e) with $\,i=4$, any $\,j\in\{1,2,3\}$, and 
$\,\lambda_4\w=-3\beta\,$ reads 
$\,D\nnh_4\w\beta=4\beta\varGamma^4_{\!j\hn j}$, and (\ref{skw}.a) yields
(\ref{iff}.ii).

Suppose now that condition (b) holds. Note that
\begin{equation}\label{rof}
\mathrm{i)}\hskip6ptR_{1424}\w=0\hh,\hskip12pt
\mathrm{ii)}\hskip6ptR_{ij4k}\w=0\hskip7pt
\mathrm{whenever}\hskip8pt\{i,j,k\}=\{1,2,3\}
\end{equation}
with the notation of (\ref{rkl}). Namely, Lemma \ref{rvovt} and (b) give
$\,R\hh(e\nh_i\w,e\nh_j\w)e\nh_k\w=0$, as well as 
$\,R\hh(e\nh_1\w,e\nh_4\w)e\nh_2\w=0\,$ (or, 
$\,R\hh(e\nh_2\w,e\nh_4\w)e\nh_1\w=0$) at points where 
$\,\lambda_2\w\ne\lambda_4\w$ (or, respectively,
$\,\lambda_1\w\ne\lambda_4\w$), so that the usual symmetries of $\,R\,$
imply (\ref{rof}).

Let us now fix $\,x\in U'\nh$. By 
(b), at least two of $\,\lambda_1\w(x),\lambda_2\w(x),\lambda_3\w(x)\,$ are 
nonzero. Rearranging $\,e\nh_1\w,e\nh_2\w,e\nh_3\w$, we may assume that 
$\,\lambda_1\w\lambda_2\w\ne0$ at $\,x$. Then, from Lemma~\ref{ggnez}(xii)
with $\,l=4\,$ and $\,\mathrm{tr}\nh_g\w\hh b=0$, on a neighborhood of $\,x$,
\begin{equation}\label{off}
\varGamma^1_{\!44}\,=\,\varGamma^2_{\!44}\,=\,0\hh.
\end{equation}
In view of (\ref{toz}.b) and (\ref{abc}), relations (\ref{gij}), (\ref{skw}.a)
and (\ref{off}) yield
\begin{equation}\label{nij}
\begin{array}{l}
\nabla\hskip-2.7pt_{e\nh_i\w}\hskip-2.4pte\nh_j\w\,
=\,\varGamma^{\hs i}_{\!ij}e\nh_i\w\,
+\,\varGamma^k_{\nh\!ij}e\nh_k\w\hskip12pt
\mathrm{if}\hskip8pt\{i,j,k\}=\{1,2,3\}\hh,\\
\nabla\hskip-3.7pt_{e\nh_4\w}\hskip-3.2pte\nh_1\w\,
=\,\nabla\hskip-3.7pt_{e\nh_4\w}\hskip-3.2pte\nh_2\w\,=\,0\hh,\hskip20pt
\nabla\hskip-3.7pt_{e\nh_4\w}\hskip-3.2pte\nh_3\w\,
=\,-\hs\varGamma^3_{\!44}e\nh_4\w
\end{array}
\end{equation}
as $\,g(\nabla\hskip-3.7pt_{e\nh_4\w}\hskip-3.2pte\nh_i\w,e\nh_4\w)=0\,$
for $\,i=1,2\,$ by (\ref{off}). From (\ref{gij}), (\ref{skw}.a) and
(\ref{abc}),
\begin{equation}\label{nif}
\mathrm{i)}\hskip6pt\nabla\hskip-3.7pt_{e\nh_i\w}\hskip-3.2pte\nh_4\w
=-\varGamma\hskip-3pt_i\w e\nh_i\w\hskip6pt\mathrm{for}\hskip6pti=1,2,3\hh,
\hskip12pt
\mathrm{where}\hskip6pt\mathrm{ii)}\hskip6pt\varGamma\hskip-3pt_i\w
=\varGamma^4_{\!ii}=-\varGamma^{\hs i}_{\!i4}\hh.
\end{equation}
By (\ref{nij}), we thus have
$\,[e\hs \nh_4\w,e\nh_1\w]=\varGamma\hskip-4pt_1\w e\nh_1\w$. and
so $\,g(\nabla\!_{[e\hs \nh_4\w,\,e\nh_1\w]}e\nh_2\w,e\nh_4\w)=0$. Also, from (\ref{nij}), 
$\,g(\nabla\hskip-3.7pt_{e\nh_4\w}\hskip-3pt\nabla\hskip-3.7pt_{e\nh_1\w}
\hskip-3.2pte\nh_2\w,e\nh_4\w)
=-\varGamma^3_{\!12}\varGamma^3_{\!44}$ and 
$\,\nabla\hskip-3.7pt_{e\nh_1\w}\hskip-3pt\nabla\hskip-3.7pt_{e\nh_4\w}
\hskip-3.2pte\nh_2\w=0$. Combining the 
last three equalities with (\ref{rvw}) and noting that 
$\,\bz_4\w=(\lambda_2\w-\lambda_3\w)\,\varGamma^3_{\!12}$ in (\ref{ale}), we 
now get 
$\,(\lambda_3\w-\lambda_2\w)\hs g(R\hh(e\nh_1\w,e\nh_4\w)e\nh_2\w,\,e\nh_4\w)
=\bz_4\w\varGamma^3_{\!44}$. Thus, by (\ref{rof}) and (b),
$\,\varGamma^3_{\!44}=0$, which, along with 
(\ref{off}), proves the first part of (\ref{iff}.i), while
the second part then follows from Lemma~\ref{ggnez}(ix) for $\,l=4\,$ if we 
set, this time, $\,\beta=\alpha^2\nh$, where $\,\alpha\,$ is given by
(\ref{aea}) with $\,l=4$. 
In view of (\ref{nif}.ii), since $\,\mathrm{tr}\nh_g\w\hh b=0$, 
Lemma~\ref{ggnez}(xi) can be rewritten as 
\begin{equation}\label{lim}
(\lambda_i\w\,-\,\lambda_j\w)(D\nnh_4\w\alpha\,
-\,2\alpha\varGamma\hskip-3.5pt_k\w)\,=\,
2\alpha(\lambda_i\w\,+\,\lambda_j\w)(\varGamma\hskip-4pt_j\w\,
-\,\varGamma\hskip-3pt_i\w)
\end{equation}
for $\,i,j,k\,$ with $\,\{i,j,k\}=\{1,2,3\}$. Furthermore,
$\,R_{ij4k}\w$ evaluated from (\ref{rvw}), (\ref{nij}), (\ref{nif}) and the
resulting relation 
$\,[\hs e\hs\nh_i\w,e\nh_j\w]=\varGamma^{\hs i}_{\!ij}e\nh_i\w-\varGamma^j_{\!ji}e\nh_j\w+
(\varGamma^k_{\nh\!ij}-\varGamma^k_{\!ji})e\nh_k\w$, valid if 
$\{i,j,k\}=\{1,2,3\}$, equals
$\,(\varGamma\hskip-4pt_j\w-\varGamma\hskip-3.5pt_k\w)\varGamma^k_{\nh\!ij}
-(\varGamma\hskip-3pt_i\w-\varGamma\hskip-3.5pt_k\w)\varGamma^k_{\!ji}$. As
$\,R_{ij4k}\w=0$ in (\ref{rof}), this means that
$\,\bz_4\w[(\lambda_i\w-\lambda_k\w)^{-1}(\varGamma\hskip-3pt_i\w
-\varGamma\hskip-3.5pt_k\w)-(\lambda_j\w
-\lambda_k\w)^{-1}(\varGamma\hskip-4pt_j\w
-\varGamma\hskip-3.5pt_k\w)]=0$ whenever $\,\{i,j,k\}=\{1,2,3\}$, cf.\
(\ref{ale}). Thus, as $\,\bz_4\w\ne0\,$ in (b). for some function $\,\psi\,$
not depending on the choice of $\,i,j\in\{1,2,3\}$, one has
\begin{equation}\label{gim}
\varGamma\hskip-3pt_i\w\,-\hs\varGamma\hskip-4pt_j\w\,
=\,(\lambda_i\w-\lambda_j\w)\psi\hh.
\end{equation}
In view of (b) and (\ref{gim}), 
relation (\ref{lim}) now becomes 
\[
\varGamma\hskip-3.5pt_k\w\,=\,\frac{D\nnh_4\w\alpha}{2\alpha}\,+
\,(\lambda_i\w\,+\,\lambda_j\w)\psi\hskip12pt
\mathrm{if}\hskip8pt\{i,j,k\}=\{1,2,3\}\hh.
\]
Hence $\,\varGamma\hskip-3.5pt_k\w-\varGamma\hskip-4pt_j\w
=(\lambda_j\w-\lambda_k\w)\psi\,$ is the {\it opposite\/} of the value in 
(\ref{gim}), implying, via (b), that $\,\psi=0$. The last displayed
equality now yields $\,D\nnh_4\w\alpha=2\alpha\hs\varGamma\hskip-3pt_i\w$ for 
$\,i=1,2,3\,$ and, as $\,\beta=\alpha^2\nh$, (\ref{iff}.ii) follows from
(\ref{nif}.ii).
\end{proof}
\begin{lemma}\label{tneez}Under the hypotheses of Lemma\/~{\rm\ref{neffz}},
for the function\/ $\,\beta\,$ in\/ {\rm(\ref{nef})}, $\,(M\nh,g)\,$ is
locally isometric to a warped product\/ 
$\,(I\times N,\,dt^2+|\beta|^{-\nh1/2}\nh h)$, with the notation of 
Lemma~\/{\rm\ref{warpd}}, where\/ $\,\beta\,$ is constant in the\/ $\,N\,$ 
direction.
\end{lemma}
In fact, (\ref{tnv}) and (\ref{nef}) give 
$\,\hat\nabla\hat e\nh_4\w=0$, where 
$\,\hat e\nh_4\w=|\beta|^{-\nh1/4}e\nh_4\w$ and $\,\hat\nabla\,$ is the 
Le\-vi-Ci\-vi\-ta connection of the metric $\,\hg=|\beta|^{1/2}g$, 
con\-for\-mal to $\,g$, 
Thus, $\,(M\nh,\hg)\,$ is locally isometric 
to a Riemannian product of an interval $\,I\,$ and a Riemannian 
$\,3$-man\-i\-fold $\,(N,h)$. Since $\,D\nnh_i\w\beta=0\,$ for $\,i=1,2,3\,$ 
by (\ref{nef}), our claim follows.
\begin{theorem}\label{nuone}Given an oriented non-Ein\-stein Riem\-ann\-i\-an 
four-man\-i\-fold $\,(M\nh,g)$ with\/ $\,\mathrm{div}\,R=0\,$ such that 
$\,W^+:\Lambda^{\!+\!}M\to\Lambda^{\!+\!}M\,$ has three distinct eigen\-values 
at some point of $\,M\nh$, let\/ $\,\text{\smallbf d}\,$ be the invariant
mentioned in Remark\/~{\rm\ref{invrd}}. 
If\/ $\,\text{\smallbf d}=1$, then\/ $\,(M\nh,g)\,$ is locally of type 
{\rm(\ref{kne}.c)}.
\end{theorem}
This is immediate from 
Lemmas \ref{tneez} and~\ref{warpd}, as the assumption on $\,W^+$ 
precludes con\-for\-mal flatness of $\,(M\nh,g)$.

\section{Case {\rm(\ref{toz}.c)}.}\label{nz}
This section discusses, in some detail, four-man\-i\-folds with harmonic
curvature such that 
$\,\text{\smallbf d}=0\,$ (notation of Section~\ref{sd}).
\begin{lemma}\label{aandb}Let\/ $\,A,b\,$ and\/ $\,(M\nh,g)\,$ satisfy the
hypotheses of Lemma\/~{\rm\ref{codwe}} and\/ {\rm(\ref{toz}.c)}, with\/ 
$\,e\nh_i\w,\varGamma^{\hs i}_{\!ij},\lambda_i\w,\sigma\nnh_{ij}\w$ as in
Lemma\/~{\rm\ref{codwe}(i)} and\/ {\rm(\ref{gij})}. Setting
\begin{equation}\label{fji}%
F\hskip-4pt_{ji}\w\nnh
=\nh\varGamma^{\hs i}_{\!ij}\hskip4pt\mathrm{and}\hskip3ptH\hskip-2.5pt_{ji}\w
\nh
=F\hskip-3.2pt_{kl}\w F\hskip-3.2pt_{li}\w\nh
+F\hskip-3.2pt_{lk}\w F\hskip-3.2pt_{ki}\w\nh
-F\hskip-3.2pt_{ki}\w F\hskip-3.2pt_{li}\w\,\mathrm{\ if\ 
}\,\{i,j,k,l\}=\{1,2,3,4\}\hh,
\end{equation}
one has the following relations, valid as long as\/
$\,\{i,j,k,l\}=\{1,2,3,4\}$.
\begin{equation}\label{nii}
\begin{array}{rl}
\mathrm{a)}&\nabla\hskip-3.7pt_{e\nh_i\w}\hskip-2.4pte\nh_i\w\,=\,
-\,F\hskip-4pt_{ji}\w e\nh_j\w\,-\,F\hskip-3.2pt_{ki}\w e\nh_k\w\,
-\,F\hskip-3.2pt_{li}\w e\nh_l\w\hh,\hskip28pt
\nabla\hskip-2.7pt_{e\nh_i\w}\hskip-2.4pte\nh_j\w\,
=\,F\hskip-4pt_{ji}\w e\nh_i\w\hh,\\
\mathrm{b)}&D\nnh_i\w\lambda_j\w\,
=\,(\lambda_i\w\,-\,\lambda_j\w)F\hskip-3.2pt_{ij}\w\hh,\\
\mathrm{c)}&D\nnh_k\w\sigma\nnh_{ij}\w\,
=\,(\sigma\nnh_{kj}\w\nh-\sigma\nnh_{ij}\w)F\hskip-3.2pt_{ki}\w\,+
\,(\sigma\nnh_{ki}\w\nh-\sigma\nnh_{ij}\w)F\hskip-3.2pt_{kj}\w\hh,\\
\mathrm{d)}&[\hs e\hs\nh_i\w,e\nh_j\w]\,=\,F\hskip-4pt_{ji}\w e\nh_i\w\,
-\,F\hskip-3.2pt_{ij}\w e\nh_j\w\hh,\\
\mathrm{e)}&D\nnh_i\w F\hskip-3.2pt_{ij}\w\,+\,D\nnh_j\w F\hskip-4pt_{ji}\w\,
+\,F\hskip-3.2pt_{ij}^{\hskip3.3pt2}\,+\,F\hskip-4pt_{ji}^{\hskip4.4pt2}\,
+\,F\hskip-3.2pt_{ki}\w F\hskip-3.2pt_{kj}\w\,
+\,F\hskip-3.2pt_{li}\w F\hskip-3.2pt_{lj}\w\,=\,-R_{ijij}\w\hh,\\
\mathrm{f)}&D\nnh_i\w F\hskip-4pt_{jk}\w\,-\,(F\hskip-4pt_{ji}\w\nnh-\nnh F\hskip-4pt_{jk}\w)F\hskip-3.2pt_{ik}\w\,
=\,R_{kijk}\w\hh.
\end{array}
\end{equation}
In the case where $\,\,\mathrm{tr}\nh_g\w\hh b\,$ is constant and 
$\,\{i,j,k,l\}=\{1,2,3,4\}$,
\begin{equation}\label{dil}%
D\nnh_k\w\lambda_k\w\,=\,(\lambda_l\w\,-\,\lambda_k\w)F\hskip-3.2pt_{kl}\w\,
+\,(\lambda_i\w\,-\,\lambda_k\w)F\hskip-3.2pt_{ki}\w\,+\,(\lambda_j\w\,
-\,\lambda_k\w)F\hskip-3.2pt_{kj}\w\hh.
\end{equation}
If $\,A=W\,$ and $\,b=\,\mathrm{Ric}\,-\,\mathrm{s}\hskip.7ptg/4\,$ then, 
whenever\/ $\,\{i,j,k,l\}=\{1,2,3,4\}$,
\begin{equation}\label{dif}%
\begin{array}{rl}
\mathrm{i)}&D\nnh_i\w F\hskip-4pt_{jk}\w\,=\,(F\hskip-4pt_{ji}\w\nnh-\nnh F\hskip-4pt_{jk}\w)F\hskip-3.2pt_{ik}\w\hh,\\
\mathrm{ii)}&(\sigma\nnh_{ki}\w\hs-\,\sigma\nnh_{li}\w)(D\nnh_k\w F\hskip-3.2pt_{lk}\w\,
-\,D\hn_l\w F\hskip-3.2pt_{kl}\w)\,=\,3\hs(H\hskip-2.5pt_{ji}\w\,
-\,H\nnh_{ij}\w)\hs\sigma\nnh_{ij}\w\hh,
\\
\mathrm{iii)}&(\lambda_k\w\,-\,\lambda_l\w)(D\nnh_k\w F\hskip-3.2pt_{lk}\w\,
-\,D\hn_l\w F\hskip-3.2pt_{kl}\w)\\
&\hskip25pt=\,\,(\lambda_k\w\,+\,\lambda_l\w
-\,2\lambda_i\w)H\hskip-2.5pt_{ji}\w\,
+\,\hs(\lambda_k\w\,+\,\lambda_l\w-\,2\lambda_j\w)H\nnh_{ij}\w\hh,\\
\mathrm{iv)}&D\nnh_i\w D\nnh_k\w F\hskip-3.2pt_{lk}\w\,
+\,2F\hskip-3.2pt_{ik}\w D\nnh_k\w F\hskip-3.2pt_{lk}\w\,
=\,(F\hskip-3.2pt_{li}\w\nh-F\hskip-3.2pt_{lk}\w)D\nnh_k\w F\hskip-3.2pt_{ik}\w\hh,\\
\mathrm{v)}&D\nnh_i\w D\nnh_k\w F\hskip-3.2pt_{kl}\w
+(F\hskip-3.2pt_{ik}\w\nnh+F\hskip-3.2pt_{il}\w)D\nnh_k\w F\hskip-3.2pt_{kl}\w
=F\hskip-3.2pt_{il}\w D\nnh_k\w F\hskip-3.2pt_{ki}\w
+(F\hskip-3.2pt_{ki}\w\nh-F\hskip-3.2pt_{kl}\w)H\hskip-2.5pt_{jl}\w\hh.
\end{array}
\end{equation}
\end{lemma}
\begin{proof}Assertions (\ref{nii}.a) -- (\ref{nii}.c) are obvious from 
(\ref{gij}) and (\ref{skw}.a) with (\ref{toz}.c) and (\ref{abc}) or,
respectively, 
(\ref{skw}.e) -- (\ref{skw}.f), while (\ref{nii}.d) is immediate from 
(\ref{nii}.a). Evaluating $\,R_{ijij}\w$ and $\,R_{kijk}\w$ from (\ref{rkl}), 
(\ref{rvw}) and (\ref{nii}.d), we easily obtain (\ref{nii}.e) -- (\ref{nii}.f).

If $\,\hs\mathrm{tr}\nh_g\w\hh b
=\lambda_i\w+\lambda_j\w+\lambda_k\w+\lambda_l\w$ 
is constant, (\ref{nii}.b) implies (\ref{dil}). 

The assumptions made in the line preceding (\ref{dif}) yield (\ref{dif}.i) as 
a consequence of (\ref{nii}.f), since Lemma \ref{simdi} applied to 
$\,b=\,\mathrm{Ric}\,-\,\mathrm{s}g/4\,$ gives $\,R_{kijk}=0$. Relations 
(\ref{dif}.ii) -- (\ref{dif}.iii), or (\ref{dif}.iv) -- (\ref{dif}.v), then 
follow in turn from the equalities 
$\,D\nnh_i\w D\nnh_j\w\sigma\nnh_{il}\w-D\nnh_j\w D\nnh_i\w\hh\sigma\nnh_{il}\w
=D\nnh_w\w\nh\sigma\nnh_{il}\w$ and $\,D\nnh_k\w D\hn_l\w\lambda_k\w-
D\hn_l\w D\nnh_k\w\lambda_k\w=D\nnh_w\w\lambda_k\w$, where 
$\,w=[e\hs \nh_k\w,e\hn_l\w]\,$ (or, respectively, 
$\,D\nnh_i\w D\nnh_k\w F\hskip-3.2pt_{lk}\w
-D\nnh_k\w D\nnh_i\w F\hskip-3.2pt_{lk}\w=D\nnh_w\w F\hskip-3.2pt_{lk}\w$ and 
$\,D\nnh_i\w D\nnh_k\w F\hskip-3.2pt_{kl}\w
-D\nnh_k\w D\nnh_i\w F\hskip-3.2pt_{kl}\w=D\nnh_w\w F\hskip-3.2pt_{kl}\w$,
where 
$\,w=[\hs e\hs\nh_i\w,e\nh_k\w]$), combined with (\ref{nii}.d), (\ref{nii}.b), 
(\ref{dif}.i) and (\ref{dil}).
\end{proof}
\begin{proof}[Proof of Theorem\/{\rm~\ref{ricev}(a)\hs-\hs(b)}]Due to
(\ref{tis}), $\,\text{\smallbf w}=3\,$ and $\,\text{\smallbf r}\in\{3,4\}$.
The assumptions of Lemma~\ref{codwe} thus hold for
$\,(A,b)=(W\hskip-2.3pt,\,\mathrm{Ric}-\mathrm{s}\hskip.7ptg/4)$.
Lemma~\ref{zotwo} now gives $\,\text{\smallbf d}\in\{0,1,2\}$, and so
$\,\hs\text{\smallbf d}\,=\,\nh0$, the cases
$\,\hs\text{\smallbf d}\,=\,\nh2\,$ and 
$\,\hs\text{\smallbf d}\,=\,\nh1\,$ being excluded by the argument in
Section~\ref{nt} and, respectively, Theorem~\ref{nuone}. Assertions (a) and
(b) of Theorem~\ref{ricev} are now immediate from (\ref{nii}.d) and the choice
of $\,e\nh_i\w$ in Lemma~\ref{codwe}(i).
\end{proof}
\begin{proof}[Proof of Theorem\/{\rm~\ref{algvr}}]The first displayed equation
is obvious from (\ref{skw}.b) and the definition of $\,\lambda_i\w$. Next,
adding $\,\lambda_l\w-\lambda_k\w$ times (\ref{dif}.ii) to (\ref{dif}.iii)
multiplied by $\,\sigma\nnh_{ki}\w\nh-\sigma\nnh_{li}\w$, we obtain
$\,0=2(H\nnh_{ij}\w Z_{klj}\w-H\hskip-2.5pt_{ji}\w Z_{kli}\w)$, where
$\,2Z_{klj}\w=3(\lambda_k\w-\lambda_l\w)\hs\sigma\nnh_{ij}\w+
(\lambda_k\w+\lambda_l\w-2\lambda_j\w)(\sigma\nnh_{ki}\w\nh
-\sigma\nnh_{li}\w)$ 
and
\[
\begin{array}{l}
2Z_{kli}\w=3(\lambda_k\w-\lambda_l\w)\hs\sigma\nnh_{ji}\w+
(\lambda_k\w+\lambda_l\w-2\lambda_i\w)(\sigma\nnh_{kj}\w\nh
-\sigma\nnh_{lj}\w)\\
\phantom{Z_{kli}\w}=\,3(\lambda_k\w-\lambda_l\w)\hs\sigma\nnh_{ij}\w+
(\lambda_k\w+\lambda_l\w-2\lambda_i\w)(\sigma\nnh_{li}\w\nh
-\sigma\nnh_{ki}\w).
\end{array}
\]
(By (\ref{skw}.b),
$\,(\sigma\nnh_{ji}\w,\sigma\nnh_{kj}\w,\sigma\nnh_{lj}\w)
=(\sigma\nnh_{ij}\w,\sigma\nnh_{li}\w,\sigma\nnh_{ki}\w)$.) Thus, as
$\,\sigma\nnh_{ij}\w=\sigma\nnh_{kl}\w$ and
$\,\sigma\nnh_{ki}\w=\sigma\nnh_{ik}\w$ in (\ref{skw}.b), the last displayed
line gives 
$\,2Z_{kli}\w=3(\lambda_k\w-\lambda_l\w)\hs\sigma\nnh_{kl}\w+
(\lambda_k\w+\lambda_l\w-2\lambda_i\w)\sigma\nnh_{li}\w\nh
+(2\lambda_i\w-\lambda_k\w-\lambda_l\w)\sigma\nnh_{ik}\w$, that is,
$\,2Z_{kli}\w
=2(\lambda_k\w-\lambda_l\w)\hs\sigma\nnh_{kl}\w
+(\lambda_k\w-\lambda_l\w)\hs\sigma\nnh_{kl}\w
+2(\lambda_l\w-\lambda_i\w)\sigma\nnh_{li}\w\nh
+(\lambda_k\w-\lambda_l\w)\sigma\nnh_{li}\w\nh
+2(\lambda_i\w-\lambda_k\w)\sigma\nnh_{ik}\w
+(\lambda_k\w-\lambda_l\w)\sigma\nnh_{ik}\w$. Due to (\ref{skw}.b), or
the definition of $\,Z\nnh_j\w$ in Theorem~\ref{algvr}, the second, fourth
and sixth (or, first, third and fifth) terms in the last six-term sum add up
to $\,0\,$ or, respectively, to $\,2Z\nnh_j\w$, if $\,(i,j,k,l)$, and hence
$\,(k,l,i,j)$, is an even permutation of $\,(1,2,3,4)$. Thus,
$\,Z_{kli}\w=Z\nnh_j\w$ while, the permutation $\,(k,l,j,i)\,$ being odd,
$\,Z_{klj}\w=-\nh Z_i\w$, and so the equality
$\,H\nnh_{ij}\w Z_{klj}\w=H\hskip-2.5pt_{ji}\w Z_{kli}\w$ obtained above
amounts to (\ref{fsi}). Furthermore,
$\,Z_i\w\nh+Z\nnh_j\w\nh+Z_k\w\nh+Z_l\w\nh=0\,$ if 
$\,(i,j,k,l)\,$ is an even permutation of $\,(1,2,3,4)$, since the term
$\,(\lambda_i\w-\lambda_j\w)\hh\sigma\nnh_{ij}\w$ in 
$\,Z_l\w\nh=(\lambda_i\w-\lambda_j\w)\hh\sigma\nnh_{ij}\w\nh
+(\lambda_j\w-\lambda_k\w)\hh\sigma\hskip-2.7pt_{jk}\w\nh
+(\lambda_k\w-\lambda_i\w)\hh\sigma\nnh_{ki}\w$ gets cancelled by 
$\,(\lambda_j\w-\lambda_i\w)\hh\sigma\nnh_{ji}\w$ in 
$\,Z_k\w$, as a consequence of (\ref{skw}.b) and evenness of the permutation 
$\,(j,i,l,k)$. Therefore,
\begin{equation}\label{zje}
Z_1\w\hs+\,Z_2\w\hs+\,Z_3\w\hs+\,Z_4\w\hs=\,0\,
=\,\lambda_1\w\hs+\,\lambda_2\w\hs+\,\lambda_3\w\hs+\,\lambda_4\w\hh,
\end{equation}
the second relation in (\ref{zje}) being a part of the
al\-read\-y-es\-tab\-lish\-ed first displayed equation in Theorem~\ref{algvr}.
From (\ref{fsi}) and (\ref{zje}) we get
\begin{equation}\label{zhe}
\left[\begin{matrix}
Z_1\w&Z_2\w&Z_3\w&Z_4\w\end{matrix}\right]\nnh\text{\medbf H}\,
=\nnh\left[\begin{matrix}0&0&0&0&0&0\end{matrix}\right]\nnh,
\end{equation}
where $\,\text{\medbf H}\,$ denotes the $\,4\times7\,$ matrix in (\ref{fsp}).
With $\,\lambda_4\w$ replaced by
$\,-\lambda_1\w\nh-\lambda_2\w\nh-\lambda_3\w$,
cf.\ (\ref{zje}), the definition of $\,Z\nnh_j\w$ gives
\begin{equation}\label{zoe}
\left[\begin{matrix}Z_1\w\cr
Z_2\w\cr
Z_3\w\cr\end{matrix}\right]
=\left[\begin{matrix}\sigma\nnh_{24}\w\nnh-\nh\sigma\nnh_{43}\w
&2\sigma\nnh_{24}\w\nnh-\nh\sigma\nnh_{32}\w\nnh-\nh\sigma\nnh_{43}\w
&\sigma\nnh_{24}\w\nnh+\nh\sigma\nnh_{32}\w\nnh-\nh2\sigma\nnh_{43}\w\cr
\sigma\nnh_{13}\w\nnh+\nh\sigma\nnh_{34}\w\nnh-\nh2\sigma\nnh_{41}\w
&\sigma\nnh_{34}\w\nnh-\nh\sigma\nnh_{41}\w
&2\sigma\nnh_{34}\w\nnh-\nh\sigma\nnh_{13}\w\nnh-\nh\sigma\nnh_{41}\w\cr
2\sigma\nnh_{14}\w\nnh-\nh\sigma\nnh_{21}\w\nnh-\nh\sigma\nnh_{42}\w
&\sigma\nnh_{21}\w\nnh+\nh\sigma\nnh_{14}\w\nnh-\nh2\sigma\nnh_{42}\w
&\sigma\nnh_{14}\w\nnh-\nh\sigma\nnh_{42}\w
\end{matrix}\right]
\left[\begin{matrix}
\lambda_1\w\cr
\lambda_2\w\cr
\lambda_3\w\end{matrix}\right]\nnh.
\end{equation}
Due to (\ref{skw}.b), the $\,3\times\nh3\,$ matrix appearing in (\ref{zoe})
equals
\[
\left[\begin{matrix}\sigma\nnh_{13}\w\nnh-\nh\sigma\nnh_{12}\w
&3\sigma\nnh_{13}\w&-\nnh3\sigma\nnh_{12}\w\cr
-\nnh3\sigma\nnh_{14}\w
&\hskip-4pt\sigma\nnh_{12}\w\nnh-\nh\sigma\nnh_{14}\w\hskip-4pt
&3\sigma\nnh_{12}\w\cr
3\sigma\nnh_{14}\w
&-\nnh3\sigma\nnh_{13}\w&\sigma\nnh_{14}\w\nnh-\nh\sigma\nnh_{13}\w
\end{matrix}\right]\nnh,
\]
and so it has the determinant
$\,-\nh8(\sigma\nnh_{13}\w\nnh-\nh\sigma\nnh_{12\w})
(\sigma\nnh_{12}\w\nnh-\nh\sigma\nnh_{14}\w)
(\sigma\nnh_{14}\w\nnh-\nh\sigma\nnh_{13}\w)$,
nonzero according to (\ref{skw}.a). while (\ref{tis}) and (\ref{zje}) give
$\,(\lambda_1\w,\lambda_2\w,\lambda_3\w)\ne(0,0,0)$. Thus, 
$\,(Z_1\w,Z_2\w,Z_3\w)\ne(0,0,0)\,$ in (\ref{zoe}), and (\ref{fsp}) follows
from (\ref{zhe}).
\end{proof}
\begin{remark}\label{lccon}
Let the tangent bundle $\,T\hn M\,$ of an $\,n$-di\-men\-sion\-al manifold
$\,M$ be trivialized by vector fields $\,e\nh_1\w,\dots,e\nh_n\w$ 
satisfying the Lie-brack\-et relations (\ref{lbr}).
Then, for the Le\-vi-Ci\-vi\-ta connection 
$\,\nabla\,$ of the metric $\,g\,$ on $\,M\,$ such that 
$\,e\nh_1\w,\dots,e\nh_n\w$ are $\,g$-or\-tho\-nor\-mal,
\begin{equation}\label{nie}
\nabla\hskip-2.7pt_{e\nh_i\w}\hskip-2.4pte\nh_j\w\,
=\,F\hskip-4pt_{ji}\w e\nh_i\w
\hskip9pt\mathrm{if}\hskip6pti\ne j\hh,\hskip32pt
\nabla\hskip-2.7pt_{e\nh_i\w}\hskip-2.4pte\nh_i\w\,=\,
-\sum_{k\ne i}\w F\hskip-3.2pt_{ki}\w e\nh_k\w\hh.
\end{equation}
In fact, the connection {\it defined\/} by (\ref{nie}) is tor\-sion-free and 
makes $\,g\,$ parallel.
\end{remark}
\begin{remark}\label{mltpl}
The assumptions of Remark \ref{lccon} will still hold if 
$\,e\nh_1\w,\dots,e\nh_n\w$ are replaced by 
$\,\phi\nh_1\w\nh e\nh_1\w,\dots,\phi\nh_n\w\nh e\nh_n\w$, for any functions 
$\,\phi\nh_i\w$ without zeros.

Locally, such $\,\phi\nh_i\w$ may then be chosen so that 
$\,\phi\nh_i\w\nh e\nh_i\w$ commute, as one sees setting 
$\,\phi\nh_i\w\nh e\nh_i\w\nh=\partial\nh_i\w$, where $\,\partial\nh_i\w$ 
are the coordinate vector fields for local coordinates $\,x^i\nh$, with each 
$\,x^i$ constant along the integrable distribution 
$\,\mathrm{Span}\hs\{e\nh_j\w:j\ne i\}$.
\end{remark}
\begin{remark}\label{prldi}
Let the objects appearing in Remark \ref{lccon} also have the property that 
$\,[\hs e\hs\nh_i\w,e\nh_k\w]=0\,$ for some fixed $\,m\in\{1,\dots,n-1\}\,$
and all $\,i,k$ with $\,i\le m<k$. Then the distributions 
$\,\mathcal{V}=\mathrm{Span}\hs\{e\nh_1\w,\dots,e\nh_m\w\}\,$ and 
$\,\mathcal{H}=\mathrm{Span}\hs\{e\nh_{m+1}\w,\dots,e\nh_n\w\}\,$ are 
$\,g$-par\-al\-lel. Namely, whenever $\,i\le m<k$, Remark \ref{lccon} with 
$\,F\hskip-3.2pt_{ik}\w=F\hskip-3.2pt_{ki}\w=0$ implies that $\,e\nh_k\w$ is 
$\,g$-par\-al\-lel along $\,e\nh_i\w$, and vice versa. Thus,
$\,\mathcal{V}\hs$ 
and $\,\mathcal{H}\,$ are $\,g$-par\-al\-lel along each other. Hence 
$\,\mathcal{V}=\hn\mathcal{H}^\perp$ is $\,g$-par\-al\-lel along 
$\,\mathcal{V}\hs$ as well. 
\end{remark}

\section{Proof of Theorem~\ref{ricev}({\rm c}): part one}\label{dz}
We will prove Theorem~\ref{ricev}(c) by assuming its negation which,
by (\ref{tis}), means that $\,\text{\smallbf r}=3$, and -- at the end of
Section~\ref{lc} -- obtaining a contradiction.

Throughout this and the next three sections $\,(M\nh,g)\,$ is a fixed oriented
Riem\-ann\-ian four-man\-i\-fold with $\,\mathrm{div}\,R=0$, belonging to
class \hbox{(D\hs0)} of Section \ref{sd} and having $\,\text{\smallbf r}=3$,
while $\,\,U\,$ denotes the set of all generic points (Section \ref{sd}) at
which $\,\sigma\nnh_{ij}\w\nh\sigma\nnh_{ik}\w\nh\sigma\nnh_{il}\w\ne0$, that
is, the eigen\-values of $\,W^+$ and $\,W^-$ are all nonzero. The indices
$\,i,j,k,l\,$ are always assumed to satisfy the condition
\begin{equation}\label{ijk}
\{i,j\}\,=\,\{1,2\}\hh,\hskip22pt\{k,l\}\,=\,\{3,4\}\hh.
\end{equation}
By (\ref{snz}) and (\ref{ana}), $\,\,U$ is an open dense subset of $\,M\nh$. 
We use the notation of Lemma \ref{aandb}, for $\,A=W\hs$ and 
$\,b=\hs\mathrm{Ric}-\mathrm{s}g/4$, cf.\ the lines following (\ref{rrw}), so
that, without loss of generality, on every connected component $\,\,U'$ of
$\,\,U\nh$, for some function $\,\lambda$,
\begin{equation}\label{loe}
\begin{array}{rl}
\mathrm{a)}&\lambda_1\w\,=\,\,\lambda_2\w\,=\,\,\lambda\,\,\ne\,\,\lambda_3\w\,
\ne\,\,\lambda_4\w\,\ne\,\,\lambda\hh,\\
\mathrm{b)}&\lambda_3\w\hs+\,\lambda_4\w\hs=\,-2\lambda
\end{array}
\end{equation}
everywhere in $\,\,U'\nnh$. Setting $\,\sigma\hs=\sigma\nnh_{1\nh2}\w$, we 
also define the function
\begin{equation}\label{gkl}
G\nh_k\w\hs=\,\,(\lambda_k\w-\lambda)^{-\nh1}D\nnh_k\w\lambda\hskip9pt
\mathrm{for}\hskip6ptk\in\{3,4\}\hh,
\end{equation}
and a metric $\,\hg\,$ on $\,\,U'$ by requiring 
$\,\hg$-or\-tho\-nor\-mal\-i\-ty of 
$\,\hat e\nh_1\w,\dots,\hat e\nh_4\w$, where
\begin{equation}\label{hat}
\hat e\nh_i\w=\sigma\hs^{-\nnh1\nh/\nh3}e\nh_i\w\hskip5pt\mathrm{if}\hskip5pti
=1,2\hskip5pt\mathrm{and}\hskip5pt\hat e\nh_k\w
=(\lambda_k\w\nnh-\lambda)^{-\nh1}e\nh_k\w\hskip5pt\mathrm{if}\hskip5ptk
=3,4\hh.
\end{equation}
Assuming (\ref{ijk}), we now obtain from (\ref{nii}.d), as in 
Remark \ref{mltpl},
\begin{equation}\label{sim}
\begin{array}{rl}
\mathrm{i)}&[\hat e\nh_i\w,\hat e\nh_j\w]=\hat F\hskip-4pt_{ji}\w\hat e\nh_i\w
-\hat F\hskip-3.2pt_{ij}\w\hat e\nh_j\w\hh,\hskip19pt
[\hat e\nh_k\w,\hat e\hn_l\w]=\hat F\hskip-3.2pt_{lk}\w\hat e\nh_k\w
-\hat F\hskip-3.2pt_{kl}\w\hat e\hn_l\w\hh,\\
\mathrm{ii)}&[\hat e\nh_i\w,\hat e\nh_k\w]=\hat F\hskip-3.2pt_{ki}\w\hat e\nh_i\w
-\hat F\hskip-3.2pt_{ik}\w\hat e\nh_k\w\hh,\hskip25pt\mathrm{where}\\
\mathrm{iii)}&\hat F\hskip-3.2pt_{ik}\w=\sigma\hs^{-\nnh1\nh/\nh3}(F\hskip-3.2pt_{ik}\w
-D\nnh_i\w\log|\lambda_k\w\nnh-\lambda|^{-\nh1}),\\
\mathrm{iv)}&\hat F\hskip-3.2pt_{ki}\w=(\lambda_k\w-\lambda)^{-\nh1}(F\hskip-3.2pt_{ki}\w\nh
-D\nnh_k\w\log|\sigma|\hh^{-\nh1\nh/\nh3}),\\
\mathrm{v)}&\hat F\hskip-3.2pt_{ij}\w=\sigma\hs^{-\nnh1\nh/\nh3}(F\hskip-3.2pt_{ij}\w
-D\nnh_i\w\log|\sigma|\hs^{-\nnh1\nh/\nh3})\hh,\\
\mathrm{vi)}&\hat F\hskip-3.2pt_{kl}\w=(\lambda_k\w-\lambda)^{-\nh1}(F\hskip-3.2pt_{kl}\w\nh
-D\nnh_k\w\log|\lambda_l\w\nnh-\lambda|^{-\nh1})\hh.
\end{array}
\end{equation}
\begin{lemma}\label{dilgk}
Under the hypotheses {\rm(\ref{ijk})} -- {\rm(\ref{sim})}, for\/ 
$\,G\nh_k\w,\hs \hg\,$ as above,
\begin{enumerate}
  \def\theenumi{{\rm\alph{enumi}}}
\item $D\nnh_i\w\lambda\,=\,D\nnh_i\w\nh G\nh_k\w\hs=\,0\,\,$ and 
$\,H\hskip-2.5pt_{ji}\w\nh
=\hs H\nnh_{ij}\w\nh=\hs F\hskip-3.2pt_{lk}\w\nh G\nh_k\w\nh
+\hs F\hskip-3.2pt_{kl}\w\nh G\nh_l\w\nh-\hs G\nh_k\w G\nh_l\w$,
\item $F\hskip-3.2pt_{ki}\w\hs=\,G\nh_k\w
=\,D\nnh_k\w\log|\sigma|\hh^{-\nh1\nh/\nh3}\,$ 
and $\,\,-2F\hskip-3.2pt_{kl}\w
=\,D\nnh_k\w\log|\sigma\hh^{-\nh1\nh/\nh3}\nh(\sigma\nnh_{il}\w\nh
-\sigma\hskip-2.7pt_{jl}\w)|$,
\item $D\nnh_k\w F\hskip-3.2pt_{lk}\w\nh=\hs D\hn_l\w F\hskip-3.2pt_{kl}\w$, 
$\,\,(F\hskip-3.2pt_{lk}\w\nh G\nh_k\w\nh
+\hs F\hskip-3.2pt_{kl}\w\nh G\nh_l\w\nh
-\hs G\nh_k\w G\nh_l\w)\hs\lambda=0$, and 
$\,\,H\nnh_{lj}\w\nh=F\hskip-3.2pt_{ik}\w\nh G\nh_k\w$,
\item $F\hskip-3.2pt_{ik}\w=\,D\nnh_i\w\log|\lambda_k\w\nnh-\lambda|^{-\nh1}\,$ and 
$\,\,H\nnh_{lj}\w\nh-\hh H\nnh_{kl}\w\nh
=\hs(F\hskip-3.2pt_{ik}\w\nh-F\hskip-3.2pt_{il}\w)(G\nh_k\w\nh-F\hskip-3.2pt_{kl}\w)$,
\item $\hat e\nh_k\w$ is $\,\hg$-par\-al\-lel along $\,\hat e\nh_i\w$, and 
vice versa, so that $\,[\hat e\nh_i\w,\hat e\nh_k\w]\,=\,0$,
\item $\hg\,$ is, locally, a Riem\-ann\-ian product of two surface metrics, 
with the factor distributions 
$\,\mathcal{V}=\mathrm{Span}\hs\{e\nh_1\w,e\nh_2\w\}\,$ and 
$\,\mathcal{H}=\mathrm{Span}\hs\{e\nh_3\w,e\nh_4\w\}$,
\item $\hat F\hskip-3.2pt_{kl}\w=(\lambda_l\w\nnh-\lambda)^{-\nh1}(F\hskip-3.2pt_{kl}\w\nh
-G\nh_k\w)\,$ 
and $\,D\nnh_i\w\hat F\hskip-3.2pt_{kl}\w=D\nnh_k\w\hat F\hskip-3.2pt_{ij}\w=0$,
\item $D\nnh_i\w\lambda_k\w$ is nonzero on a dense open set, and\/ 
$\,(D\nnh_k\w\lambda+4\lambda F\hskip-3.2pt_{kl}\w)G\nh_l\w\nh=0$.
\end{enumerate}
\end{lemma}
\begin{proof}By (\ref{nii}.b) and (\ref{loe}), 
$\,D\nnh_i\w\lambda=D\nnh_i\w\lambda_j\w=(\lambda-\lambda)F\hskip-3.2pt_{ij}\w
=0\,$ and $\,D\nnh_k\w\lambda=D\nnh_k\w\lambda_i\w
=(\lambda_k\w-\lambda)F\hskip-3.2pt_{ki}\w$. Therefore $\,D\nnh_i\w\lambda=0$, 
while $\,F\hskip-3.2pt_{ki}\w=G\nh_k\w$ does not depend on $\,i\in\{1,2\}$, so 
that (\ref{dif}.i) yields $\,D\nnh_i\w\nh G\nh_k\w
=D\nnh_i\w F\hskip-3.2pt_{kj}\w=(G\nh_k\w\nh-G\nh_k\w)F\hskip-3.2pt_{ij}\w=0$. 
Combined with (\ref{fji}), this proves (a) and the first equality in (b). 
Since, by (\ref{skw}.b), $\,\sigma\nnh_{kj}\w+\sigma\nnh_{ki}\w
=-\hs\sigma\nnh_{kl}\w=-\hs\sigma\nnh_{ij}\w=-\hs\sigma$, (\ref{nii}.c) 
similarly gives $\,D\nnh_k\w\sigma=D\nnh_k\w\sigma\nnh_{ij}\w
=(\sigma\nnh_{kj}\w\nh-\sigma\nnh_{ij}\w)G\nh_k\w+
(\sigma\nnh_{ki}\w\nh-\sigma\nnh_{ij}\w)G\nh_k\w=-3\sigma G\nh_k\w$ as well as 
$\,\,D\nnh_k\w(\sigma\nnh_{il}\w\nh-\sigma\hskip-2.7pt_{jl}\w)
=(\sigma\hskip-2.7pt_{jl}\w\nh-\sigma\nnh_{il}\w)(G\nh_k\w\nh+2F\hskip-3.2pt_{kl}\w)$, and the 
remainder of (b) follows. Now (\ref{dif}.iii) reads 
$\,(\lambda_k\w\nnh-\lambda_l\w)(D\nnh_k\w F\hskip-3.2pt_{lk}\w\nh
-\hs D\hn_l\w F\hskip-3.2pt_{kl}\w)
=8\lambda(\nh G\nh_k\w G\nh_l\w\nh-\hs F\hskip-3.2pt_{lk}\w\nh G\nh_k\w\nh
-\hs F\hskip-3.2pt_{kl}\w\nh G\nh_l\w\hh)$, cf.\ (\ref{loe}). Thus, (c) is clear from 
(\ref{dif}.ii), (a) and (\ref{fji}). By (a), (\ref{loe}) and 
(\ref{nii}.b), $\,D\nnh_i\w(\lambda_k\w\nnh-\lambda)=D\nnh_i\w\lambda_k\w
=(\lambda-\lambda_k\w)F\hskip-3.2pt_{ik}\w$. This, along with (\ref{fji}), (b) 
and (c), implies (d). Next, (\ref{sim}.i) -- (\ref{sim}.ii) and
Remark \ref{prldi} 
lead to (e) -- (f), since
$\,\hat F\hskip-3.2pt_{ik}\w=\hat F\hskip-3.2pt_{ki}\w=0$ from 
(\ref{sim}.iii) -- (\ref{sim}.iv), (b) and (d).

Also, by (e), $\,\hat e\nh_i\w$ commutes with 
$\,[\hat e\nh_k\w,\hat e\hn_l\w]$ and 
$\,\hat e\nh_k\w$ with $\,[\hat e\nh_i\w,\hat e\nh_j\w]$, which in view of 
(\ref{sim}.i) yields $\,D\nnh_i\w\hat F\hskip-3.2pt_{lk}\w=D\nnh_k\w\hat F\hskip-3.2pt_{ij}\w=0$, 
and consequently (g): from (\ref{loe}), (\ref{nii}.b) and (\ref{gkl}), we get 
$\,(\lambda_l\w\nnh-\lambda)\hs D\nnh_k\w\log|\lambda_l\w\nnh-\lambda|^{-\nh1}
=D\nnh_k\w(\lambda-\lambda_l\w)=(\lambda_k\w\nnh-\lambda)G\nh_k\w
-(\lambda_k\w\nnh-\lambda_l\w)F\hskip-3.2pt_{kl}\w$, and thus 
$\,(\lambda_k\w\nnh-\lambda)(\lambda_l\w\nnh-\lambda)\hat F\hskip-3.2pt_{kl}\w$ equals 
$\,(\lambda_k\w\nnh-\lambda)(F\hskip-3.2pt_{kl}\w-G\nh_k\w)\,$ according to 
(\ref{sim}.vi),  while $\,\lambda_k\w\nnh\ne\lambda\,$ (see (\ref{loe}.a)). 

To prove the first claim in (h), we may suppose that, on the contrary, 
$\,D\nnh_i\w\lambda_k\w=0\,$ for both $\,k=3,4\,$ and both 
$\,i=1,2\,$ since, according to (\ref{loe}.b) and (a), 
$\,D\nnh_i\w\lambda_3\w=0\,$ if and only if $\,D\nnh_i\w\lambda_4\w=0$. In view 
of (a) and (d), $\,F\hskip-3.2pt_{ik}\w\nh=0$, so that (\ref{skw}.b) and 
(\ref{nii}.c) give 
$\,D\nnh_i\w\hh\sigma=D\nnh_i\w\hh\sigma\nnh_{ij}\w=D\nnh_i\w\hh\sigma\nnh_{kl}\w=0$. 
Thus, from (a), 
$\,D\nnh_i\w[(\lambda_k\w\nnh-\lambda)\sigma\hs^{-\nnh1\nh/\nh3}]=0\,$ and, as 
$\,\sigma\hs^{-\nnh1\nh/\nh3}e\nh_k\w
=(\lambda_k\w\nnh-\lambda)\sigma\hs^{-\nnh1\nh/\nh3}\hat e\nh_k\w$, the 
relation $\,[\hat e\nh_i\w,\hat e\nh_k\w]=0\,$ in (e) yields 
$\,[\hat e\nh_i\w,\sigma\hs^{-\nnh1\nh/\nh3}e\nh_k\w]=0$, that is, 
$\,[\sigma\hs^{-\nnh1\nh/\nh3}e\nh_i\w,\sigma\hs^{-\nnh1\nh/\nh3}e\nh_k\w]=0$. 
Applying Remarks \ref{mltpl} and \ref{prldi} to the vector fields 
$\,\sigma\hs^{-\nnh1\nh/\nh3}e\nh_i\w,\sigma\hs^{-\nnh1\nh/\nh3}e\nh_k\w$ and 
to the metric $\,\sigma\hs^{2\nh/\nh3}g\,$ making them or\-tho\-nor\-mal, 
we see that $\,g\,$ is, locally, con\-for\-mal to a Riem\-ann\-ian product of 
two surface metrics. Therefore, by Remark \ref{prsfm}, 
$\,\text{\smallbf w}\in\{1,2\}$, contradicting the definition of class 
\hbox{(D\hs0)} in Section \ref{sd}, which includes the requirement that 
$\,\text{\smallbf w}=3$.

By (\ref{loe}.b), $\,\lambda_l\w\nnh-\lambda=-\hh(\lambda_k\w\nnh+3\lambda)$, 
and so (g) implies that
\begin{equation}\label{fkl}
\begin{array}{rl}
\mathrm{i)}&F\hskip-3.2pt_{lk}\w\nh=\hs G\nh_l\w\nh+\hs(\lambda_k\w\nnh
-\lambda)\hat F\hskip-3.2pt_{lk}\w\hh,
\hskip15ptF\hskip-3.2pt_{kl}\w\nh=\hs G\nh_k\w\nh
-\hs(\lambda_k\w\nnh+3\lambda)\hat F\hskip-3.2pt_{kl}\w\hh,\\
\mathrm{ii)}&(\hat F\hskip-3.2pt_{lk}\w\nh G\nh_k\w\nh-\hs\hat F\hskip-3.2pt_{kl}\w\nh G\nh_l\w)\hh
\lambda\lambda_k\w\nh
=\hs(\hat F\hskip-3.2pt_{lk}\w\nh G\nh_k\w\nh+\hs3\hat F\hskip-3.2pt_{kl}\w\nh G\nh_l\w)\hh\lambda^2
-\hs\lambda G\nh_k\w\nh G\nh_l\w\hh,
\end{array}
\end{equation}
(\ref{fkl}.ii) being the result of using (\ref{fkl}.i) to rewrite the second 
equality in (c). Thus, 
$\,\lambda\hat F\hskip-3.2pt_{lk}\w\nh G\nh_k\w\nh=\hs\lambda\hat F\hskip-3.2pt_{kl}\w\nh G\nh_l\w$, 
or else, in an open set on which 
$\,\lambda\hat F\hskip-3.2pt_{lk}\w\nh G\nh_k\w\nh\ne\hs\lambda\hat F\hskip-3.2pt_{kl}\w\nh G\nh_l\w$, 
the formula for $\,\lambda_k\w$ arising from (\ref{fkl}.ii) would, by (a) and 
(g), show that $\,D\nnh_i\w\lambda_k\w=0\,$ for both $\,i=1,2$, 
contradicting the second part of (h). Since 
$\,\lambda\hat F\hskip-3.2pt_{lk}\w\nh G\nh_k\w\nh
=\hs\lambda\hat F\hskip-3.2pt_{kl}\w\nh G\nh_l\w$, (\ref{fkl}.ii) reads
\begin{equation}\label{flq}
4\lambda^2\hat F\hskip-3.2pt_{lk}\w\nh G\nh_k\w
=\,4\lambda^2\hat F\hskip-3.2pt_{kl}\w\nh G\nh_l\w
=\,\lambda G\nh_k\w G\nh_l\w\hh.
\end{equation}
By (\ref{fkl}), (\ref{flq}) and (\ref{gkl}), 
$\,4\lambda^2F\hskip-3.2pt_{kl}\w\nh G\nh_l\w\nh
=4\lambda^2\hh G\nh_k\w G\nh_l\w\nh-\hs4(\lambda_k\w\nnh
+3\lambda)\lambda^2\hat F\hskip-3.2pt_{kl}\w\nh G\nh_l\w\nh
=\hs(\lambda-\lambda_k\w)\lambda G\nh_k\w G\nh_l\w\nh=
\hs-\lambda G\nh_l\w D\nnh_k\w\lambda$. This proves the second equality in 
(h), as it obviously holds on any open set on which $\,\lambda=0$.
\end{proof}
\begin{remark}\label{difij}Assuming (\ref{ijk}), we have 
$\,F\hskip-3.2pt_{ki}\w\nh=G\nh_k\w$ and $\,\lambda=\lambda_i\w$ in view of 
Lemma \ref{dilgk}(b) and (\ref{loe}.a), while 
$\,\sigma=\sigma\nnh_{ij}\w\nh=\sigma\nnh_{kl}\w$ by (\ref{skw}.b), so that 
combining (\ref{nii}.e) with (\ref{jij}) we obtain 
$\,D\nnh_i\w F\hskip-3.2pt_{ij}\w\nh
+D\nnh_j\w F\hskip-4pt_{ji}\w\nh+F\hskip-3.2pt_{ij}^{\hskip3.3pt2}
+F\hskip-4pt_{ji}^{\hskip4.4pt2}+G\nh_k^{\hs2}\nh
+G\nh_l^{\hs2}\nh=-\hs\sigma-\lambda-\mathrm{s}/12\,$ along with 
$\,D\nnh_i\w F\hskip-3.2pt_{ik}\w\nh+D\nnh_k\w\nh G\nh_k\w\nh
+F\hskip-3.2pt_{ik}^{\hskip3.3pt2}
+G\nh_k^{\hs2}+F\hskip-4pt_{ji}\w F\hskip-4pt_{jk}\w\nh+G\nh_l\w F\hskip-3.2pt_{lk}\w\nh
=\hbox{$-\hs\sigma\nnh_{ik}\w\nh-(\lambda+\lambda_k\w\nh+\mathrm{s}/6)/2$}\,$
and, from (\ref{loe}.b), 
$\,D\nnh_k\w F\hskip-3.2pt_{kl}\w\nh+D\hn_l\w F\hskip-3.2pt_{lk}\w\nh+F\hskip-3.2pt_{kl}^{\hskip3.3pt2}
+F\hskip-3.2pt_{lk}^{\hskip3.3pt2}+F\hskip-3.2pt_{ik}\w F\hskip-3.2pt_{il}\w\nh+F\hskip-4pt_{jk}\w F\hskip-4pt_{jl}\w\nh
=-\hs\sigma+\lambda-\mathrm{s}/12$.
\end{remark}

\section{Proof of Theorem~\ref{ricev}({\rm c}): part two}\label{ce}
Throughout this section, again, $\,\{i,j\}=\{1,2\}\,$ and $\,\{k,l\}=\{3,4\}$. 
According to (\ref{loe}.b) and (\ref{skw}.b), there exist functions 
$\,\mu,\tau$, not depending on the choice of $\,i\in\{1,2\}\,$ and 
$\,k\in\{3,4\}$, such that, with $\,\sigma=\sigma\nnh_{ij}\w$,
\begin{equation}\label{lem}
\begin{array}{l}
\mathrm{a)}\hskip6pt\lambda_k\w=\,(\nnh-\nnh1\nh)^k\hn\mu-\lambda\hh,\hskip12pt
\mathrm{b)}\hskip6pt\sigma\nnh_{ik}\w\hs=\,(\nnh-\nnh1\nh)^{i+k}\tau\,
-\,\sigma\nnh/\nh2\hh,\\
\mathrm{c)}\hskip6pt\sigma\nnh_{ij}\w\nh=\sigma\nnh_{kl}\w\nh=\sigma\hh,
\hskip12pt
\mathrm{d)}\hskip6pt-2F\hskip-3.2pt_{kl}\w\nh=D\nnh_k\w\nnh\log|\mu|
+[\hs1-2(\nnh-\nnh1\nh)^k\lambda/\nnh\mu\hh]\hs G\nh_k\w\hh,
\end{array}
\end{equation}
where (\ref{lem}.d) follows from (\ref{lem}.a) and 
(\ref{gkl}), since 
$\,D\nnh_k\w\lambda_l\w\nh=(\lambda_k\w\nnh-\lambda_l\w)F\hskip-3.2pt_{kl}\w$, cf.\ 
(\ref{nii}.b). Note that $\,\mu\ne0\,$ (and (\ref{lem}.d) holds) on an open 
set which is nonempty (and therefore dense, by (\ref{ana})), or else 
(\ref{lem}.a) would give $\,\lambda_3\w\nh=\lambda_4\w$, contradicting 
(\ref{loe}.a). In view of Lemma \ref{dilgk}(g) and (\ref{lem}),
\begin{equation}\label{did}
2(\nnh-\nnh1\nh)^k\hat F\hskip-3.2pt_{kl}\w\hs=\,[D\nnh_k\w\nnh\log|\mu|\,+\,\{\hs3\hs
-\hs2(\nnh-\nnh1\nh)^k\lambda/\nnh\mu\hh\}\hs G\nh_k\w]/[\hs\mu
+\hs2(\nnh-\nnh1\nh)^k\lambda\hh]\hh.
\end{equation}
Using (\ref{skw}.b), Lemma \ref{dilgk}(b) and equations 
(\ref{nii}.b) -- (\ref{nii}.c), we see that
\begin{equation}\label{tdi}
\begin{array}{l}
\mathrm{a)}\hskip6pt2D\nnh_i\w\hh\sigma
=-3(F\hskip-3.2pt_{ik}\w\nh+F\hskip-3.2pt_{il}\w)\sigma
-2(\nnh-\nnh1\nh)^{i+k}(F\hskip-3.2pt_{ik}\w\nh-F\hskip-3.2pt_{il}\w)\tau\hh,\\
\mathrm{b)}\hskip6pt4D\nnh_i\w\tau
=-3(\nnh-\nnh1\nh)^{i+k}(F\hskip-3.2pt_{ik}\w\nh-F\hskip-3.2pt_{il}\w)\sigma
-2[4F\hskip-3.2pt_{ij}\w\nh+F\hskip-3.2pt_{ik}\w\nh+F\hskip-3.2pt_{il}\w]\tau
\hh,\\
\mathrm{c)}\hskip6ptD\nnh_k\w\sigma\nnh=\nnh-3G\nnh_k\w\sigma\hh,\hskip7pt
D\nnh_k\w\nh\tau\nnh=\nnh-\hh(2F\hskip-3.2pt_{kl}\w\nnh\nh
+\nnh G\nh_k\w)\tau\hh,\hskip7pt
D\nnh_k\w\lambda\nnh=\nnh[(\nnh-\nnh1\nh)^k\hn\mu\nnh-\nnh2\lambda]\hh 
G\nh_k\w\hh,
\end{array}
\end{equation}
the last equality in (\ref{tdi}.c) being due to (\ref{gkl}) and (\ref{lem}.a). 
As $\,D\nnh_i\w\lambda=0$, cf.\ Lemma \ref{dilgk}(a), from Lemma \ref{dilgk}(d) 
and (\ref{lem}.a) we get
\begin{equation}\label{fik}
\begin{array}{l}
F\hskip-3.2pt_{ik}\w\nh+F\hskip-3.2pt_{il}\w\hs=\,D\nnh_i\w\nnh\log|\mu^2\nh-4\lambda^2|\,=
\,2(\mu^2\nh-4\lambda^2)^{-\nnh1}\nh\mu\hs D\nnh_i\w\mu\hh,\phantom{_{j_j}}\\
(\nnh-\nnh1\nh)^{i+k}(F\hskip-3.2pt_{ik}\w\nh-F\hskip-3.2pt_{il}\w)\,
=\,D\nnh_i\w\nnh\log|[\mu+2(\nnh-\nnh1\nh)^i\lambda]/[\mu-2(\nnh-\nnh1\nh)^i\lambda]|
\phantom{^{1^1}}\\
\phantom{(\nnh-\nnh1\nh)^{i+k}(F\hskip-3.2pt_{ik}\w\nh-F\hskip-3.2pt_{il}\w)\,}=\,\,
-\hs4(\nnh-\nnh1\nh)^i(\mu^2\nh-4\lambda^2)^{-\nnh1}\nh\lambda\hs D\nnh_i\w\mu\hh.
\end{array}
\end{equation}
The functions $\,P\hskip-2.7pt_i\w,Q\nh_i\w$ given by 
$\,2P\hskip-2.7pt_i\w\nh=(\nnh-\nnh1\nh)^k(F\hskip-3.2pt_{ik}\w\nh-F\hskip-3.2pt_{il}\w)\,$ 
and $\,2Q\nh_i\w\nh=-(F\hskip-3.2pt_{ik}\w\nh+F\hskip-3.2pt_{il}\w)$ are 
clearly independent of the choice of $\,k,l\,$ with $\,\{k,l\}=\{3,4\}$. Also,
\begin{equation}\label{dst}
\begin{array}{rl}
\mathrm{i)}&
F\hskip-3.2pt_{ik}\w\hs=\,(\nnh-\nnh1\nh)^k\nnh P\hskip-2.7pt_i\w\hs-\,Q\nh_i\w\hh,
\hskip11pt-\hh F\hskip-3.2pt_{il}\w\hs=\,(\nnh-\nnh1\nh)^k\nnh P\hskip-2.7pt_i\w\hs
+\,Q\nh_i\w\hh,\\
\mathrm{ii)}&
D\nnh_i\w\hh\sigma=3Q\nh_i\w\sigma-2(\nnh-\nnh1\nh)^i\nnh P\hskip-2.7pt_i\w\tau\hh,\\
\mathrm{iii)}&D\nnh_i\w\tau=(Q\nh_i\w\nh-2F\hskip-3.2pt_{ij}\w)\tau
-3(\nnh-\nnh1\nh)^i\nnh P\hskip-2.7pt_i\w\sigma\nnh/\nh2\hh,\\
\mathrm{iv)}&
D\nnh_i\w\mu=2\lambda P\hskip-2.7pt_i\w\nh+\mu Q\nh_i\w\hh,\hskip12pt
\mathrm{v)}\hskip6pt
\mu P\hskip-2.7pt_i\w\nh+2\lambda Q\nh_i\w\nh=0.
\end{array}
\end{equation}
Namely, (\ref{dst}.i) is obvious, and (\ref{dst}.ii) -- (\ref{dst}.iii) follow 
from (\ref{tdi}). On the other hand, 
$\,(\nnh-\nnh1\nh)^k D\nnh_i\w\mu=D\nnh_i\w\lambda_k\w\nh
=(\lambda-\lambda_k\w)F\hskip-3.2pt_{ik}\w$ by (\ref{lem}.a), 
Lemma \ref{dilgk}(a), (\ref{nii}.b) and (\ref{loe}.a); simultaneously, 
(\ref{lem}.a) and (\ref{dst}.i) give 
$\,(\lambda-\lambda_k\w)F\hskip-3.2pt_{ik}\w\nh
=[2\lambda-(\nnh-\nnh1\nh)^k\hn\mu]
[(\nnh-\nnh1\nh)^k\nnh P\hskip-2.7pt_i\w\nh-Q\nh_i\w]$. Hence 
$\,D\nnh_i\w\mu
=[2\lambda-(\nnh-\nnh1\nh)^k\hn\mu][P\hskip-2.7pt_i\w\nh-(\nnh-\nnh1\nh)^kQ\nh_i\w]$, 
that is, $\,D\nnh_i\w\mu=2\lambda P\hskip-2.7pt_i\w\nh+\mu Q\nh_i\w  
-(\nnh-\nnh1\nh)^k(\mu P\hskip-2.7pt_i\w\nh+2\lambda Q\nh_i\w)$, which holds for both 
$\,k\in\{3,4\}$, thus implying (\ref{dst}.iv) and (\ref{dst}.v).
Next,
\begin{equation}\label{dip}
\begin{array}{l}
D\nnh_i\w\nh P\hskip-2.7pt_i\w\nh=2P\hskip-2.7pt_i\w Q\nh_i\w\nh-F\hskip-4pt_{ji}\w P\hskip-2.7pt_j\w\nh
-(\nnh-\nnh1\nh)^i\tau-\mu/2\\
\phantom{D\nnh_i\w\nh P\hskip-2.7pt_i\w\hs}
-\nnh(\nnh-\nnh1\nh)^k(D\nnh_k\w\nh G\nh_k\w\nh-D\hn_l\w\nh G\nh_l\w\nh+
G\nh_k^{\hs2}-G\nh_l^{\hs2}+G\nh_l\w F\hskip-3.2pt_{lk}\w\nh
-G\nh_k\w F\hskip-3.2pt_{kl}\w)/2
\hh,\phantom{_{j_j}}\\
D\nnh_i\w Q\nh_i\w\nh=P\nnh_i^{\,2}+Q\nh_i^{\hs2}-F\hskip-4pt_{ji}\w Q\nh_j\w\nh
-(\sigma-\mathrm{s}/6)/2\phantom{^{1^1}}\\
\phantom{D\nnh_i\w Q\nh_i\w}
+\hn(D\nnh_k\w\nh G\nh_k\w\nh
+D\hn_l\w\nh G\nh_l\w\nh+G\nh_k^{\hs2}+G\nh_l^{\hs2}+G\nh_l\w F\hskip-3.2pt_{lk}\w\nh
+G\nh_k\w F\hskip-3.2pt_{kl}\w)/2\hh.
\end{array}
\end{equation}
This is immediate from the definitions of $\,P\hskip-2.7pt_i\w$ and
$\,Q\nh_i\w$, 
(\ref{dst}.i), and the second conclusion of Remark \ref{difij}, the right-hand 
side of which is equal, by (\ref{lem}.a) and (\ref{lem}.b), to 
$\,-\hs(\nnh-\nnh1\nh)^{i+k}\tau+[\sigma-(\nnh-\nnh1\nh)^k\hn\mu
-\mathrm{s}/6]/2$. Furthermore,
\begin{equation}\label{gkz}
\begin{array}{l}
\lambda\hs\,\mathrm{\ is\ nonzero\ everywhere\ in\ some\ dense\ open\ set\
and,\nnh\ whenever}\\
\mathrm{it\ is\ constant,\ }F\hskip-3.2pt_{ki}\w\nh
=G\nh_k\w\nh=D\nnh_k\w\sigma=0\,\mathrm{\ for\ all\ 
}\,(i,k)\in\{1,2\}\nh\times\nnh\{3,4\}\hh.\\
\end{array}
\end{equation}
In fact, the second claim is obvious from (\ref{gkl}) and
Lemma \ref{dilgk}(b). 
As for the first, let $\,\lambda=0\,$ on a nonempty open set. On such a set, 
with (\ref{ijk}), the last equality in (\ref{fik}) gives 
$\,F\hskip-3.2pt_{ik}\w\nh=F\hskip-3.2pt_{il}\w$ or, in the notation of 
(\ref{dst}), $\,P\hskip-2.7pt_i\w\nh=P\hskip-2.7pt_j\w\nh=0$. Since 
$\,G\nh_k\w\nh=G\nh_l\w\nh=0\,$ by (\ref{gkz}), the first formula in 
(\ref{dip}) now reads $\,(\nnh-\nnh1\nh)^i\tau=-\mu/2$. This being the case 
for both $\,i\in\{1,2\}$, it follows that $\,\tau=\mu=0$. Consequently, as 
$\,\lambda=0$, (\ref{lem}.a) yields $\,\lambda_3\w\nh=\lambda_4\w\nh=0$, 
contradicting (\ref{loe}.a).

Subtracting from (\ref{dip}) its version obtained by switching $\,i,j\,$ and 
noting that $\,(\nnh-\nnh1\nh)^j\nh=-(\nnh-\nnh1\nh)^i\nh$, we get
\begin{equation}\label{dpi}
\begin{array}{l}
D\nnh_i\w\nh P\hskip-2.7pt_i\w\nh-2P\hskip-2.7pt_i\w Q\nh_i\w\nh
-F\hskip-3.2pt_{ij}\w P\hskip-2.7pt_i\w-\,(D\nnh_j\w P\hskip-2.7pt_j\w\nh-
2P\hskip-2.7pt_j\w\nh Q\nh_j\w\nh-F\hskip-4pt_{ji}\w P\hskip-2.7pt_j\w)\,
=\,-\nh2(\nnh-\nnh1\nh)^i\tau\hh,\phantom{_{j_j}}\\
D\nnh_i\w Q\nh_i\w\nh-P\nnh_i^{\,2}-Q\nh_i^{\hs2}
-F\hskip-3.2pt_{ij}\w Q\nh_i\w\,
=\,D\nnh_j\w Q\nh_j\w\nh\nh-P\nnh_j^{\,2}-Q\nh_j^{\hs2}
-F\hskip-4pt_{ji}\w Q\nh_j\w\hh.\phantom{^{1^1}}\\
\end{array}
\end{equation}
Since $\,\lambda_i\w\nh=\lambda_j\w\nh=\lambda\,$
and $\,F\hskip-3.2pt_{ki}\w\nh=F\hskip-3.2pt_{kj}\w\nh=G\nh_k\w$, cf.\ 
(\ref{loe}.a) and Lemma \ref{dilgk}(b), the second displayed equation in
Theorem~\ref{algvr} with $\,i\,$ and $\,l\,$ switched takes, by
(\ref{dst}.i) and (\ref{lem}.a) -- (\ref{lem}.c), the form
\begin{equation}\label{tfg}
\begin{array}{l}
[(2F\hskip-3.2pt_{kl}\w\hs-\,G\nh_k\w)P\hskip-2.7pt_i\w\hs
-\,(\nnh-\nnh1\nh)^k\nh Q\nh_i\w G\nh_k\w]
[4\lambda\tau+3(\nnh-\nnh1\nh)^i\mu\sigma]\\
\hskip110pt=\,\hs[(\nnh-\nnh1\nh)^k\nh P\hskip-2.7pt_i\w G\nh_k\w\hs
-\,Q\nh_i\w G\nh_k\w][4(\nnh-\nnh1\nh)^k\nh\lambda\tau\,-\,2\mu\tau]\hh.
\end{array}
\end{equation}

\section{Proof of Theorem~\ref{ricev}({\rm c}): part three}\label{mc}
We continue making the same assumptions as in Lemma \ref{dilgk}.
\begin{lemma}\label{gtgfz}
One has $\,G\nh_3\w G\nh_4\w\nh=0\,$ everywhere.
\end{lemma}
\begin{proof}Suppose that, on the contrary, neither $\,G\nh_3\w$ nor 
$\,G\nh_4\w$ is identically $\,0$. As $\,\lambda\ne0\,$ somewhere by 
(\ref{gkz}), 
using (\ref{ana}) we may restrict our discussion to a dense open 
set, at every point of which
\begin{equation}\label{lgn}
\lambda G\nh_3\w G\nh_4\w\,\ne\,\,0\hh.
\end{equation}
With (\ref{ijk}), for some functions $\,C_3\w,C_4\w,E_3\w,E_4\w,\Pi\,$ and 
$\,\varepsilon=\mathrm{sgn}\hh\lambda\in\{1,-\nnh1\}$,
\begin{equation}\label{lee}
\begin{array}{rlrl}
\mathrm{i)}&\hskip-3ptF\hskip-3.2pt_{kl}\w\nnh
=D\nnh_k\w\nnh\log|\lambda|\hh^{-\nh1\nh/\nh4}\nh,&\hskip-8pt
\mathrm{ii)}&\hskip-3ptD\nnh_k\w\nh C_l\w\nh=0\hh,\\
\mathrm{iii)}&\hskip-3pt\lambda=\varepsilon(C_3\w\nh+C_4\w)^2,&\hskip-8pt
\mathrm{iv)}&\hskip-3pt\varepsilon(C_3\w\nh+C_4\w)>0\hh,\\
\mathrm{v)}&\hskip-3pt\lambda-\lambda_k\w\nnh
=4\varepsilon(C_3\w\nh+C_4\w)C_k\w,&\hskip-8pt
\mathrm{vi)}&\hskip-3pt\lambda_k\w\nnh
=\hs\varepsilon(C_3\w\nh+C_4\w)(C_l\w\nh-3C_k\w)\hh,\\
\mathrm{vii)}&\hskip-3ptD\nnh_i\w(C_3\w\nh+C_4\w)=0\hh,&\hskip-8pt
\mathrm{viii)}&\hskip-3ptF\hskip-3.2pt_{ik}\w=\,-D\nnh_i\w\log|C_k\w|\hh,\\
\mathrm{ix)}&\hskip-3ptC_k\w\nh\nnh=E_k\w\nh+(\nnh-\nnh1\nh)^k\Pi\hh,&\hskip-8pt
\mathrm{x)}&\hskip-3ptD\nnh_k\w E_l\w\nh=D\nnh_i\w E_k\w\nh=D\nnh_k\w \Pi
=0\hh,\\
\mathrm{xi)}&\hskip-3ptD\nnh_i\w C_k\w\nnh=(\nnh-\nnh1\nh)^kD\nnh_i\w \Pi\hh,&\hskip-8pt
\mathrm{xii)}&\hskip-3pt\Pi\hskip7pt\mathrm{is\ nonconstant.}
\end{array}
\end{equation}
In fact, (\ref{lgn}) and the second claim in Lemma \ref{dilgk}(h) imply 
(\ref{lee}.i). Setting 
$\,C_k\w\nnh=|\lambda|^{-\nnh1/2}(\lambda-\lambda_k\w)/4\,$ and using 
(\ref{loe}.b) we obtain (\ref{lee}.ii) (since (\ref{nii}.b), (\ref{lee}.i) 
and (\ref{loe}.b) give $\,4\lambda D\nnh_k\w\lambda_l\w
=\hs(\lambda_l\w\nnh-\lambda_k\w)D\nnh_k\w\lambda
=\hs2(\lambda+\lambda_l\w)D\nnh_k\w\lambda$), as well as (\ref{lee}.iii) 
-- (\ref{lee}.v), while subtraction of (\ref{lee}.v) from (\ref{lee}.iii) 
yields (\ref{lee}.vi). Equality (\ref{lee}.vii) is a consequence of 
(\ref{lee}.iii) and Lemma \ref{dilgk}(a). Similarly, (\ref{lee}.v), 
(\ref{lee}.vii) and Lemma \ref{dilgk}(d) lead to (\ref{lee}.viii). For 
$\,\partial\nh_i\w,\partial\nh_k\w$ chosen as in Remark \ref{mltpl}, the 
relations $\,\partial\nh_k\w\nh C_l\w\nh=\partial\nh_i\w(C_k\w\nh+C_l\w)=0$, 
due to (\ref{lee}.ii) and (\ref{lee}.vii), show that the function 
$\,\partial\nh_i\w\nh C_k\w\nh=-\hs\partial\nh_i\w\nh C_l\w$ is constant along 
$\,\mathcal{H}=\mathrm{Span}\hs\{e\nh_3\w,e\nh_4\w\}$, and so 
$\,C_k\w$, for $\,k=3,4$, equals a function constant along $\,\mathcal{H}$ 
plus a function constant along 
$\,\mathcal{V}=\mathrm{Span}\hs\{e\nh_1\w,e\nh_2\w\}$. Combined with 
(\ref{lee}.vii), this proves the existence of $\,E_3\w,E_4\w$ and $\,\Pi\,$ 
satisfying (\ref{lee}.ix) -- (\ref{lee}.xi). Finally, if $\,\Pi\,$ were 
constant, (\ref{lee}.xi) and (\ref{lee}.vi) would give 
$\,D\nnh_i\w\nh C_k\w\nnh=D\nnh_i\w\lambda_k\w\nnh=0\,$ for both $\,i=1,2\,$ and 
both $\,k=3,4$, contradicting the first claim in Lemma \ref{dilgk}(h).

Due to (\ref{gkl}) and (\ref{lee}.ii) -- (\ref{lee}.v), Lemma \ref{dilgk}(b) 
gives $\,D\nnh_k\w\nh\log|\sigma|\hh^{-\nh1\nh/\nh3}\hskip-2.9pt
=D\nnh_k\w\nh\log|C_k\w|\hh^{-\nh1\nh/\nh2}\nh$, where 
$\,C_k\w\nh\ne0\,$ in view of (\ref{lee}.v) and (\ref{loe}.a). Thus, by 
(\ref{lee}.ii), 
$\,D\nnh_k\w(|C_3\w\nh C_4\w|\hh^{-\nh3\nh/\nh2}\nh\sigma)=0$, and so, for the 
function $\,L=|C_3\w\nh C_4\w|\hh^{-\nh3\nh/\nh2}\nh\sigma$,
\begin{equation}\label{dkl}
\sigma=|C_3\w\nh C_4\w|\hh^{3\nh/\nh2}\nh L\hskip7pt\mathrm{and}\hskip7pt
D\nnh_k\w L=0\hskip7pt\mathrm{whenever}\hskip7ptk\in\{3,4\}\hh.
\end{equation}
By (\ref{lee}.i), 
$\,2F\hskip-3.2pt_{kl}\w\nnh
=D\nnh_k\w\nnh\log|\lambda|\hh^{-\nh1\nh/\nh2}\nh$. Adding this to the second 
equality in Lemma \ref{dilgk}(b), one gets $\,D\nnh_k\w\hn S=0\,$ for 
$\,S=(\nnh-\nnh1\nh)^{i+k}\nh|\lambda|\hh^{-\nh1\nh/\nh2}
\sigma\hh^{-\nh1\nh/\nh3}\nh(\sigma\nnh_{ik}\w\nh-\sigma\hskip-2.7pt_{jk}\w)$, with 
(\ref{ijk}), and $\,k=3,4$, where, in view of (\ref{skw}.b), $\,S\,$ does not 
depend on the choice of $\,i,j,k,l\,$ satisfying (\ref{ijk}). As (\ref{skw}.b) 
with $\,\sigma=\sigma\nnh_{ij}\w$ also yields 
$\,\sigma\nnh_{ik}\w\nh-\sigma\hskip-2.7pt_{jk}\w\nh=2\sigma\nnh_{ik}\w\nh+\sigma$, the 
definition of $\,S\,$ amounts to
\begin{equation}\label{tsi}
2\sigma\nnh_{ik}\w\nh=-\sigma
+(\nnh-\nnh1\nh)^{i+k}\nh|\lambda|\hh^{1\nh/\nh2}\sigma\hh^{1\nh/\nh3}\nh S\,
\mathrm{\ with\ (\ref{ijk})\ and\ }\,D\nnh_k\w\hn S=0\hh.
\end{equation}
Next, (\ref{nii}.c) and (\ref{skw}.b) give  
$\,D\nnh_i\w\hh\sigma=D\nnh_i\w\hh\sigma\nnh_{kl}\w
=(\sigma\nnh_{il}\w\nh-\sigma)F\hskip-3.2pt_{ik}\w\nh+
(\sigma\nnh_{ik}\w\nh-\sigma)F\hskip-3.2pt_{il}\w$. Replacing 
$\,\sigma\nnh_{ik}\w,\sigma\nnh_{il}\w$ and $\,F\hskip-3.2pt_{ik}\w,F\hskip-3.2pt_{il}\w$ with the 
expressions provided by (\ref{tsi}) and (\ref{lee}.viii), we easily verify 
that $\,2D\nnh_i\w\hh\sigma
=(\nnh-\nnh1\nh)^{i+k}\nh|\lambda|\hh^{1\nh/\nh2}\sigma\hh^{1\nh/\nh3}\nh 
SD\nnh_i\w\log|C_k\w/C_l\w|+3\sigma D\nnh_i\w\log|C_k\w\nh C_l\w|$. At the same 
time, from (\ref{dkl}), 
$\,2D\nnh_i\w\hh\sigma=2|C_k\w\nh C_l\w|\hh^{3\nh/\nh2}\nh D\nnh_i\w\nh L
+3\sigma D\nnh_i\w\nnh\log|C_k\w\nh C_l\w|$. As 
$\,|\lambda|\hh^{1\nh/\nh2}\nh=\varepsilon(C_k\w\nh+C_l\w)\,$ by 
(\ref{lee}.iii) -- (\ref{lee}.iv), while 
$\,D\nnh_i\w\hskip-2.4pt\log|C_k\w/C_l\w|
=(\nnh-\nnh1\nh)^k\nh(C_k^{-\nnh1}\nnh+C_l^{-\nnh1})D\nnh_i\w\nh \Pi\nh$, cf.\ 
(\ref{lee}.xi), and $\,\sigma\hh^{1\nh/\nh3}\hskip-1.7pt
=\nh|C_3\w\nh C_4\w|\hh^{1\nh/\nh2}\nh L^{1\nh/\nh3}$ from (\ref{dkl}), 
equating the two expressions for $\,2D\nnh_i\w\hh\sigma\,$ one gets, from 
(\ref{lee}.iv),
\begin{equation}\label{cki}
(C_k^{-\nnh1}\nnh+C_l^{-\nnh1})^2\nh
=2|D\nnh_i\w\nh L|/|SL\nnh^{\nh1\nh/\nh3}\hs D\nnh_i\w\nh \Pi|
\end{equation}
on a nonempty open set $\,\,U''\nnh$, where (\ref{lee}.xii) allows us to
choose $\,i\in\{1,2\}$ so that $\,|D\nnh_i\w\nh \Pi|>0\,$ on $\,\,U''\nnh$.
Since, for $\,\partial\nh_i\w,\partial\nh_k\w$ as in Remark \ref{mltpl}, 
$\,|D\nnh_i\w\nh L|/|D\nnh_i\w\nh \Pi|
=|\partial\nh_i\w\nh L|/|\partial\nh_i\w\nh \Pi|$, and so $\,\partial\nh_k\w$ 
applied to the right-hand side of (\ref{cki}) yields $\,0\,$ in view of 
(\ref{tsi}), (\ref{dkl}) and (\ref{lee}.x), we have, from (\ref{lee}.ii), 
$\,0=D\nnh_k\w\nh(C_k^{-\nnh1}\nnh+C_l^{-\nnh1})
=-\hh C_k^{-\nnh2}D\nnh_k\w\nh C_k\w$. Thus, by (\ref{lee}.ii) -- 
(\ref{lee}.iii), $\,D\nnh_k\w\nh\lambda=0$. Combined with (\ref{gkl}) and 
(\ref{lgn}), this leads to a contradiction, proving that 
$\,G\nh_3\w G\nh_4\w\nh=0$.
\end{proof}
Assume (\ref{ijk}). In view of (\ref{nii}.a), for any function $\,\theta\,$ we 
have
\begin{equation}\label{dij}
(\nabla\nnh d\theta)(e\nh_i\w,e\nh_j\w)\,
=\,D\nnh_i\w\nh D\nnh_j\w\theta\,-\,F\hskip-4pt_{ji}\w\nh D\nnh_i\w\theta\hh.
\end{equation}
By Lemma \ref{dilgk}(d), $\,F\hskip-3.2pt_{ik}\w\nh=D\nnh_i\w\theta\nh_k\w$, where 
$\,\theta\nh_k\w\nh=\log|\lambda_k\w\nnh-\hn\lambda|^{-\nh1}\nh$. Now 
(\ref{dif}.i) reads $\,D\nnh_i\w\nh D\nnh_j\w\theta\nh_k\w\nh
=(F\hskip-4pt_{ji}\w-D\nnh_j\w\theta\nh_k\w)D\nnh_i\w\theta\nh_k\w$, and so, 
with $\,\theta=\theta\nh_k\w$, from (\ref{dij}), 
$\,(\nabla\nnh d\theta)(e\nh_i\w,e\nh_j\w)+(D\nnh_j\w\theta)D\nnh_i\w\theta
=0$, 
that is, $\,(\nabla\nnh de^\theta)(e\nh_i\w,e\nh_j\w)=0$. In other words, for 
$\,\psi\nh_k\w\nh=(\lambda_k\w\nnh-\lambda)^{-\nh1}\nh$, 
\begin{equation}\label{ndp}
(\nabla\nnh d\psi\nh_k\w)(e\nh_i\w,e\nh_j\w)\,
=\,(\nabla\nnh d\psi\nh_l\w)(e\nh_i\w,e\nh_j\w)\,
=\,0\hskip10pt\mathrm{if}\hskip6pt\{k,l\}=\{3,4\}\hh.
\end{equation}
When $\,k,l\,$ with $\,\{k,l\}=\{3,4\}\,$ are fixed, setting 
$\,\psi=\psi\nh_k\w$ we get $\,\psi\nh_l\w=-\hh(4\lambda+1/\psi)^{-\nh1}$ from 
(\ref{loe}.b).Now (\ref{dij}) applied to $\,\theta=\psi\nh_l\w$ and 
$\,\theta=\psi\nh_k\w$ yields 
$\,-\hh(4\lambda\psi+1)^3(\nabla\nnh d\psi\nh_l\w)(e\nh_i\w,e\nh_j\w)
=(4\lambda\psi+1)(\nabla\nnh d\psi)(e\nh_i\w,e\nh_j\w)
-8\lambda(D\nnh_j\w\nh\psi)D\nnh_i\w\psi$, with $\,\psi=\psi\nh_k\w$, as 
$\,D\nnh_i\w\lambda=0\,$ (see Lemma \ref{dilgk}(a)). In view of (\ref{ndp}), 
this gives $\,(D\nnh_j\w\nh\psi\nh_k\w)D\nnh_i\w\psi\nh_k\w\nh=0$, since 
$\,\lambda\ne0\,$ by (\ref{gkz}). Using (\ref{lem}.a) and, again, the 
relation $\,D\nnh_i\w\lambda=0\,$ in Lemma \ref{dilgk}(a), we see that
\begin{equation}\label{ldj}
(D\nnh_j\w\hh\mu)D\nnh_i\w\mu\hs=\,0\hskip8pt
\mathrm{if}\hskip6pt\{i,j\}=\{1,2\}\hh.
\end{equation}
From now on 
the symbols $\,\bullet\,$ stand for any indices for which the 
expression makes sense. For instance, $\,j,\bullet,\bullet\,$ in 
$\,H\hskip-2.5pt_{j\bullet\bullet}\w$ are mutually distinct.
\begin{lemma}\label{fjbuz}
One can fix\/ $\,i,j\in\{1,2\}\,$ and a nonempty open connected set\/ 
$\,\,U''\nnh$, so that, whenever\/ $\,\{k,l\}=\{3,4\}$,
\begin{enumerate}
  \def\theenumi{{\rm\alph{enumi}}}
\item $D\nnh_j\w\hh\mu=0\ne\hs D\nnh_i\w\mu\,\,$ everywhere,
\item $D\nnh_i\w\lambda\,=\,D\nnh_j\w\lambda\,=\,D\nnh_i\w\nh G\nh_k\w\hs
=\,D\nnh_j\w\nh G\nh_k\w\hs=\,G\nh_l\w G\nh_k\w\hs=\,0$,
\item $F\hskip-4pt_{j\bullet}\w\hs=\,0\,\,$ and 
$\,\,H\hskip-2.5pt_{j\bullet\bullet}\w\hs=H\nnh_{\bullet j\bullet}\w\hs=\,0$,
\item $D\nnh_j\w\lambda_k\w\hs=\,D\nnh_j\w\sigma\hs
=\,D\nnh_j\w\sigma\nnh_{\bullet\bullet}\w\hs=\,D\nnh_j\w F\hskip-3.2pt_{ij}\w\hs=\,0$.
\end{enumerate}
\end{lemma}
\begin{proof}
Lemmas \ref{dilgk}(a) and \ref{gtgfz} imply (b) for $\,i,j\in\{1,2\}$. 
Thus, by real-an\-a\-lyt\-ic\-i\-ty, cf.\ (\ref{ana}), and 
(\ref{ldj}), 
$\,D\nnh_j\w\hh\mu=0\,$ everywhere for one choice of $\,j\,$ (but not for both, 
since that would, in view of (b) and (\ref{lem}.a), yield 
$\,D\nnh_i\w\lambda_k\w\nnh=0$ whenever $\,i\le2<k$, contradicting 
Lemma \ref{dilgk}(h)). Now (a) follows. From (a), (b), (\ref{lem}.a) and 
Lemma \ref{dilgk}(d) we in turn obtain 
$\,D\nnh_j\w\lambda_k\w\nh=F\hskip-4pt_{jk}\w\nh=0$, although, as we just saw, 
$\,D\nnh_i\w\lambda_k\w\nh\ne0$. (Here and below, $\,i,j\in\{1,2\}$ are 
fixed, so as to satisfy (a), and $\,\{k,l\}=\{3,4\}$.) Next, (\ref{ndp}) and 
(\ref{dij}), for 
$\,\theta=\psi\nh_k\w\nh=(\lambda_k\w\nnh-\lambda)^{-\nh1}\nh$, give 
$\,D\nnh_i\w\nh D\nnh_j\w\theta=F\hskip-4pt_{ji}\w\nh D\nnh_i\w\theta$, while 
$\,D\nnh_j\w\theta=0\ne D\nnh_i\w\theta\,$ by (b) with 
$\,D\nnh_j\w\lambda_k\w\nh=0\ne D\nnh_i\w\lambda_k\w$. Hence 
$\,F\hskip-4pt_{ji}\w\nh=0$, that is, $\,F\hskip-4pt_{j\bullet}\w\nh=0$, as required in 
(c), and the rest of (c) is obvious from (\ref{fji}). Also, combining (c) with 
the equality 
$\,(\sigma\hskip-2.7pt_{jk}\w\nh-\sigma\nnh_{ik}\w)(D\nnh_j\w F\hskip-3.2pt_{ij}\w\nh
-D\nnh_i\w F\hskip-4pt_{ji}\w)=3\hs(H\nnh_{lk}\w\nh-H\nnh_{kl}\w)\hs\sigma\nnh_{kl}\w$, which is 
a special case of (\ref{dif}.ii) or, respectively, with the expression for 
$\,D\nnh_j\w\sigma\nnh_{\bullet\bullet}\w$ resulting from (\ref{nii}.c), we 
obtain (d), since $\,\sigma\hskip-2.7pt_{jk}\w\nh\ne\sigma\nnh_{ik}\w$ and 
$\,\sigma=\sigma\nnh_{ij}\w\nh=\sigma\nnh_{kl}\w$ according to (\ref{skw}.a) 
and (\ref{skw}.b), while we already saw that $\,D\nnh_j\w\lambda_k\w\nh=0$.
\end{proof}

\section{Proof of Theorem~\ref{ricev}({\rm c}): conclusion}\label{lc}
Fixing $\,i,j\,$ with $\,\{i,j\}=\{1,2\}\,$ and $\,\,U''$ as in
Lemma \ref{fjbuz}, for 
$\,\mu,\sigma,\tau,P\hskip-2.7pt_i\w$, $\,Q\nh_i\w$ appearing in (\ref{lem})
-- 
(\ref{dst}), let us set $\,P\nh=P\hskip-2.7pt_i\w$ and $\,Q=Q\nh_i\w$. We have
\begin{equation}\label{mpt}
\begin{array}{rl}
\mathrm{i)}&\mu P\,+\,2\lambda\hh Q\,=\,0\hh,\\
\mathrm{ii)}&P\nh\nnh_j\w\hs=\,Q\nh_j\w\hs=\,F\hskip-4pt_{ji}\w\hs=
\,D\nnh_i\w\lambda\,=\,D\nnh_i\w\nh G\nh_k\w\hs=\,0\hh,\\
\mathrm{iii)}&D\nnh_i\w\nh P\,=\,2PQ\,+\,P\nh F\hskip-3.2pt_{ij}\w\hs
-\,2(\nnh-\nnh1\nh)^i\tau\hh,\\
\mathrm{iv)}&D\nnh_i\w\nh Q\,
=\,P^2\hs+\,Q^2\hs+\,Q\nh F\hskip-3.2pt_{ij}\w\hh.\\
\mathrm{v)}&2\lambda\hh(P^2-\,Q^2)\,=\,(\nnh-\nnh1\nh)^i\mu\tau\hh,\\
\mathrm{vi)}&2(\nnh-\nnh1\nh)^i(2F\hskip-3.2pt_{ij}\w\mu\,-5\lambda P)\hh\tau\,
=\,3\lambda Q\hh\sigma\hh.
\end{array}
\end{equation}
In fact, (\ref{dst}.v) amounts to (\ref{mpt}.i); the definitions of
$\,P\nh\nnh_j\w,Q\nh_j\w$ preceding (\ref{dst}) and
Lemma~\ref{fjbuz}(b)\hs--\hs(c), and (\ref{dst}.v) yield (\ref{mpt}.ii); also,
(\ref{mpt}.iii) -- (\ref{mpt}.iv) are trivial consequences of (\ref{dpi}) and
(\ref{mpt}.ii). Next, (\ref{mpt}.v) follows if one applies $\,D\nnh_i\w$ to
(\ref{mpt}.i), uses (\ref{dst}.iv), (\ref{mpt}.iii), (\ref{mpt}.ii),
(\ref{mpt}.iv), and simplifies, with the aid of (\ref{mpt}.i), the resulting
equality
\[
\,(F\hskip-3.2pt_{ij}\w\nh+3Q)(\mu P+2\lambda\hh Q)
+2[2\lambda\hh(P^2\nh-Q^2)-(\nnh-\nnh1\nh)^i\mu\tau]=0\hh.
\]
Finally, by 
(\ref{mpt}.ii) -- (\ref{mpt}.iv) and  (\ref{dst}.iii) -- (\ref{dst}.iv), 
$\,D\nnh_i\w$ applied to (\ref{mpt}.v) shows that the left-hand side of 
(\ref{mpt}.vi) coincides with
\[
3\lambda Q\hh\sigma
-2(F\hskip-3.2pt_{ij}\w\nh+Q)[2\lambda\hh(P^2\nh-Q^2)-(\nnh-\nnh1\nh)^i\mu\tau]
-3(\mu P+2\lambda\hh Q)\hh\sigma\nnh/\nh2\hh.
\]
Thus, (\ref{mpt}.v) and 
(\ref{mpt}.i) give (\ref{mpt}.vi). By (\ref{loe}.a) and 
Lemma \ref{dilgk}(b),
\begin{equation}\label{cnv}
\lambda_i\w\nh=\lambda_j\w\nh=\lambda\,\mathrm{\ and\ 
}\,F\hskip-3.2pt_{ki}\w\nh=F\hskip-3.2pt_{kj}\w\nh=G\nh_k\w\mathrm{,\ while\ 
}\,P\hskip-2.7pt_i\w\nh=P\,\mathrm{\ and\ }\,Q\nh_i\w\nh=Q\hh.
\end{equation}
Now, from (\ref{dst}.i) and (\ref{cnv}), along with (\ref{dif}.i) and 
(\ref{mpt}.iii) -- (\ref{mpt}.iv),
\begin{equation}\label{fpq}
\begin{array}{l}
F\hskip-3.2pt_{ik}\w\nh=(\nnh-\nnh1\nh)^k\nnh P-\,Q\hh,\hskip34pt
D\nnh_k\w F\hskip-3.2pt_{ij}\w\nh
=[(\nnh-\nnh1\nh)^k\nnh P\nh-\hh Q-\nh F\hskip-3.2pt_{ij}\w]\hh G\nh_k\w\hh,\\
D\nnh_i\w F\hskip-3.2pt_{ik}\w\nh=(\nnh-\nnh1\nh)^k[2PQ
+P\nh F\hskip-3.2pt_{ij}\w\nh-2(\nnh-\nnh1\nh)^i\tau]
-(P^2\nh+\,Q^2\nh+Q\nh F\hskip-3.2pt_{ij}\w)\hh,\\
D\nnh_i\w F\hskip-3.2pt_{kl}\w\nh=
[(\nnh-\nnh1\nh)^k\nnh P\nh+\hh Q](F\hskip-3.2pt_{kl}\w\nh-G\nh_k\w)\hh.
\end{array}
\end{equation}
By Lemma \ref{gtgfz}, restricting our discussion to an nonempty open subset 
of $\,\,U''\nnh$, we may fix $\,k,l\,$ with $\,\{k,l\}=\{3,4\}\,$ and 
$\,G\nh_l\w\nh=0\,$ everywhere. Then
\begin{equation}\label{eor}
\begin{array}{l}
\mathrm{either\ }\hs G\nh_k\w\nh\mathrm{\ vanishes\ identically\ on\ 
}\,U'\nnh,\mathrm{\ or\ }\hs G\nh_k\w\nh\ne\hs0\\
\mathrm{at\hs\ all\hs\ points\hs\ of\hs\ some\hs\ open\hs\ 
dense\hs\ subset\hs\ of\hs\ }\,\,U'\nnh,
\end{array}
\end{equation}
as a consequence of real-an\-a\-lyt\-ic\-i\-ty, cf.\ (\ref{ana}). Furthermore,
\begin{equation}\label{dfj}
\begin{array}{rl}
\mathrm{i)}&D\nnh_i\w F\hskip-3.2pt_{ij}\w\hs
=\,-\hs(F\hskip-3.2pt_{ij}^{\hskip3.3pt2}\,+\,G\nh_k^{\hs2}\hs
\,+\hs\sigma\,+\,\lambda\,+\,\mathrm{s}/12)\hh,\\
\mathrm{ii)}&P\nh F\hskip-3.2pt_{ij}\w\hs
=\,(\nnh-\nnh1\nh)^i\tau\,-\,\mu/2\hh,\\
\mathrm{iii)}&Q\nh F\hskip-3.2pt_{ij}\w\hs+\,(\sigma\hs-\mathrm{s}/6)/2\,
=\,G\nh_k\w F\hskip-3.2pt_{kl}\w\hh.
\end{array}
\end{equation}
Namely, (\ref{dfj}.i) is due to the first conclusion of Remark~\ref{difij} 
with $\,F\hskip-4pt_{ji}\w\nh=0$, cf.\ (\ref{mpt}.ii), and $\,G\nh_l\w\nh=0$. 
To prove (\ref{dfj}.ii) -- (\ref{dfj}.iii) we begin by observing that 
(\ref{dip}) and (\ref{mpt}.ii) -- (\ref{mpt}.iv) easily give
\[
\begin{array}{rl}
\mathrm{(a)}&D\nnh_k\w\nh G\nh_k\w\hs+\,G\nh_k^{\hs2}\,
+\,G\nh_k\w F\hskip-3.2pt_{kl}\w\,=
\,\,2\hs Q\nh F\hskip-3.2pt_{ij}\w\hs+\,\sigma\,-\,\mathrm{s}/6\hh,\\
\mathrm{(b)}&(\nnh-\nnh1\nh)^k(D\nnh_k\w\nh G\nh_k\w\hs+\,G\nh_k^{\hs2}\,
-\,G\nh_k\w F\hskip-3.2pt_{kl}\w)\,
=\,2\hh[P\nh F\hskip-3.2pt_{ij}\w\hs-\,(\nnh-\nnh1\nh)^i\tau]\,+\,\mu\hh,\\
\mathrm{(c)}&Q\nh F\hskip-3.2pt_{ij}\w\hs-\,G\nh_k\w F\hskip-3.2pt_{kl}\w\hs
+\,(\sigma\hs-\mathrm{s}/6)/2\,
=\,(\nnh-\nnh1\nh)^k\nh[P\nh F\hskip-3.2pt_{ij}\w\hs
-\,(\nnh-\nnh1\nh)^i\tau\,+\,\mu/2]\hh,
\end{array}
\]
(c) arising when one subtracts (b) multiplied by 
$\,(\nnh-\nnh1\nh)^k$ from (a). Next, applying (\ref{dif}.v) to the 
triple $\hs(k,i,j)\hs$ rather than $\hs(i,k,l)$, one obtains
$\,D\nnh_k\w D\nnh_i\w F\hskip-3.2pt_{ij}\w\nh
+(F\hskip-3.2pt_{ki}\w\nnh
+F\hskip-3.2pt_{kj}\w)D\nnh_i\w F\hskip-3.2pt_{ij}\w\nh
=F\hskip-3.2pt_{kj}\w D\nnh_i\w F\hskip-3.2pt_{ik}\w\nh
+(F\hskip-3.2pt_{ik}\w\nh-F\hskip-3.2pt_{ij}\w)H\nnh_{lj}\w$. Equivalently,
due to (\ref{cnv}), $\,D\nnh_k\w D\nnh_i\w F\hskip-3.2pt_{ij}\w\nh
=[(F\hskip-3.2pt_{ik}\w\nh-F\hskip-3.2pt_{ij}\w)F\hskip-3.2pt_{ik}\w\nh
+D\nnh_i\w F\hskip-3.2pt_{ik}\w\nh
-2D\nnh_i\w F\hskip-3.2pt_{ij}\w]\hh G\nh_k\w$, as 
(\ref{fji}) and (\ref{cnv}) give $\,H\nnh_{lj}\w\nh
=F\hskip-3.2pt_{ik}\w G\nh_k\w$. With $\,F\hskip-3.2pt_{ik}\w$, 
$\,D\nnh_i\w F\hskip-3.2pt_{ik}\w$, $\,D\nnh_i\w F\hskip-3.2pt_{ij}\w$
replaced by the expressions in (\ref{fpq}) and (\ref{dfj}.i), this shows that
$\,D\nnh_k\w D\nnh_i\w F\hskip-3.2pt_{ij}\w$ equals $\,2\hh G\nh_k\w$ times
\[
F\hskip-3.2pt_{ij}^{\hskip3.3pt2}+G\nh_k^{\hs2}\nh-(\nnh-\nnh1\nh)^{i+k}\tau
+\lambda+\hs\sigma+\mathrm{s}/12\hh. 
\]
Simultaneously, by (\ref{dfj}.i), $\,D\nnh_k\w D\nnh_i\w F\hskip-3.2pt_{ij}\w
=-D\nnh_k\w(F\hskip-3.2pt_{ij}^{\hskip3.3pt2}+G\nh_k^{\hs2}\nh
+\hs\sigma+\lambda+\mathrm{s}/12)\,$ which, evaluated via (\ref{fpq}), (a), 
(\ref{tdi}.c) and (\ref{cst}), equals $\,2\hh G\nh_k\w$ times
\[
F\hskip-3.2pt_{ij}^{\hskip3.3pt2}+G\nh_k^{\hs2}\nh
-Q\nh F\hskip-3.2pt_{ij}\w\nh+G\nh_k\w F\hskip-3.2pt_{kl}\w\nh
-(\nnh-\nnh1\nh)^k\nh[P\nh F\hskip-3.2pt_{ij}\w\nh+\mu/2]
+\lambda+(\sigma-\mathrm{s}/3)/2\hh.
\]
Equating the two displayed expression, one easily gets
\begin{equation}\label{pqf}
[Q\nh F\hskip-3.2pt_{ij}\w\nh-G\nh_k\w F\hskip-3.2pt_{kl}\w\nh
+(\sigma\hs-\mathrm{s}/6)/2+(\nnh-\nnh1\nh)^k\nh\{P\nh F\hskip-3.2pt_{ij}\w\nh
-(\nnh-\nnh1\nh)^i\tau+\mu/2\}]\hh G\nh_k\w\hs=\,0\hh.
\end{equation}
In the first case of (\ref{eor}), (a) -- (b) clearly yield (\ref{dfj}.ii) and 
(\ref{dfj}.iii). In the second case, the two equal sides of (c) are, by 
(\ref{pqf}), also each other's opposites, and so both vanish, proving 
(\ref{dfj}.ii) -- (\ref{dfj}.iii).
\begin{lemma}\label{nonzr}
Each of the following seven functions\hh{\rm:}
\[
(\nnh-\nnh1\nh)^k\nnh P\nh-Q\hh,\hskip13pt2(\nnh-\nnh1\nh)^i\tau+
3(\nnh-\nnh1\nh)^k\nh\sigma\hh,\hskip13pt\lambda\hh,\hskip13pt\mu\hh,
\hskip13ptP,\hskip13ptQ\hh,\hskip13ptD\nnh_i\w\mu
\]
is nonzero everywhere in some open dense subset of\/ $\,\,U'\nnh$.
\end{lemma}
\begin{proof}Due to real-an\-a\-lyt\-ic\-i\-ty, cf.\ (\ref{ana}), it suffices 
to show that none of the seven functions can vanish on a nonempty open subset 
$\,\,U''$ of $\,\,U'\nnh$. For $\,\lambda,D\nnh_i\w\mu$ (and, consequently, 
$\,\mu$) this is clear from (\ref{gkz}) and Lemma~\ref{fjbuz}(a). If $\,P$ 
were identically zero on $\,\,U''\nnh$, so would be $\,\tau$, and hence 
$\,\mu$, by (\ref{mpt}.iii) and and (\ref{dfj}.ii), contrary to what we just 
showed about $\,\mu$. Thus, $\,Q\ne0\,$ on $\,\,U''$ from (\ref{mpt}.i) with 
$\,\mu P\nh\ne0$. In view of (\ref{mpt}.i), vanishing of 
$\,(\nnh-\nnh1\nh)^k\nnh P\nh-Q\,$ on on $\,\,U''$ would give 
$\,2(\nnh-\nnh1\nh)^k\nnh\lambda P\nh=2\lambda Q=-\mu P$ on $\,\,U''$ and, 
as $\,P\nh\ne0$, the equality $\,\mu=-\nh2(\nnh-\nnh1\nh)^k\nnh\lambda\,$ 
would follow, even though $\,D\nnh_i\w\mu\ne0=D\nnh_i\w\lambda$, cf.\ 
(\ref{mpt}.ii). Finally, suppose that 
$\,2(\nnh-\nnh1\nh)^i\tau+3(\nnh-\nnh1\nh)^k\nnh\sigma=0\,$ on $\,\,U''\nnh$. 
Using, successively, (\ref{skw}.b), (\ref{lem}.c) and (\ref{lem}.b), we now 
get $\,\sigma\nnh_{ij}\w\nh-\sigma\nnh_{il}\w\nh
=\sigma\nnh_{ij}\w\nh+(\sigma\nnh_{ij}\w\nh+\sigma\nnh_{ik}\w)
=2\sigma+\sigma\nnh_{ik}\w\nh
=2\sigma+(\nnh-\nnh1\nh)^{i+k}\tau-\sigma\nnh/\nh2
=(\nnh-\nnh1\nh)^k\hn[2(\nnh-\nnh1\nh)^i\tau+3(\nnh-\nnh1\nh)^k\nh\sigma]/2
=0$, so that $\,\sigma\nnh_{ij}\w\nh=\sigma\nnh_{il}\w$, which contradicts 
(\ref{skw}.a).
\end{proof}
\begin{lemma}\label{occur}
The second case of\/ {\rm(\ref{eor})} cannot occur.
\end{lemma}
\begin{proof}Let us assume that, on the contrary, $\,G\nh_k\w\nh\ne0\,$ 
everywhere in a dense open set, and apply $\,D\nnh_i\w$ to (\ref{dfj}.iii), 
using (\ref{dfj}.i), (\ref{mpt}.iv), (\ref{dst}.ii) with (\ref{cnv}),
(\ref{cst}), (\ref{mpt}.ii) and (\ref{fpq}). We consequently get
\[
\begin{array}{l}
(\nnh-\nnh1\nh)^k\nh(G\nh_k\w\nh-F\hskip-3.2pt_{kl}\w)PG\nh_k\w\nh
+[Q\nh F\hskip-3.2pt_{ij}\w\nh-G\nh_k\w F\hskip-3.2pt_{kl}\w\nh
+(\sigma\hs-\mathrm{s}/6)/2]\hh Q\\
\hskip67pt+\,[P\nh F\hskip-3.2pt_{ij}\w\nh-(\nnh-\nnh1\nh)^i\tau+\mu/2]P\nh
-(\mu P\nh+2\lambda\hh Q)/2=0
\end{array}
\]
which, by (\ref{dfj}.iii), (\ref{dfj}.ii) and (\ref{mpt}.i), amounts to 
$\,(G\nh_k\w\nh-F\hskip-3.2pt_{kl}\w)PG\nh_k\w\nh=0$. As Lemma~\ref{nonzr} 
now gives $\,PG\nh_k\w\nh\ne0$, one has $\,G\nh_k\w\nh=F\hskip-3.2pt_{kl}\w$. 
Replacing, in (\ref{tfg}), the triple 
$\,(P\hskip-2.7pt_i\w\hh,\hs\,Q\nh_i\w\hh,\nh\,F\hskip-3.2pt_{kl}\w)\,$ with 
$\,(P\nh,Q,G\nh_k\w)$, we easily obtain the equality
$\,[(\nnh-\nnh1\nh)^k\nnh P\nh-Q]\hh[2(\nnh-\nnh1\nh)^i\tau
+3(\nnh-\nnh1\nh)^k\nh\sigma]\hh G\nh_k\w\hs\mu=0$. This contradicts 
Lemma~\ref{nonzr}.
\end{proof}
Lemma~\ref{occur} and the line preceding (\ref{eor}) give 
$\,G\nh_k\w\nh=G\nh_l\w\nh=0$, and so, by (\ref{tdi}.c), 
Lemma~\ref{fjbuz}(b), and the first claim in (\ref{gkz})
\begin{equation}\label{lcs}
\lambda\,\,\mathrm{\ is\ a\ nonzero\ constant,}
\end{equation}
if $\,\,U''$ is connected (or replaced by a connected component). Also,
\begin{equation}\label{pfi}
\begin{array}{rl}
\mathrm{a)}&\mu P\,+\,2\lambda\hh Q\,=\,0\hh,\\
\mathrm{b)}&P\nh F\hskip-3.2pt_{ij}\w\hs
=\,(\nnh-\nnh1\nh)^i\tau\,-\,\mu/2\hh,\\
\mathrm{c)}&Q\nh F\hskip-3.2pt_{ij}\w\hs=\,-\hh(\sigma\hs-\mathrm{s}/6)/2\hh,\\
\mathrm{d)}&2(\nnh-\nnh1\nh)^i(2F\hskip-3.2pt_{ij}\w\mu\,
-\hs5\lambda P)\hh\tau\,=\,3\lambda Q\hh\sigma\hh,\\
\mathrm{e)}&[(\nnh-\nnh1\nh)^i\tau\,-\,\mu/2]\hh\mu\,
=\,(\sigma\hs-\mathrm{s}/6)\lambda\hh,\\
\mathrm{f)}&(\nnh-\nnh1\nh)^i(F\hskip-3.2pt_{ij}\w\mu\,-\hs2\lambda P)\hh\tau\,
=\,(2\lambda\,+\,\hh\sigma\,-\,\mathrm{s}/6)\lambda\hh Q\hh,\\
\mathrm{g)}&(\nnh-\nnh1\nh)^i\nnh F\hskip-3.2pt_{ij}\w\mu\hh\tau\,
=\,(10\lambda\,+\,2\hh\sigma\,-\,5\hs\mathrm{s}/6)\lambda\hh Q\hh,\\
\mathrm{h)}&4(\nnh-\nnh1\nh)^i\nnh\lambda\tau\,
+\,(8\lambda\,+\,\hs\sigma\,-\,2\hs\mathrm{s}/3)\hh\mu\,=\,0\hh.
\end{array}
\end{equation}
In fact, (\ref{pfi}.a) -- (\ref{pfi}.d) are just certain parts of (\ref{mpt}) 
and (\ref{dfj}), with $\,G\nh_k\w\nh=0\,$ in (\ref{dfj}). Equality 
(\ref{pfi}.e) arises in turn due to (\ref{pfi}.a), if one adds (\ref{pfi}.b) 
multiplied by $\,\mu\,$ to (\ref{pfi}.c) times $\,2\lambda$. Let us 
now apply $\,D\nnh_i\w$ to (\ref{pfi}.e), using (\ref{lcs}), (\ref{cst}),  
(\ref{dst}) and (\ref{cnv}). The resulting relation   
$\,2(\nnh-\nnh1\nh)^i(F\hskip-3.2pt_{ij}\w\mu\,-\,\nh2\lambda P)\hh\tau\,
+\,(2\lambda\,+\,3\hh\sigma\nnh/\nh2)(\mu P+2\lambda\hh Q)\,
-\,2(2\lambda+\hh\sigma-\mathrm{s}/6)\lambda\hh Q\,
+\,2\{(\sigma\,-\,\mathrm{s}/6)\lambda-[(\nnh-\nnh1\nh)^i\tau\,
  -\,\mu/2]\mu\}\hh Q=0\,$ becomes (\ref{pfi}.f) when combined with
(\ref{pfi}.a) and 
(\ref{pfi}.e), while (\ref{pfi}.g) is just the side-by-side difference of 
(\ref{pfi}.f) multiplied by $\,5\,$ and (\ref{pfi}.d). Next, subtracting 
(\ref{pfi}.d) from $\,4$ times (\ref{pfi}.f) and cancelling the factor 
$\,\lambda$, as allowed due to (\ref{cnv}), we get 
$\,2(\nnh-\nnh1\nh)^i\nnh P\hn\tau
=(8\lambda+\hs\sigma-2\hs\mathrm{s}/3)\hh Q$. Multiplying this by $\,\mu\,$ 
and then replacing $\,\mu P$ with $\,-\nh2\lambda\hh Q$, cf.\ (\ref{pfi}.a), 
we see that $\,Q\,$ times the left-hand side of (\ref{pfi}.h) equals zero, and 
so (\ref{pfi}.h) follows, since $\,Q\ne0\,$ according to Lemma~\ref{nonzr}. 

Now (\ref{mpt}.vi) times $\,\mu$, with $\,\mu P\,$ replaced by
$\,-\nh2\lambda\hh Q\,$ as above, reads 
$\,4(\nnh-\nnh1\nh)^i\nh F\hskip-3.2pt_{ij}\w\mu^2\nh\tau
=[3\mu\sigma-20(\nnh-\nnh1\nh)^i\nh\lambda\tau]\lambda Q$, while (\ref{pfi}.g)
multiplied by $\,4\mu\,$ yields
$\,4(\nnh-\nnh1\nh)^i\nh F\hskip-3.2pt_{ij}\w\mu^2\nh\tau
=(10\lambda+2\hh\sigma-5\hs\mathrm{s}/6)\hh\mu\lambda\hh Q$. As $\,Q\ne0\,$ in
Lemma~\ref{nonzr}, equating the two right-hand sides, we easily get 
$\,\mu\sigma=20(\nnh-\nnh1\nh)^i\nh\lambda\tau
+(10\lambda-5\hs\mathrm{s}/6)\mu$. However, from 
(\ref{pfi}.h),
$\,\mu\sigma=-\nh4(\nnh-\nnh1\nh)^i\nnh\lambda\tau
-(8\lambda-2\hs\mathrm{s}/3)\hh\mu$. Equating, again, the two new right-hand
sides, we see that, due to (\ref{cst}) and (\ref{lcs}) (that is, constancy of
both $\,\lambda\,$ and $\,\mathrm{s}$), $\,\tau=c\hh\mu\,$ for 
some constant $\,c$, and so $\,\mu\sigma\,$ is a constant multiple of
$\,\mu\,$ as well. By Lemma~\ref{nonzr}, $\,\sigma\,$ must be constant.
We have four further equalities:
\[
\begin{array}{rl}
\mathrm{(i)}&2(\nnh-\nnh1\nh)^i\nnh P\hn\tau\,-\,3Q\sigma\,=\,0\hh,\\
\mathrm{(ii)}&\mu P\,+\,2\lambda\hh Q\,=\,0\hh,\\
\mathrm{(iii)}&4(\nnh-\nnh1\nh)^i\nh\lambda\tau\,+\,3\mu\hh\sigma\,
=\,0\hh,\\
\mathrm{(iv)}&(10\lambda\,+\,\hs\sigma\,-\,2\hs\mathrm{s}/3)\hh\mu^2\hs
+\,4(\sigma\,-\,\mathrm{s}/6)\lambda^2\hs=\,0\hh.
\end{array}
\]
Here (i) follows from constancy of $\,\sigma$, by (\ref{dst}.ii) and
(\ref{cnv}), while (ii) is nothing else than (\ref{mpt}.i), repeated here for
convenience, and (iii) amounts to vanishing of the determinant of
the system (i) -- (ii), due to nontriviality of its solution $\,(P\nh,Q)\,$ 
(see Lemma~\ref{nonzr}). Finally, (iv) is the result of subtracting 
(\ref{pfi}.e) multiplied by $\,4\lambda\,$ from (\ref{pfi}.h) multiplied by 
$\,\mu$.

Using (iii) to replace $\,4(\nnh-\nnh1\nh)^i\nh\lambda\tau\,$ in (\ref{pfi}.h) 
with $\,-3\mu\hh\sigma$, and then cancelling the factor $\,\mu$, cf.\ 
Lemma~\ref{nonzr}, we see that $\,\sigma=4\lambda-\mathrm{s}/3$, which allows 
us to rewrite (iv) as 
$\,(14\lambda-\mathrm{s})\mu^2\nh+2(8\lambda-\mathrm{s})\lambda^2\nh=0$
By (\ref{lcs}) and (\ref{cst}), this last equality implies that $\,\mu\,$ is
constant, which contradicts the assertion about $\,D\nnh_i\w\mu\,$ in
Lemma~\ref{nonzr}.

\section*{Appendix: The other curvature components in Theorem~\ref{cmwtw}}
Whenever $\,\{i,j\}=\{1,2\}\,$ and $\,\{k,l\}=\{3,4\}$, with 
$\,\bz_l\w$ as in\/ {\rm(\ref{ale})},
\[
\begin{array}{l}
R_{ijij}\w\,
=\,\displaystyle{\frac{D\nnh_i\w D\nnh_i\w\lambda_i\w+D\nnh_j\w D\nnh_j\w\lambda_i\w}{2\lambda_i\w}\,-\,3\hs\frac{(D\nnh_i\w\lambda_i\w)^2\nh+(D\nnh_j\w\lambda_i\w)^2}{4\lambda_i^2}}\,\\
\hskip145pt-\,\displaystyle{\frac{(D\nnh_k\w\lambda_i\w)^2\nh+(D\nh_l\w\lambda_i\w)^2\nh+\bz_k^2\nh+\bz_l^2}{\lambda_i^2\nh-\lambda_k^2}}\hh,\\
R_{ikij}\w\,
=\,\displaystyle{\frac{D\nnh_k\w D\nnh_j\w\lambda_i\w}{2\lambda_i\w}\,
+\,(\lambda_k\w-2\lambda_i\w)\frac{(D\nnh_k\w\lambda_i\w)D\nnh_j\w\lambda_i\w}{2\lambda_i^2(\lambda_i\w-\lambda_k\w)}\,
+\,\frac{(\lambda_i\w+3\lambda_k\w)\bz_l\w
D\nnh_i\w\lambda_i\w}{4\lambda_i^2(\lambda_i\w+\lambda_k\w)}\,
-\,\frac{D\nnh_i\w\bz_l\w}{2\lambda_i\w}}\hh,\\
R_{klij}\w\,
=\,D\nh_l\w\displaystyle{\frac{\bz_l\w}{2\lambda_i\w}}\,
-\,D\nnh_k\w\displaystyle{\frac{\bz_k\w}{2\lambda_i\w}}\hh,\qquad
R_{klik}\w\,=\,R_{klkl}\w\,=\,0\hh.
\end{array}
\]



\begin{thebibliography}{99}

\bibitem{aubin-75}{\sc Th. Aubin}, \emph{Le probl\`eme de Yamabe concernant la 
courbure scalaire,} C. \ R. \ Acad. \ Sci. \ Paris S\'er. \ A-B \textbf{280} 
(1975), A721--A724.

\bibitem{aubin-78}{\sc Th. Aubin}, \emph{\'Equations du type Monge-Amp\`ere 
sur les vari\'et\'es k\"ahl\'eriennes compactes,} Bull. \ Sci. \ Math. (2) 
\textbf{102} (1978), 63--95.

\bibitem{besse}{\sc A.\hskip1.6ptL. Besse}, \emph{Ein\-stein Manifolds,}
Ergebnisse, ser. \ 3, Vol. \ 10, Springer-Ver\-lag,
Ber\-lin-Hei\-del\-berg-New York, 1987.

\bibitem{bourguignon}{\sc J.\hskip1.6ptP. Bourguignon}, \emph{Les vari\'et\'es
de dimension\/ $4$ \`a signature non nulle dont la courbure est harmonique
sont d'Einstein,} Invent. \ Math. \textbf{63} (1981), 263--286.

\bibitem{chen-lebrun-weber}{\sc X. Chen, C. LeBrun and B. Weber}, \emph{On 
con\-for\-mal\-ly K\"ah\-ler, Einstein manifolds,} J. \ Amer. \ Math. \ Soc. 
\textbf{21} (2008), 1137--1168.

\bibitem{derdzinski-80}{\sc A. Derdzi\'nski}, \emph{Classification of certain 
compact Riemannian manifolds with harmonic curvature and non-par\-al\-lel 
Ric\-ci tensor,} Math. \ Z. \textbf{172} (1980), 273--280.

\bibitem{derdzinski-88}{\sc A. Derdzi\'nski}, \emph{Riemannian metrics with 
harmonic curvature on $2$-sphere bundles over compact surfaces,} Bull. \ Soc. 
\ Math. \ France \textbf{116} (1988), 133--156.

\bibitem{derdzinski-piccione}{\sc A. Derdzinski and P\nnh. Piccione}, 
\emph{Max\-i\-mal\-ly-\nh warp\-ed metrics with harmonic curvature.} In:  
Proceedings of the AMS Special Session on Geometry of Submanifolds, ed. by J. 
Van der Veken, A. Carriazo, I. Dimitric, Y.-M. Oh, B. Suceava, and L. 
Vrancken, Contemp.\ Math.\ \textbf{756}, AMS, Providence, RI, 2020, pp. 83--96 
(preprint available from arxiv.org/pdf/1812.06027.pdf).

\bibitem{derdzinski-shen}{\sc A. Derdzi\'nski and C.-\hs L. Shen}, 
\emph{Co\-daz\-zi tensor fields, curvature and Pon\-trya\-gin forms,} Proc.\ 
London Math.\ Soc. \textbf{47} (1983), 15--26.

\bibitem{deturck-goldschmidt}{\sc D. DeTurck \and H. Goldschmidt}, 
\emph{Regularity theorems in Riemannian geometry. II: Harmonic curvature and 
the Weyl tensor,} Forum Math. \textbf{1} (1989), 377--394.

\bibitem{kuiper}{\sc N.\hskip1.6ptH. Kuiper}, \emph{On con\-for\-mally-flat
spaces in the large,} Ann.\ of Math. (2) \textbf{50} (1949), 916--924.

\bibitem{maillot}{\sc H. Maillot}, \emph{Sur les vari\'et\'es riemanniennes 
\`a op\'erateur de courbure pur,} C. \ R. \ Acad. \ Sci. \ Paris S\'er. \ A 
\textbf{278} (1974), 1127--1130.

\bibitem{matsushima}{\sc Y. Matsushima}, \emph{Sur la structure du groupe 
d'hom\'eomorphismes analytiques d'une certaine vari\'et\'e k\"ahl\'erienne,} 
Nagoya Math. \ J. \textbf{11} (1957), 145--150.

\bibitem{page}{\sc D. Page}, \emph{A compact rotating gravitational 
in\-stant\-on,} Phys. \ Lett. \textbf{79B} (1978), 235--238.

\bibitem{schoen}{\sc R. Schoen}, \emph{Con\-for\-mal deformation of a
Riemannian metric to constant scalar curvature,} J. \ Differential Geom.
\textbf{20} (1984), 479--495. 

\bibitem{singer-thorpe}{\sc I.\hskip1.6ptM. Singer and J.\hskip1.6ptA.
Thorpe}, \emph{The curvature of $4$-dimensional Einstein spaces,} Global
Analysis (Papers in Honor of K. Kodaira), 355-–365, Univ. \ Tokyo Press,
Tokyo, 1969.

\bibitem{tian}{\sc G. Tian}, \emph{On Ca\-la\-bi's conjecture for complex 
surfaces with positive first Chern class,} Invent. \ Math. \textbf{101} 
(1990), 101--172.

\bibitem{tod}{\sc K.\hskip1.6ptP. Tod}, \emph{On choosing coordinates to
di\-ag\-o\-nal\-ize the metric,} Class.\ Quantum Gravity \textbf{9} (1992),
1693--1705.
  
\bibitem{yau-73}{\sc S.-T. Yau}, \emph{Remarks on con\-for\-mal
transformations.} J.\ Differential Geom. \textbf{8} (1973), 
369--381.

\bibitem{yau-78}{\sc S.-T. Yau}, \emph{On the Ricci curvature of a compact 
K\"ahler manifold and the complex Monge-Amp\`ere equation, I,} Comm. \ Pure 
Appl. \ Math. \textbf{31} (1978), 339--411.

\end{thebibliography}
\end{document}